\documentclass[10pt]{svjour3}                     
\smartqed  
\usepackage{amsmath}
\usepackage{amssymb}
\usepackage{graphicx}
\usepackage{url}
\usepackage{hyperref}
\usepackage{algorithm}
\usepackage{algpseudocode}
\usepackage{pdflscape}
\usepackage{xcolor}
\usepackage{comment}
\usepackage{wrapfig}
\usepackage{verbatim}
\usepackage{caption}
\usepackage{paralist}
\usepackage{subcaption}
\usepackage{booktabs} 
\usepackage{float} 
\usepackage{multicol} 
\usepackage{caption} 
\usepackage{colortbl}
\usepackage[left=3.25cm,top=3.1cm,right=3.0cm,bottom=3.25cm,bindingoffset=0.5cm]{geometry}

\newtheorem{assumption}{Assumption A.\!\!}

\newcommand{\Id}{\mathbb{I}}

\newcommand{\Tc}{\mathcal{T}}
\newcommand{\R}{\mathbb{R}}

\newcommand{\Rext}{\R\cup\{+\infty\}}

\newcommand{\abs}[1]{\left\vert#1\right\vert}
\newcommand{\absc}[1]{\big\vert#1\big\vert}
\newcommand{\set}[1]{\left\{#1\right\}}
\newcommand{\sets}[1]{\{#1\}}
\newcommand{\norm}[1]{\left\Vert#1\right\Vert}
\newcommand{\tnorm}[1]{\vert\!\Vert#1\vert\!\Vert}
\newcommand{\norms}[1]{\Vert#1\Vert}

\newcommand{\Eproof}{\hfill $\square$}
\newcommand{\prox}{\mathrm{prox}}

\newcommand{\relint}[1]{\mathrm{ri}\left(#1\right)}
\newcommand{\argmin}{\mathrm{arg}\!\displaystyle\min}

\newcommand{\dom}[1]{\mathrm{dom}(#1)}
\newcommand{\zero}[1]{\boldsymbol{0}}
\newcommand{\cl}[1]{\mathrm{cl}\left(#1\right)}
\newcommand{\xb}{x}

\newcommand{\ub}{u}
\newcommand{\vb}{v}

\newcommand{\Qc}{\mathcal{Q}}

\newcommand{\Xc}{\mathcal{X}}

\newcommand{\Sc}{\mathcal{S}}
\newcommand{\Dc}{\mathcal{D}}

\newcommand{\Ec}{\mathcal{E}}

\newcommand{\Uc}{\mathcal{U}}

\newcommand{\Pc}{\mathcal{P}}

\newcommand{\iprod}[1]{\left\langle #1\right\rangle}
\newcommand{\iprods}[1]{\langle #1\rangle}
\newcommand{\intx}[1]{\mathrm{int}\left(#1\right)}

\newcommand{\BigO}[1]{\mathcal{O}\left(#1\right)}

\newcommand{\rvnn}[1]{{#1}}

\newcommand{\beforesec}{\vspace{-3.5ex}}
\newcommand{\aftersec}{\vspace{-2.25ex}}
\newcommand{\beforesubsec}{\vspace{-4.5ex}}
\newcommand{\aftersubsec}{\vspace{-2.5ex}}
\newcommand{\beforesubsubsec}{\vspace{-3.0ex}}
\newcommand{\aftersubsubsec}{\vspace{-2.5ex}}
\newcommand{\beforeparagraph}{\vspace{-2.5ex}}

\journalname{Comput Optim Appl}

\begin{document}

\title{Composite Convex Optimization with Global and Local Inexact Oracles
\thanks{
The work of Q. Tran-Dinh  was supported by the National Science Foundation (NSF), grant: DMS-1619884, and the Office of Naval Research (ONR), grant: N00014-20-1-2088 (2020-2023).
The work of I. Necoara was supported by the Executive Agency for Higher Education, Research and Innovation Funding (UEFISCDI), Romania,  PNIII-P4-PCE-2016-0731, project ScaleFreeNet, no. 39/2017.
}
}

\titlerunning{Composite Convex Minimization with New Inexact Oracles}        
\authorrunning{T. Sun, I. Necoara, \textit{and} Q. Tran-Dinh}

\author{Tianxiao Sun$^{\dagger}$ \and Ion Necoara$^{\ddagger}$ \and Quoc Tran-Dinh$^{\dagger\ast}$
}


\institute{$^{\dagger}$Tianxiao Sun \and $^{\dagger}$Quoc Tran-Dinh \at
		Department of Statistics and Operations Research,  The University of North Carolina at Chapel Hill\\
		333 Hanes Hall, CB\# 3260, UNC Chapel Hill, NC 27599-3260, USA\\ 
		Email: \url{quoctd@email.unc.edu}
		\and
		$^{\ddagger}$Ion Necoara \at Department of Automatic Control and Systems Engineering, University Politehnica Bucharest\\ 
		Spl. Independentei 313, Bucharest, 060042, Romania\\ 
		Email: \url{ ion.necoara@acse.pub.ro}
		\and
		$^{\ast}$\textit{Corresponding author} (\email{\tt quoctd@email.unc.edu}).
}

\date{First version: 08/06/2018 (on Arxiv). Third version: 02/22/2020.}

\maketitle

\begin{abstract}
We introduce new  global and local  inexact oracle concepts for a wide class of convex functions in composite convex minimization.  
Such  inexact oracles naturally arise in many situations, including primal-dual frameworks, barrier smoothing, and inexact evaluations of gradients and Hessians. 
We also provide examples showing that the class of convex functions equipped with the newly inexact oracles is larger than standard self-concordant and  Lipschitz gradient  function classes.  
Further, we investigate several properties of convex and/or self-concordant functions under our inexact oracles which are useful for algorithmic development.  
Next, we apply our theory to develop inexact proximal Newton-type schemes for minimizing  general composite convex optimization problems  equipped with such inexact oracles. 
Our theoretical results consist of new optimization algorithms accompanied with global convergence guarantees to solve a wide class of composite convex optimization problems. 
When the first objective term is additionally  self-concordant,  we establish different local convergence results for our method.  
In particular, we prove that depending on the choice of accuracy levels of the inexact second-order oracles, we obtain different local convergence rates ranging from linear and superlinear to quadratic.  In special cases, where convergence bounds are known, our theory recovers the best known rates. We also apply our settings to derive a new primal-dual method for composite convex minimization problems involving linear operators.   Finally, we present some representative numerical examples to illustrate the benefit of the new algorithms. 


\end{abstract}

\beforesec
\section{Introduction}\label{sec:intro}
\aftersec
We consider the following composite convex optimization problem:
\begin{equation}\label{eq:c_sc_min}
F^{\star} = \min_{x\in\R^p} \Big\{ F(x) := f(x) + R(x) \Big\},
\end{equation}
where $f$ and $R$ are proper, closed, and convex from $\R^p \to \Rext$.
It is well-known that problem \eqref{eq:c_sc_min} covers various applications in machine learning, statistics, signal and image processing, and control. 
Very often in applications, $f$ can be considered as a loss or a data fidelity term, while $R$ is referred to as a regularizer that can promote desired structures of the solution. 
In particular, if $R$ is the indicator of a convex set $\Xc$, then  \eqref{eq:c_sc_min} also covers constrained settings.

Optimization methods for solving \eqref{eq:c_sc_min} often rely on a so-called \textit{oracle} \cite{Nemirovskii1983} to query information about proximal operators, function values, [sub]gradients, and Hessians of the objective functions in \eqref{eq:c_sc_min} for computing  an approximate solution.
However, such an oracle may not be available in practice, but only  approximate information can be accessed.
This paper is concerned with inexact oracles to design numerical methods for solving \eqref{eq:c_sc_min}.
We first focus on a relatively general convex setting of  \eqref{eq:c_sc_min} by equipping $f$ with a global inexact oracle. 
Then, we limit our consideration to a class of self-concordant functions $f$ and introduce a local  inexact second-order oracle.

Self-concordance is a mathematical concept introduced by Y. Nesterov and A. Nemirovskii in the early 1990s to develop polynomial-time interior-point methods for convex optimization. 
This function class, which will be formally defined in Definition~\ref{de:self_con_func} below, covers many key applications such as  conic programming \cite{BenTal2001},  graphical model selection \cite{Friedman2008}, Poisson imaging \cite{Harmany2012},  logistic regression \cite{zhang2015disco}, statistical learning \cite{Ostrovskii2018}, and control \cite{Necoara2009}. 
Although several inexact first-order oracles have been proposed for the class of smooth convex functions in many settings, see, e.g.,  \cite{d2008smooth,Devolder2010,necoara2015complexity}, inexact second-order oracles for self-concordant functions have not yet been studied in the literature to the best of our knowledge. 
Note that the inexact second-order oracles we discuss here are very different from the inexact methods for self-concordant minimization, where the subproblems in an optimization routine are approximately solved \cite{Lu2016a,TranDinh2012c,TranDinh2016c,zhang2015disco}.

\beforeparagraph
\paragraph{\bf Why inexact oracle?}
Inexact oracles arise in many practical problems under different situations.
Among these, in the following two cases, inexact oracles appear to be natural:
\begin{compactitem}
\item[(i)]\textit{Accumulation of errors:}
We often encounter the accumulation of errors during evaluating, storing, and transferring data. 
This happens  frequently in sequential methods and distributed computation, where accumulation of errors clearly affects input data of the underlying optimization problem or communication. Another situation is where we truncate the output of an evaluation to fit a storage device or approximate our computation to reduce execution time as well as memory storage.

\item[(ii)]\textit{Inexact evaluation:}
Inexact evaluations of function values and its proximal operator and derivatives arise in many  optimization algorithms. For instance, in a primal-dual method where the primal subproblem is approximately solved and thus we cannot evaluate the true function values and derivatives of the corresponding dual function. It also arises when we evaluate the functions and derivatives through its Fenchel conjugate or through smoothing techniques. These cases also lead to inexact oracles of the underlying functions.

\end{compactitem}
As the examples of Section \ref{sec:examples_of_inexact_oracle} will show, the class of functions equipped with inexact second-order  oracles  is rather large, covering convex  (non)smooth   functions and self-concordant functions.

\beforeparagraph
\paragraph{\bf Related work:}
Inexact oracles have been widely studied for smooth convex optimization in first-order methods, see, e.g.,  \cite{d2008smooth,Devolder2010,dvurechensky2016stochastic,necoara2015complexity}. Among these frameworks, \cite{Devolder2010} provides a general inexact first-order oracle  capturing  a wide class of objective functions, including nonsmooth functions, and covering many other existing inexact first-order oracles as special cases. However, \cite{Devolder2010} only studied a global first-order inexact oracle to  analyze the behavior of  first-order methods of smooth convex optimization. Such an oracle cannot be used  to study the local behavior of second-order methods, in particular, for self-concordant functions.  Moreover, in second-order methods, quasi-Newton algorithms are usually approximating the Hessian mapping  via secant equations \cite{gao2016quasi,Nocedal2006}. We show in this paper that this setup can also be cast into our  Newton-type methods with inexact oracles. Furthermore, inexact methods for self-concordant minimization, where the subproblems in an optimization routine are approximately solved \cite{li2017inexact,Lu2016a,TranDinh2012c,TranDinh2016c,zhang2015disco}, are also covered by our inexact oracle algorithmic framework. Alternative to deterministic inexact oracles, stochastic gradient-type schemes can be also viewed as optimization methods with inexact oracles \cite{Shapiro2009}. Function values and gradients are approximated by a stochastic sampling scheme to obtain also inexact oracles.  Finally, derivative-free optimization can be considered as optimization methods with inexact oracles \cite{bogolubsky2016learning,Conn2008}.

\beforeparagraph
\paragraph{\bf Our approach and contribution:}
Our approach, inspired by \cite{Devolder2010},  essentially introduces new global and local inexact second-order oracles and develops some key bounds to design inexact optimization algorithms for solving \eqref{eq:c_sc_min}. 
While \cite{Devolder2010} aimed at developing first-order methods, we focus on second-order methods, which not only require inexact function values and gradients, but also inexact Hessian mappings. 
Moreover, since we design Newton-type methods, we first introduce a global inexact second-order oracle to investigate the global behavior of the proposed algorithms. 
Then, we also define a local inexact second-order oracle to study the local convergence of our second-order methods.
In contrast to \cite{Devolder2010}, where the accuracy level is fixed at the target accuracy, our methods can either fix the accuracies  of the oracles or adapt them in such a way that they can be rough in the early iterations and become finer in the latter iterations.

\beforeparagraph
\paragraph{Our contribution}
To this end, our main contribution can be summarized as follows:
\begin{compactitem}
\item[(a)]
We introduce new global and local inexact second-order oracles for a large class of convex functions. 
The global inexact oracle covers a wide range  of convex functions including smooth convex functions with Lipschitz gradient continuity,  nonsmooth Lipschitz continuous  convex functions with bounded domain, and self-concordant convex functions.  
For the local inexact oracle, we limit our consideration to the class of self-concordant functions.  
Relying on these global and local inexact oracles, we develop several key properties that are useful for algorithmic development.

\item[(b)]
We develop a proximal Newton algorithm based on inexact oracles and approximate computations of the proximal Newton directions to solve the composite convex minimization problem \eqref{eq:c_sc_min}. 
Our global inexact oracle allows us to prove  global convergence guarantees for the proposed proximal Newton method. 
When limited to the self-concordant class for $f$, by using the new local inexact oracle,
we show how to adapt the inner accuracy parameters  of the oracles so that our algorithm still enjoys a global convergence guarantee, while having either linear, superlinear, or quadratic local convergence rate. 
Our new convergence analysis requires nontrivial technical steps and novel tools such as new local norms, inequalities, and bounds for inexact quantities compared to the analysis in the exact oracle case \cite{Nesterov1994}.

\item[(c)] 
Finally, we customize our method to handle a class of convex programs in a primal-dual setting, where our method is applied to solve the dual problem.
This particular application provides a new primal-dual method for handling some classes of convex optimization problems including constrained formulations and linear compositions.

\end{compactitem}
Let us emphasize the following points of our contributions.
Firstly, our global inexact second-order oracle  is defined via a weighted local norm and  via a non-quadratic term and thus very different from the inexact first-order oracle in  \cite{Devolder2010}. 
Secondly, our global convergence result is independent of the self-concordance of $f$. 
This global convergence result holds for a large class of functions, including Lipschitz gradient convex functions analyzed in \cite{Devolder2010}. 
Thirdly, our inexact algorithm covers the inexact Newton  methods from \cite{li2017inexact,Lu2016a,TranDinh2012c,zhang2015disco} and quasi-Newton methods developed in \cite{gao2016quasi,TranDinh2016c} as special cases   (see Subsection \ref{subsec:special_cases}). In these cases, where convergence bounds are known, our theory recovers the best known rates.   Finally, we strongly believe that our theory can be used to further develop other methods such as sub-sampled Newton-type methods rather than just the inexact proximal Newton method as in this paper.

\beforeparagraph
\paragraph{\bf Paper organization:}
The rest of this paper is organized as follows.
Section \ref{eq:inexact_oracle} recalls the concepts of self-concordant functions and self-concordant barriers from \cite{Nesterov1994,Nesterov2004}.
We also introduce the concept of global and local inexact oracles in this section.
Section \ref{sec:examples_of_inexact_oracle} presents several examples of inexact oracles.
Section \ref{sec:proximal_nt_methods} develops proximal Newton-type methods using inexact oracles. We show that the obtained algorithms achieve both global convergence and local convergence from linear to quadratic rates. We also show that our methods cover some existing inexact methods in the literature as special cases. Section \ref{sec:primal_dual_method} shows an application to primal-dual methods, and the last section provides some representative examples to illustrate the theory.


\beforesec
\section{Global and Local  Inexact Oracles}
\label{eq:inexact_oracle}
\aftersec
First, we introduce a global and a local inexact oracle concept  for a class of convex functions. 
Then, utilizing these new notions, we develop several properties for this function class that can be useful for algorithmic development. 

\beforesubsec
\subsection{\bf Basic notations and terminologies}\label{subsec:notation}
\aftersubsec
Let $\iprods{\ub, \vb}$ or $\ub^{\top}\vb$ denote standard inner product, and $\Vert \ub\Vert_2$ denote the Euclidean norm for any $\ub,\vb\in\R^p$.
For a nonempty, closed, and convex set $\Xc$ in $\R^p$, $\mathrm{ri}(\Xc)$ denotes the relative interior of $\Xc$ and $\mathrm{int}(\Xc)$ stands for the interior of $\Xc$.
For a proper, closed, and convex function $f : \R^p\to\Rext$,  $\dom{f} := \set{ x \in \R^p \mid f(x) < + \infty}$ denotes its domain, $\cl{\dom{f}}$ denotes the closure of $\dom{f}$,  $\partial{f}( x) := \big\{ w \in\R^p \mid f(\ub) \geq f( x) + \iprods{w, u -  x}, ~\forall \ub \in\dom{f} \big\}$ denotes its subdifferential at $ x$, and $f^{*}(y) := \sup_{x}\set{\iprods{x, y} - f(x)}$ denotes its Fenchel conjugate \cite{Rockafellar1970}.

We also use $\mathcal{C}^3(\mathcal{X})$ to denote the class of three-time continuously differentiable functions from $\mathcal{X}\subseteq\mathbb{R}^p$ to $\mathbb{R}$.
$\Sc^p_{+}$ stands for the symmetric positive semidefinite cone of order $p$, and $\Sc^p_{++}$ is its interior, i.e., $\Sc_{++}^p = \intx{\Sc_{+}^p}$.
For a three-time continuously differentiable and convex function $f : \R^p\to\Rext$ such that $\nabla^2{f}( x) \succ 0 $ at some $ x\in\dom{f}$ (i.e., $\nabla^2{f}( x)$ is symmetric positive definite), we define a local norm, and its corresponding dual norm, respectively as
\begin{equation}\label{eq:local_norm}
\norm{\ub}_{ x} := \iprods{\nabla^2{f}( x)\ub, \ub}^{1/2} ~~~\text{and}~~~ \norm{\vb}_{ x}^{\ast} := \iprods{\nabla^2{f}( x)^{-1}\vb, \vb}^{1/2},
\end{equation}
for given $\ub, \vb\in\R^p$.
It is obvious that $\iprods{\ub, \vb} \leq \norm{\ub}_{ x}\norm{\vb}^{\ast}_{ x}$ due to the Cauchy-Schwartz inequality.
Let $H(x) \in \Sc^p_{++}$ be an approximation of $\nabla^2{f}(x)$ at $x\in\dom{f}$.
We define the following weighted norm and its dual norm for any $u$ and $v$ as:
\begin{equation}\label{eq:weighted_norm}
\tnorm{u}_x := \norm{u}_{H(x)} = (u^{\top}H(x)u)^{1/2}~~~~\text{and}~~~~\tnorm{v}_x^{\ast} := \norms{v}_{H(x)^{-1}} = (v^{\top}H(x)^{-1}v)^{1/2}.
\end{equation}
We still have the relation $\iprods{u,v} \leq \tnorm{u}_x\tnorm{v}_x^{\ast}$.

We will repeatedly use the following two strictly univariate convex functions: 
\begin{compactitem}
\item[$\bullet$] $\omega : \R_{+}\to \R_{+}$ with $\omega(\tau) := \tau - \ln(1+\tau)$;
\item[$\bullet$] $\omega_{\ast} : [0, 1) \to \R_{+}$ with $\omega_{\ast}(\tau) := -\tau - \ln(1-\tau)$.
\end{compactitem}
As shown, e.g.,  in \cite{Nesterov1994,Nesterov2004} that $\omega(\tau) \geq \frac{\tau^2}{2(1+\tau)}$ for all $\tau \geq 0$ and $\omega_{\ast}(\tau) \leq \frac{\tau^2}{2(1-\tau)}$ for all $\tau \in [0, 1)$.
It is also easy to prove that $\omega_{\ast}(\tau) \geq \frac{\tau^2}{2}$ for all $\tau \in [0, 1)$.

\beforesubsec
\subsection{\bf Self-concordant functions}
\label{subsec:self_concordance}
\aftersubsec
We recall the self-concordant concept introduced in \cite{Nesterov1994,Nesterov2004}.
This concept has been intensively used in interior-point methods and  has recently been used in other applications of machine learning, image processing, and control  \cite{gao2016quasi,Lu2016a,Necoara2009,Ostrovskii2018,Tran-Dinh2013a,TranDinh2016c,zhang2015disco}.

\begin{definition}\label{de:self_con_func}
\textit{
A univariate convex function $\varphi \in \mathcal{C}^3(\dom{\varphi})$ is called $M_{\varphi}$-\emph{self-concordant} if $\abs{\varphi'''(\tau)}$ $  \leq M_{\varphi}\varphi''(\tau)^{3/2}$  for all $\tau\in\dom{\varphi}$, where $\dom{\varphi}$ is an open set in $\R$ and $M_{\varphi} \geq 0$.
A function $f : \dom{f}\subseteq \R^{p} \to \R$ is $M_f$-self-concordant if for any $ x\in\dom{f}$ and $\vb\in\R^p$, the univariate function $\varphi$ defined by $\tau \mapsto \varphi(\tau) := f( x + \tau\vb)$ is $M_f$-self-concordant.
If $M_f = 2$, then we say that $f$ is standard self-concordant. }
\end{definition}

Examples of self-concordant functions include, but not limited to $f(x) = -\sum_{j=1}^p\ln(x_j)$ on $\R^p_{++}$ (see \cite[page 213]{Nesterov2004}), $f(X) = -\ln\det(X)$ on $\Sc^p_{++}$ (see \cite[page 216]{Nesterov2004}), $f(x) = \sum_{i=1}^n\ln(1 + e^{-a_i^{\top}x}) + \tfrac{\mu}{2}\norms{x}^2$ with $\mu > 0$ on $\R^p$ (see \cite[Proposition 5]{SunTran2017gsc}), and $f(x) = -\sum_{j=1}^p(x_j\ln(x_j) - \ln(x_j))$ on $\R^p_{++}$ (see \cite[page 8]{SunTran2017gsc}).
For other examples, we refer to \cite{Nesterov2004,zhang2015disco}.

\begin{remark}\label{re:standard_self_con}
If $f$ is an $M_f$-self-concordant function, then we can rescale it as $\hat{f} := \frac{M_f^2}{4}f$ such that $\hat{f}$ is standard self-concordant.
Therefore, without loss of generality, from now on, if we say ``$f$ is a self-concordant function'', then it means that $f$ is a \textit{\textbf{standard self-concordant function}}.
\end{remark}

It is known that the following two inequalities are necessary and sufficient characteristics of self-concordant functions (see Theorem 4.1.9 in \cite{Nesterov2004}):
\begin{equation}\label{eq:ns_sc}
\omega(\norm{y-x}_x)  \leq f(y) - f(x) - \iprods{\nabla f(x), y-x} \leq \omega_{\ast}(\norm{y-x}_x)  \quad \forall x,y \in \dom{f},
\end{equation}
where the right-hand side holds for $\norm{y-x}_x <1$, and  $\omega(t) = t - \ln(1+t)$ and its conjugate $\omega_{\ast}(\tau) = -\tau - \ln(1 - \tau)$ are defined above.

This equivalent characterization of self-concordant functions motivates us to introduce in the next subsection  the notion of  inexact global and local oracles,
and analyze the  behavior  of Newton-type  methods  of  standard self-concordant  optimization using such oracles.
Intensive theory of self-concordance can be found, e.g.,  in \cite{Nesterov1994,Nesterov2004}.

\beforesubsec
\subsection{\bf Inexact oracles for convex functions}
\label{subsec:inexact_oracle}
\aftersubsec
Let $f$ be a  convex function with $\dom{f} \subseteq \R^p$.
Given three mappings $\tilde{f}(\cdot) \in\R$, $g(\cdot)\in\R^p$, and $H(\cdot) \in \Sc^p_{++}$ defined on $\dom{f}$, we introduce the following two types of inexact oracle of $f$.

\begin{definition}[Global inexact oracle]\label{de:global_inexact_oracle}
\textit{
For a general convex and possibly nonsmooth function $f$, a triple $(\tilde{f}, g, H)$ is called a $(\delta_0,\delta_1)$-global inexact oracle of $f$ with accuracies $\delta_0 \in [0, 1]$ and $\delta_1 \geq 0$, if for any $x\in\dom{f}$ and $y\in\R^p$ such that $\tnorm{y-x}_x < \frac{1}{1+\delta_0}$, we have $y\in\dom{f}$,   $H(x) \succ 0$, and the following inequality holds:
\begin{equation}\label{eq:global_inexact_oracle}
\omega\left( (1-\delta_0)\tnorm{y - x}_x\right) \leq f(y) - \tilde{f}(x) -  \iprods{g(x), y-x} \leq  \omega_{\ast}\left((1+\delta_0)\tnorm{y - x}_x\right) + \delta_1.
\end{equation}
}
\end{definition}

Note that the condition ``$\tnorm{y-x}_x < \frac{1}{1+\delta_0}$ implies $y\in\dom{f}$'' is required for the right-hand side of \eqref{eq:global_inexact_oracle} to be well-defined.  
It automatically holds if $f$ is self-concordant and $H(x) = \nabla^2{f}(x)$ with $\delta_0=0$.  
If $\dom{f} = \R^p$, then this condition  also holds.  
However, when $\dom{f} \subset \R^p$ we need to impose the condition ``$\tnorm{y-x}_x < \frac{1}{1+\delta_0}$ implies $y\in\dom{f}$'' in our definition for global inexact oracle.

Note also that the inexact oracle is defined at any $x \in \dom{f}$. Hence, it is referred to as a global inexact oracle.
Here, $H(\cdot)\succ 0$ is only required for $\xb$ in some level set of $\xb^0$, which will be further discussed later. 
Moreover, it does not require the differentiability of $f$.

Nevertheless, for this inexact oracle, if $f$ is twice differentiable, then $\tilde{f}$ gives an approximation to $f$, $g$ is an approximation to $\nabla{f}$, and $H$ is an approximation to $\nabla^2f$.  
Clearly, from \cite[Theorem 4.1.9]{Nesterov2004}, if $f$ is a self-concordant function, then it admits a $(0, 0)$-global inexact oracle, namely $\tilde{f}(x)=f(x)$, $g(x) = \nabla f(x)$, and $H(x) = \nabla^2 f(x)$ by setting $\delta_0 = 0$ and $\delta_1 = 0$.  

A global inexact oracle will be used to analyze global convergence of our algorithms developed in the next sections.
In order to investigate local convergence of Newton-type methods we also require a local inexact second-order oracle in addition to the above global inexact one.

\begin{definition}[Local inexact second-order  oracle]\label{de:local_inexact_oracle}
\textit{
For a twice differentiable convex function $f$ and a subset $\Xc \subset \dom{f}$, a triple $(\tilde{f}, g, H)$ is called a $(\delta_0,\delta_1,\delta_2,\delta_3)$-local inexact second-order oracle of $f$ on $\Xc$ if   for any $x\in\Xc$ and $y\in\R^p$ such that $\tnorm{y-x}_x < \frac{1}{1+\delta_0}$, we have $y\in\dom{f}$,   $H(x) \succ 0$, and \eqref{eq:global_inexact_oracle} holds.
Additionally, the following approximations for the gradient and Hessian mappings hold:
\begin{equation}\label{eq:local_inexact_oracle}
\left\{\begin{array}{ll}
&\tnorm{g(x)  - \nabla{f}(x)}_{x}^{\ast} \leq \delta_2, \vspace{1ex}\\
&(1-\delta_3)^2 \nabla^2{f}(x) \preceq H(x) \preceq (1+\delta_3)^2 \nabla^2{f}(x),
\end{array}\right.
\end{equation}
for all $x \in \Xc$, where $\boldsymbol{\delta} :=(\delta_0, \delta_1,\delta_2,\delta_3) \geq 0$   and $0 \leq \delta_0,\delta_3 < 1$.
}
\end{definition}

Note that we only require the two conditions of \eqref{eq:local_inexact_oracle} in a given subset $\Xc$ of $\dom{f}$, therefore this inexact oracle is local. 
Again, we observe that any self-concordant function admits a $(0, 0, 0, 0)$-local oracle, which is the exact second-order oracle.

\beforesubsec
\subsection{\bf Properties of global inexact oracle}
\label{subsec:properties_of_global_inexact_oracle}
\aftersubsec
Convex functions, including self-concordant functions, have many important properties on the function values, gradients, and Hessian mapping \cite{Nesterov1994,Nesterov2004}. 
These properties are necessary  to develop Newton-type and interior-point methods. 

The following lemma provides some key properties of our global inexact oracle of $f$ whose proof is given in Appendix \ref{apdx:le:global_properties}. 
Note that these properties hold for general convex functions endowed with such  a global inexact  oracle.

\begin{lemma}\label{le:global_properties}
Let $(\tilde{f}, g, H)$ be a $(\delta_0,\delta_1)$-global inexact oracle of a convex function $f$ as defined in Definition~\ref{de:global_inexact_oracle}.
Then:
\begin{compactitem}
\item[$\mathrm{(a)}$] For any $x\in\dom{f}$, we have
\begin{equation}\label{eq:global_inexact_oracle_pro1}
\tilde{f}(x)\leq f(x) \leq \tilde{f}(x) + \delta_1.
\end{equation}

\item[$\mathrm{(b)}$]
Assume that $f^{\star} := \min_{x}\set{ f(x) \mid x \in\dom{f}} > -\infty$.
Then, the inexact gradient $g(\bar{x})$ certifies a $\delta_1$-approximate minimizer $\bar{x} \in\dom{f}$ of $f$;
that is, if $\iprod{g(\bar{x}), y -\bar{x}}\geq 0$ for all $y\in\dom{f}$, then
\begin{equation*}
f^{\star} \leq f(\bar{x}) \leq f^{\star} + \delta_1.
\end{equation*}

\item[$\mathrm{(c)}$]
For any $x\in\intx{\dom{f}}$, the interior of $\dom{f}$, the difference between  $g(x)$ and the true $($sub$)$gradient of  a convex function  $f$ is bounded as
\begin{equation}\label{eq:global_inexact_oracle_pro4}
\tnorm{\nabla f(x)-g(x)}_{x}^{\ast} \leq \delta_2(\delta_0,\delta_1),
\end{equation}
where $\delta_2(\delta_0,\delta_1)$ is the unique nonnegative solution $($always exists$)$ of the equation $\omega \left(\frac{\delta_2}{1+\delta_0} \right) = \delta_1$ in $\delta_2$. 
Moreover, $\delta_2(\delta_0,\delta_1) \to 0$ as $\delta_0,\delta_1 \to 0$.

\item[$\mathrm{(d)}$]
For any $x \in\intx{\dom{f}}$ and $y \in\dom{f}$, we have
\begin{equation}\label{eq:global_inexact_oracle_pro6a}
\omega\left(\tfrac{\tnorm{g(x) - \nabla{f}(y)}_{x}^{\ast}}{1+\delta_0}\right) \leq  \tnorm{g(x) - \nabla{f}(y)}_{x}^{\ast}\tnorm{y - x}_{x} + \delta_1,
\end{equation}
\end{compactitem}
\end{lemma}
Note that $\delta_2$ depends on $\delta_0$ and $\delta_1$ only in the statement $\mathrm{(c)}$ of Lemma~\ref{le:global_properties}.

\beforesubsec
\subsection{\bf Properties of local inexact oracle}
\label{subsec:properties_of_local_inexact_oracle}
\aftersubsec

First,  it follows from  Lemma~\ref{le:global_properties} that the condition \eqref{eq:global_inexact_oracle} is also sufficient to deduce that $\tnorm{g(x)  - \nabla{f}(x)}_{x}^{\ast} \leq \delta_2$. However, in this case $\delta_2$ is a function of $\delta_0$ and $\delta_1$, and $\delta_2 = \delta_2(\delta_0, \delta_1) \to 0$ as $\delta_0,\delta_1\to 0$. Therefore, the first condition \eqref{eq:local_inexact_oracle} can be guaranteed from the global inexact oracle in Definition \ref{de:global_inexact_oracle}. In order to make our method more flexible, we use the first condition of \eqref{eq:local_inexact_oracle} to define local inexact oracle instead of deriving it from a global inexact oracle as in Lemma~\ref{le:global_properties}. 
This allows all the tolerance parameters to be independent.
We now prove some properties of local inexact oracle in the following lemma, whose proof is given in Appendix \ref{apdx:le:norm_relation}.

\begin{lemma}\label{le:norm_relation}
Let $(\tilde{f}, g, H)$ be a local inexact oracle of a twice differentiable convex function  $f$ on $\Xc \subset \dom{f}$ defined in Definition \ref{de:local_inexact_oracle}. Then, for any $u, v\in\R^p$ and $x\in \Xc$, we have
\begin{equation}\label{eq:norm_relation}
(1-\delta_3)\norms{u}_x \leq \tnorm{u}_x \leq (1+\delta_3)\norms{u}_x ~~~\text{and}~~~\frac{1}{1+\delta_3}\norms{v}_x^{\ast} \leq \tnorm{v}_x^{\ast} \leq \frac{1}{1-\delta_3}\norms{v}_x^{\ast}.
\end{equation}
If, in addition, $f$ is self-concordant, then for any $x, y \in \Xc$ such that $\tnorm{y-x}_x < 1- \delta_3$, we  have:
\begin{equation}\label{eq:key_bound3}
\left\{\begin{array}{ll}
\Big[\frac{1 - \delta_3 - \tnorm{y-x}_x}{1+\delta_3} \Big]^2H(x) & \preceq H(y) \preceq \left[\frac{1+\delta_3}{1 - \delta_3 - \tnorm{y-x}_x}\right]^2 H(x), \vspace{1ex}\\
\tnorm{(\nabla^2 f(x)-H(x))v}_y^{\ast} & \leq \frac{2\delta_3 + \delta_3^2}{(1-\delta_3)(1 -\delta_3 - \tnorm{y-x}_x)} \tnorm{v}_x.
\end{array}\right.
\end{equation}
\end{lemma}

\beforesec
\section{Examples of Inexact Oracles}\label{sec:examples_of_inexact_oracle}
\aftersec
The notion of inexact oracles naturally appears in the context  of Fenchel conjugate, barrier smoothing, inexact computation, and many other situations. 
Below are some examples to show that our definitions of inexact oracle cover many subclasses of convex functions.

\beforesubsec
\subsection{\bf Example 1: The generality of new global inexact oracle}\label{ex:global}
\aftersubsubsec
We will show in this example that the class of convex functions satisfying Definition \ref{de:global_inexact_oracle} is larger than the class of standard self-concordant functions \cite{Nesterov1994} and Lipschitz gradient convex functions.

\vspace{1ex}
\noindent\textbf{(a)~Lipschitz gradient convex functions:}
Let $f$ be a convex function with $L_f$-Lipschitz gradient on $\dom{f} = \R^p$.
Then, $(f, \nabla{f}, \frac{L_f}{4}\Id)$ is a $(\delta_0, \delta_1)$-global inexact oracle of $f$ in the sense of Definition \ref{de:global_inexact_oracle} with $\delta_0 = 1$ and $\delta_1 := 0$. 

\begin{proof}
We have $0 \leq f(y) - f(x) - \iprods{\nabla{f}(x), y - x} \leq \frac{L_f}{2}\norms{y - x}^2_2$ for any $x, y \in\dom{f}$.
The left-hand side inequality of \eqref{eq:global_inexact_oracle} automatically holds since $\delta_0 = 1$. Now, note that $\frac{\tau^2}{2} \leq \omega_{\ast}(\tau)$ for all $\tau \in [0, 1)$.
Hence, using $H(x) = \frac{L_f}{4}\Id$, we can show that
\begin{equation*}
\frac{L_f}{2}\norm{y-x}_2^2 =   \frac{4\tnorm{y-x}^2_{x}}{2} \leq \omega_{\ast}(2\tnorm{y-x}_x),
\end{equation*}
provided that $\tnorm{y-x}_x < \frac{1}{2}$.
Therefore, we obtain $f(y) - f(x) - \iprods{\nabla{f}(x), y - x} \leq \omega_{\ast}(2\tnorm{y-x}_x)$, which means that the right-hand side of \eqref{eq:global_inexact_oracle} holds. 
The left-hand side inequality in \eqref{eq:global_inexact_oracle} of Definition \ref{de:global_inexact_oracle} automatically holds since $\dom{f} = \R^p$.
We emphasize that the class of Lipschitz gradient convex functions plays a central role in gradient and accelerated gradient-type methods \cite{Nesterov2004}, including inexact variants, e.g.,  in \cite{Devolder2010}.
\Eproof
\vspace{-1ex}
\end{proof}

\noindent\textbf{(b)~The sum of self-concordant and convex functions:}
Let us consider a function $f$ composed of a self-concordant  function $f_1$ and a possibly nonsmooth and convex function  $f_2$ as:
\begin{equation}\label{eq:exam1}
f(x) := f_1(x) + f_2(x).
\end{equation}
We have $\dom{f} = \dom{f_1}\cap\dom{f_2}$.
We assume that for any $g_2(x) \in \partial f_2(x)$  there exists a constant $\delta_1 > 0$, which may depend on $x$, such that  for any $x,y \in \dom{f}$ with $\tnorm{y-x}_x <1$, one has
\begin{equation}
\label{ex_f2nons}
f_2(y) - f_2(x) -  \iprods{g_2(x), y-x} \leq \delta_1.
\end{equation}
We construct  the following three quantities for $f$ defined by \eqref{eq:exam1}:
\begin{equation*}
\tilde{f}(x) := f_1(x) +  f_2(x), ~g(x) := \nabla f_1(x) +  g_2(x) ~ \text{for any} ~g_2(x) \in \partial f_2(x), ~\text{and} ~H(x) := \nabla^2 f_1(x).
\end{equation*}
Then, we can show that $(\tilde{f}, g, H)$ is a $(0, \delta_1)$-global inexact oracle of $f$ by Definition \ref{de:global_inexact_oracle}. 

\begin{proof}
Since $f_1$ is  self-concordant, it holds that
\begin{equation*}
\omega\left(\norm{y - x}_x \right) \leq f_1(y) - f_1(x) -  \iprods{\nabla f_1(x), y-x} \leq  \omega_{\ast}\left(\norm{y - x}_x\right),~~\forall x, y\in\dom{f_1},
\end{equation*}
for all $x,y \in \dom{f}$, where the right-hand side inequality holds for any $\tnorm{y-x}_x = \norms{y-x}_x <1$.
Moreover, by  convexity  of $f_2$ and \eqref{ex_f2nons} we also have
\begin{equation*}
0 \leq f_2(y) - f_2(x) -  \iprods{g_2(x), y-x} \leq \delta_1, ~~\forall x, y\in\dom{f}.
\end{equation*}
Summing up these two inequalities,  we can easily show that the triple $(\tilde{f}, g, H)$ defined above satisfies \eqref{eq:global_inexact_oracle} for  $(0, \delta_1)$-global inexact oracle. 
\Eproof
\vspace{-1ex}
\end{proof}

As a special case of \eqref{eq:exam1}, let us consider a function $f(x) := f_1(x) + \beta f_2(x)$, where $f_1$ is  a self-concordant barrier, $f_2$ is a Lipschitz continuous and convex, but possibly nonsmooth function, and  $\beta > 0$ is a given parameter. 
Assume that  $\dom{f} = \dom{f_1}\cap\dom{f_2}$ is bounded.
(In particular, if $\dom{f_1}$ or $\dom{f_2}$ is bounded, then $\dom{f}$ is bounded.)
Hence, the diameter of $\dom{f}$, $\Dc := \sup_{x,y \in \dom{f}} \norm{x-y}_2$ is finite.
Moreover, since $f_2$ is Lipschitz continuous, there exists  $L_2 > 0$ such that $\vert f_2(x) - f_2(y) \vert \leq L_2 \norms{x-y}_2$  for all $x,y \in \dom{f_2}$. 
We have $\sup_{x \in \dom{f_2}}\set{ \norms{g_2(x)}_2 \mid g_2(x) \in \partial{f_2(x)}} \leq L_2$.
Using these two facts, we can show that
\begin{equation*}
0 \leq f_2(y) - f_2(x) -  \iprods{g_2(x), y-x} \leq L_2 \norm{x-y}_2 + \norm{g_2(x)}_2 \norm{x-y}_2 \leq 2L_2\Dc,~~\forall x, y\in\dom{f}.
\end{equation*}
Therefore, we can construct a $(0,\delta_1)$-global inexact oracle for $f$ with $\delta_1 := 2\beta L_2\Dc$.

\noindent\textbf{(c)~An example with unbounded domain:}
The boundedness of  $\dom{f}$ in the previous example is not necessary.
For example, let us choose
\begin{equation*}
f(x) := f_1(x) + f_2(x), ~\text{where } f_1(x) := -\ln(x)~~\text{and}~~f_2(x) := \max\set{\delta_1, \delta_1 x}~~\text{for any $\delta_1 > 0$}.
\end{equation*}
It is clear that $\dom{f} = \set{ x\in\R \mid x > 0}$, which is unbounded. 
If we take $\tilde{f}(x) := f_1(x) + f_2(x)$, $g(x) := f_1^{\prime}(x) + g_2(x)$, with $g_2(x) \in \partial{f_2}(x)$, and $H(x) := f_1^{\prime\prime}(x)$, then we can show that $(\tilde{f}, g, H)$ is a $(0, \delta_1)$-global inexact oracle of $f$. 

\begin{proof}
Processing as before, the left-hand side inequality of \eqref{eq:global_inexact_oracle} holds for $\delta_0=0$.
The right-hand side inequality of \eqref{eq:global_inexact_oracle} has to hold for $\tnorm{y-x}_x < \frac{1}{1+\delta_0}$, which induces a bound on  $y$ of the form $(y-x)^2/x^2 < 1/(1+\delta_0)^2$, that is, for $\delta_0=0$ we have  $0 < y < 2x$. Then, by using this fact and the definition of $f_2$, we can show that (the detailed proof of this inequality is given in Appendix~\ref{app:detail_proofs}(b)):
\begin{equation}\label{eq:exam1c_bound}
f_2(y) - f_2(x) -  \iprods{g_2(x), y-x} \leq \delta_1,~~\forall x,y \in \dom{f}, ~~ \tnorm{y-x}_x < 1.
\end{equation}
This shows that the triple $(\tilde{f}, g, H)$ is a $(0, \delta_1)$-global inexact oracle of the nonsmooth convex function $f$ with unbounded domain.
\Eproof
\end{proof}
\vspace{-1ex}

\beforesubsec
\subsection{\bf Example 2: Inexact computation}\label{ex:inego}
\aftersubsec
It is natural to approximate the function value $f(x)$ at $x$ by $\tilde{f}(x)$ such that $\absc{f(x) - \hat{f}(x)} \leq \varepsilon$ for some $\varepsilon \geq 0$.
In this case, we can define a new inexact oracle as follows.
Assume that the triple $(\tilde{f}, g, H)$ satisfies the following inequalities:
\begin{equation}\label{eq:inexact_oracle2}
\left\{\begin{array}{lll}
&\absc{\hat{f}(x) - f(x)} &\leq \varepsilon, \vspace{1ex}\\
&\tnorm{g(x)  - \nabla{f}(x)}_{x}^{\ast} &\leq \delta_2, \vspace{1ex}\\
&(1-\delta_3)^2 \nabla^2{f}(x) &\preceq H(x) \preceq (1+\delta_3)^2 \nabla^2{f}(x),
\end{array}\right. ~~\forall x \in\dom{f}.
\end{equation}
where $\varepsilon \geq 0$, $\delta_2 \geq 0$,  and $\delta_3 \in [0, 1)$.
In addition, $H$ satisfies the condition that for any $x\in\dom{f}$, if $\tnorm{y - x}_x < \frac{1}{1+\delta_0}$ for $y\in\R^p$, then $y\in\dom{f}$. 

The last condition on $H(x)$ automatically holds if $\dom{f} = \R^p$.
It also holds if $f(x) := f_1(x) + f_2(x)$ and $H(x) := f''_1(x)$ as in Example 1(b) above such that $\dom{f} = \dom{f_1}$.{\!\!\!\!}

Clearly, if $f$ is self-concordant, then \eqref{eq:inexact_oracle2} is more restrictive than the oracles defined in Definitions~\ref{de:global_inexact_oracle} and  \ref{de:local_inexact_oracle} as we show in Lemma~\ref{le:inexact_oracle_bounds} below, whose proof is given in Appendix \ref{apdx:le:inexact_oracle_bounds}.

\begin{lemma}\label{le:inexact_oracle_bounds}
Let $f$ be a standard self-concordant function and $(\hat{f}, g, H)$  satisfy the condition \eqref{eq:inexact_oracle2} for some $\delta_2\geq 0$ and $\delta_3 \in [0,1)$ such that $2\delta_2 + \delta_3 < 1$.
Define 
\begin{equation}\label{eq:psi_stars}
\left\{\begin{array}{lll}
\underline{c}_{23} &:=  \sqrt{(1-\delta_2 - \delta_3)^2 + (1-\delta_3)(1-2\delta_2 - \delta_3)} &\geq 0, \vspace{1ex}\\
\bar{c}_{23} &:=  \sqrt{3(1+\delta_2+\delta_3)^2 - (1+\delta_3)(1+2\delta_2+\delta_3)} &\geq 0, \vspace{1ex}\\
\underline{\psi}^{*}(\delta_2,\delta_3) &:=  \frac{\delta_2}{ (1-\delta_2-\delta_3) + \underline{c}_{23} } - \ln\left(1 + \frac{2\delta_2}{ (2- 3\delta_2 - 2\delta_3) + \underline{c}_{23}}\right) & \geq 0, \vspace{1ex}\\
\bar{\psi}^{*}(\delta_2,\delta_3) &:= \frac{3\delta_2}{3(1+\delta_2+\delta_3) + \rvnn{\bar{c}_{23}}} - \ln\left(1 + \frac{2\delta_2}{2(1+\delta_2+\delta_3) + \bar{c}_{23}}\right) & \geq 0.
\end{array}\right.
\end{equation}
Let us also define $\delta_0 := 2\delta_2 + \delta_3 \in (0, 1]$, $\delta_1 := \max\set{0, 2\varepsilon + \bar{\psi}^{*}(\delta_2,\delta_3) - \underline{\psi}^{*}(\delta_2,\delta_3)} \geq 0$, and $\tilde{f}(x) := \hat{f}(x) - \varepsilon + \underline{\psi}^{*}(\delta_2,\delta_3)$.
Then, $(\tilde{f}, g, H)$ is a $(\delta_0, \delta_1)$-global inexact oracle of $f$ as in Definition~\ref{de:global_inexact_oracle}.
\end{lemma}

\beforesubsubsec
\subsection{\bf Example 3: Fenchel conjugates}\label{subsec:Fen_conj}
\aftersubsubsec
Any convex function $f$ can be written as $f(x) = \sup_{y}\set{ \iprods{x, y} - f^{\ast}(y)}$, where $f^{\ast}$ is the Fenchel conjugate of $f$.
Borrowing this interpretation, we consider the following general convex function:
\begin{equation}\label{eq:max_func}
f(x) := \max_{u \in\dom{\varphi}}\big\{ \iprods{u,A^{\top}x} - \varphi(u) \big\},
\end{equation}
where $\varphi : \R^n\to\Rext$ is a self-concordant function, and $A\in\R^{p\times n}$ is a given linear operator.
In order to evaluate $f$ and its derivatives, we need to solve the following convex program:
\begin{equation}\label{eq:max_func_arg}
u^{\ast}(x) := \mathrm{arg}{\!\!\!\!\!\!\!}\min_{u \in\dom{\varphi}}\big\{ \varphi(u) - \iprods{u,A^{\top}x} \big\}, ~~\text{or equivalent to}~~\nabla{\varphi}(u^{\ast}(x)) - A^{\top}x = 0.
\end{equation}
Clearly,  $u^{\ast}(x) = \nabla\varphi^{\ast}(A^{\top}x)$.
As shown in \cite{Nesterov1994}, $f$ defined by \eqref{eq:max_func} is convex, twice differentiable, and  self-concordant on
\begin{equation*}
\dom{f} := \set{x \in\R^p \mid \varphi(u) - \iprods{u, A^{\top}x} ~\text{is bounded from below on $\dom{\varphi}$}}.
\end{equation*}
The exact gradient and Hessian maps of $f$ are respectively given by
\begin{equation}\label{eq:exact_derivs}
\nabla f(x) = A u^{\ast}(x)~~\text{and}~~\nabla^2 f(x) = A [\nabla^2 \varphi(u^{\ast}(x))]^{-1} A^{\top}.
\end{equation}
However, in many settings, we can only approximate $u^{\ast}(x)$ by $\tilde{u}^{\ast}(x)$ up to a given accuracy $\delta$ in the following sense, which leads to inexact evaluations of $\nabla{f}$ and $\nabla^2{f}$.

\begin{definition}\label{de:def_inexact_a}
Given $x \in \dom{f}$ and $\delta \in [0, 1)$, we say that $\tilde{u}^{\ast}(x)\in \dom{\varphi}$ is a $\delta$-solution of \eqref{eq:max_func_arg} if  $\delta(x) := \norm{\tilde{u}^{\ast}(x) - u^{\ast}(x)}_{\tilde{u}^{\ast}(x)} \leq \delta$,  where the local norm is defined w.r.t. $\nabla^2{\varphi}(\cdot)$.
\end{definition}

For $\tilde{u}^{\ast}(\cdot)$ given in Definition \ref{de:def_inexact_a}, we define
\begin{equation}\label{eq:inexact_oracle3}
\begin{array}{l}
\tilde{f}(x) := \iprods{\tilde{u}^{\ast}(x),A^{\top}x} - \varphi(\tilde{u}^{\ast}(x)), ~~~ g(x) := A\tilde{u}^{\ast}(x), ~\text{and}~ H(x) := A [\nabla^2{\varphi}(\tilde{u}^{\ast}(x))]^{-1} A^{\top}.
\end{array}
\end{equation}
We show in the following lemma that this triplet satisfies our conditions for inexact oracles.
In addition, since $u^{\ast}(x)$ is unknown, it is impractical to check $\delta(x) \leq \delta$ directly.
We show how to guarantee this condition by approximately checking the optimality condition of  \eqref{eq:max_func_arg} in the following lemma, whose proof is given in Appendix \ref{apdx:le:est_A}.

\begin{lemma}\label{le:est_A}
Let $\tilde{u}^{\ast}(\cdot)$ be a $\delta$-approximate solution of $u^{\ast}(\cdot)$ in the sense of Definition \ref{de:def_inexact_a} and $(\tilde{f}, g, H)$ be given by \eqref{eq:inexact_oracle3}.
If $\delta\in [0, 0.292]$, then $(\tilde{f}, g, H)$ with $\hat{f}(x) := \tilde{f}(x) - \omega_{\ast}\big(\frac{\delta}{1-\delta}\big) + \underline{\psi}^{*}(\delta, \frac{\delta}{1-\delta})$ is also a $(\delta_0, \delta_1)$-global inexact oracle of $f$ as in Definition \ref{de:global_inexact_oracle}, where $\delta_0, \delta_1$, and  $\underline{\psi}^{*}(\cdot)$  are defined in Lemma~\ref{le:inexact_oracle_bounds}.

Moreover,  we have the following estimates:
\begin{equation}\label{eq:bounds_exam1}
\tnorm{g(x) - \nabla f(x)}_x^{\ast} \leq \delta ~~~\text{and}~~(1-\delta_3)^2 \nabla^2 f(x)  \preceq H(x) \preceq (1+\delta_3)^2 \nabla^2 f(x),~~\text{with}~~\delta_3 := \tfrac{\delta}{1-\delta}.
\end{equation}
If $\norms{\nabla{\varphi}(\tilde{u}^{\ast}(x)) - A^{\top}x}_{\tilde{u}^{\ast}(x)}^{\ast} \leq \frac{\delta}{1 + \delta}$, then we can guarantee that $\delta(x) := \norm{\tilde{u}^{\ast}(x) - u^{\ast}(x)}_{\tilde{u}^{\ast}(x)} \leq \delta$.
\end{lemma}

As an example of \eqref{eq:max_func}, we consider the following constrained convex optimization problem:
\begin{equation*}
\min_{u\in\R^n}\Big\{ \phi(u) ~~\textrm{s.t.}~~ Au = b, ~u \in \Uc \Big\},
\end{equation*}
where $\phi$ is a standard self-concordant function, $A\in\R^{n\times p}$, $b\in\R^n$, and $\Uc$ is a nonempty, closed, and convex set in $\R^n$ that admits a self-concordant barrier (see \cite{Nesterov1994,Nesterov2004}).
The dual function associated with this problem is defined as
\begin{equation*}
f(x) := \displaystyle\max_{u\in\R^n}\set{  \iprods{x, Au - b} - \phi(u) \mid u\in\Uc}
\end{equation*}
is convex and differentiable, but does not have  Lipschitz gradient and is not  self-concordant in general.
Hence, we often smooth it using a self-concordant barrier function $b_{\Uc}$ of $\Uc$ to obtain
\begin{equation}\label{eq:barrier_smoothing}
f_{\gamma}(x) := \max_{u}\Big\{ \iprods{x, Au - b} - \phi(u) - \gamma b_{\Uc}(u) \Big\},
\end{equation}
where $\gamma > 0$ is a smoothness parameter. When $\gamma$ is sufficiently small, $f_{\gamma}(x)$ can be considered as an approximation of the dual function $f(x)$ at $x$.
Note that in this case $\varphi= \phi + \gamma b_{\Uc}$.
Similar to \eqref{eq:max_func_arg}, very often, we cannot solve the maximization problem \eqref{eq:barrier_smoothing} exactly to evaluate $f$ and its derivatives. We only obtain an approximate solution $\tilde{u}_{\gamma}^{\ast}(x)$ of its true solution $u_{\gamma}^{\ast}(x)$.
In this case, the oracle we obtain via $\tilde{u}_{\gamma}^{\ast}(\cdot)$ generates an inexact oracle for the function $f_{\gamma}(\cdot)$.

\beforesec
\section{Inexact Proximal Newton Methods Using Inexact Oracles}\label{sec:proximal_nt_methods}
\aftersec
We utilize our inexact oracles to develop an \underline{\textbf{i}}nexact \underline{\textbf{P}}roximal \underline{\textbf{N}}ewton \underline{\textbf{A}}lgorithm \eqref{eq:prox_nt_scheme} for solving \eqref{eq:c_sc_min}.
Our algorithm allows one to use both inexact oracles and inexact computation for the proximal Newton direction. Therefore, it is different from some recent works on this topic such as \cite{gao2016quasi,Lu2016a,zhang2015disco}. Note that 
\cite{Lu2016a,zhang2015disco} only focus on inexact computation of Newton-type directions, while \cite{gao2016quasi} approximates Hessian mappings using quasi-Newton schemes.
Our approach combines both aspects but for a more general setting.

\beforesubsec
\subsection{\bf Scaled proximal operator and fixed-point formulation}
\aftersubsec
A key component of our algorithm is the following scaled proximal operator of a proper, closed, and convex function $R$:
\begin{equation}\label{eq:scaled_prox_oper}
\Pc_x(u):=(\Id + H(x)^{-1}\partial R)^{-1}(u),
\end{equation}
where $H(x) \in \Sc^p_{++}$ is a given symmetric positive definite matrix at $x\in\dom{f}$ and $\Id$ is the identity mapping in $\R^p$.
Evaluating $\Pc_x(u)$ is equivalent to solving the following strongly convex subproblem defined through the local norm $\tnorm{\cdot}_x$:
\begin{equation}\label{eq:scvx_subprob}
\Pc_x(u) = \argmin_{z\in\R^n}\Big\{ R(z) + \tfrac{1}{2}\tnorm{z - u}_x^2 \Big\}.
\end{equation}
If $H(x)$ is the identity matrix, then \eqref{eq:scaled_prox_oper} reduces to a standard proximal operator.
In this case, we say that $R$ is proximally tractable if $\Pc_x(\cdot)$ can be computed efficiently (e.g.,  by a closed form or a low-order polynomial-time algorithm).
Examples of tractably proximal functions can be found, e.g.,  in \cite{Parikh2013}. 
One key property of  $\Pc_x(\cdot)$ defined by \eqref{eq:scaled_prox_oper} is the nonexpansiveness \cite[Lemma 2]{Tran-Dinh2013a}:
\begin{equation}\label{eq:nonexpansiveness}
\tnorm{\mathcal{P}_x(u)-\mathcal{P}_x(v)}_x\leq \tnorm{u-v}_x,~~\forall u, v\in\R^p,~x\in\dom{f}.
\end{equation}
Assuming that $\relint{\dom{f}}\cap\relint{\dom{R}} \neq\emptyset$, we can express the optimality condition of \eqref{eq:c_sc_min} as
\begin{equation}\label{eq:opt_cond}
0\in \partial{F}(x^{\star}) \equiv \nabla f(x^{\star})+\partial R(x^{\star})
\end{equation}
Using $\Pc_x$ and for any $x\in\dom{f}$, we can reformulate \eqref{eq:opt_cond} into a fixed-point formulation as:
\begin{equation}
\label{opt:quad0}
\begin{array}{lrl}
& x^{\star} - H(x)^{-1}\nabla f(x^{\star}) & \in x^{\star} + H(x)^{-1}\partial R(x^{\star}) \iff  x^{\star}  =\mathcal{P}_{x}\left(x^{\star} - H(x)^{-1}\nabla f(x^{\star})\right).
\end{array}
\end{equation}
This shows that $x^{\star}$ is a fixed-point of the mapping $\mathcal{R}_x(\cdot) := \Pc_x\left(\cdot - H(x)^{-1}\nabla{f}(\cdot) \right)$.

\beforesubsec
\subsection{\bf \ref{eq:prox_nt_scheme} with global inexact oracle: Global convergence}\label{pN_inx}
\aftersubsec
In this subsection, we  describe our new \eqref{eq:prox_nt_scheme} to solve \eqref{eq:c_sc_min} under  general setting. 

\beforeparagraph
\paragraph{\bf The inexact proximal Newton scheme:}
Given a  global   inexact oracle $(\tilde{f}, g, H)$ of $f$, we first build a quadratic surrogate of $f$ at $x^k\in\dom{F}$ as
\begin{equation*}
\Qc(x; x^k):= \tilde{f}(x^k) + \iprods{g(x^k), x - x^k} + \tfrac{1}{2}\iprods{H(x^k)(x-x^k),x-x^k}.
\end{equation*}
\eqref{eq:prox_nt_scheme}  for solving \eqref{eq:c_sc_min} consists of two steps:
\begin{equation}\label{eq:prox_nt_scheme}
\left\{\begin{array}{ll}
z^k &:\approx \displaystyle\argmin_{x\in\R^p}\big\{ \hat{F}_k(x) := \Qc(x; x^k) + R(x) \big\} \vspace{1.5ex}\\
x^{k+1} &:= (1-\alpha_k)x^k + \alpha_kz^k = x^k + \alpha_kd^k~~~\text{with}~d^k := z^k - x^k,
\end{array}\right.\tag{iPNA}
\end{equation}
where $d^k$ is called an inexact proximal Newton direction,  $\alpha_k\in (0, 1]$ is a given stepsize, and the approximation $:\approx$ means that $z^k$ is computed until satisfying the following stopping criterion:
\begin{equation}\label{eq:inexact_subp}
\tnorm{\nu^k}_{x^k}^{\ast} \leq \delta_4 \ \tnorm{z^k - x^k}_{x^k}, \; \text{where} \;  \nu^k\in g(x^k)+H(x^k)(z^k - x^k)+\partial R(z^k) \; \text{and} \; \delta_4 \in [0, 1).
\end{equation}
Note that one can solve the subproblem in \eqref{eq:prox_nt_scheme} by any first-order scheme, such as  FISTA \cite{Beck2009}, and check criterion \eqref{eq:inexact_subp} as described  in Appendix \ref{app:sub}.
One can also easily adapt the accuracy $\delta_4 := \delta_4^k \in [0, 1)$ at each iteration. However, for the sake of notation, we use the same $\delta_4$.
Clearly, if $\delta_4 = 0$, then $z^k = \bar{z}^k := \argmin_{x\in\R^p}\big\{ \hat{F}_k(x) := \Qc(x; x^k) + R(x) \big\}$, the exact solution of the subproblem in \eqref{eq:prox_nt_scheme}.
Practically, we do not need to evaluate $\tilde{f}(x^k)$ to form $\Qc$ in \eqref{eq:prox_nt_scheme}.

\beforeparagraph
\paragraph{\bf Global convergence:}
We now state one of our main results, which is the  global convergence of our inexact proximal Newton algorithm.

\begin{theorem}\label{th:global_convergence3}
Assume that $(\tilde{f}, g, H)$ is a $(\delta_0,\delta_1)$-global inexact oracle of $f$ given in Definition~\ref{de:global_inexact_oracle}.
Let $\sets{x^k}$ be the sequence computed by \eqref{eq:prox_nt_scheme} starting from $\xb^0 \in\dom{F}$ and using a stepsize $\alpha_k$ as
\begin{equation}\label{eq:step_size2}
\alpha_k :=\frac{1-\delta_4}{(1 + \delta_0)(1+\delta_0 + (1-\delta_4)\lambda_k)},~~~\text{with}~~~\lambda_k := \tnorm{d^k}_{x^k},
\end{equation}
where $\delta_4 \in [0, 1)$ is the accuracy of the inexact proximal Newton step defined by \eqref{eq:inexact_subp}.
Then the following statements hold:
\begin{compactitem}
\item[$\mathrm{(a)}$] 
The following property on $F$  is satisfied:
\begin{equation}\label{eq:descent_property3}
F(x^{k+1}) \leq F(x^k) - \omega\left(\frac{(1-\delta_4)\lambda_k}{1+\delta_0}\right) + \delta_1.
\end{equation}
Consequently, we have
\begin{equation}\label{eq:convergence_rate}
\min_{0 \leq i \leq k}\omega\left(\frac{(1-\delta_4)\lambda_i}{1+\delta_0}\right) \leq \frac{1}{k+1}\sum_{i=0}^k\omega\left(\frac{(1-\delta_4)\lambda_i}{1+\delta_0}\right) \leq \frac{F(x^0) - F^{\star}}{k+1} + \delta_1.
\end{equation}

\item[$\mathrm{(b)}$] 
Let  $k_{*} := \argmin\set{\lambda_i \mid 0 \leq i \leq k}$ and  $\hat{\varepsilon} :=  \omega\left( \frac{(1-\delta_4)\varepsilon}{1+\delta_0}\right)$ for a given tolerance $\varepsilon > 0$. If $\delta_1 \leq \frac{\hat{\varepsilon}}{2}$, then after 
\begin{equation*}
k := \BigO{\frac{2(F(x^0) - F^{\star})}{\hat{\varepsilon}}}~~ \text{iterations}, 
\end{equation*}
we obtain $\lambda_{k_{*}} \leq \varepsilon$. 

\item[$\mathrm{(c)}$] 
Let us choose $\varepsilon$ such that $0 < \varepsilon < \frac{1}{1+\delta_0}$ and $0 < \delta_1 \leq 1$.
If  $\lambda_{k_{*}} \leq \frac{1}{4(1+\delta_0)}$, then the sequence $\sets{z^k}$  satisfies
\begin{equation*}
\inf_{r(z^{k_{*}})  \in \partial{R}(z^{k_{*}})} \tnorm{\nabla f(z^{k_{*}}) + r(z^{k_{*}})}^{\ast}_{x^{k_{*}}} \equiv \inf_{\nabla{F}^{k_{*}} \in \partial{F}(z^{k_{*}})} \tnorm{\nabla{F}^{k_{*}}}^{\ast}_{x^{k_{*}}} \leq \BigO{ \lambda_{k_{*}} + \sqrt{\delta_1}}.
\end{equation*}
If $\lambda_{k_{*}} \leq \BigO{\varepsilon} \leq \frac{1}{4(1+\delta_0)}$ and $\delta_1 \leq \BigO{\varepsilon^2} \leq 1$, then $\inf_{r(z^{k_{*}})  \in \partial{R}(z^{k_{*}})} \tnorm{\nabla f(z^{k_{*}}) + r(z^{k_{*}})}^{\ast}_{x^{k_{*}}} \leq \BigO{\varepsilon}$, which $\BigO{\varepsilon}$-approximately guarantees the optimality condition of \eqref{eq:c_sc_min} in the weighted norm $\tnorm{\cdot}_{x^{k_{*}}}^*$ $($i.e., up to a constant factor$)$. 
\item[$\mathrm{(d)}$] 
In particular, under the conditions of $\mathrm{(c)}$, if there exists $L \in [0, +\infty)$ such that $H(x^{k_{*}}) \preceq L \Id$ and $\lambda_{k_{*}} \leq \BigO{\varepsilon} \leq \frac{1}{4(1+\delta_0)}$ for $x^{k_{*}} \in \dom{F}$ and $\delta_1 \leq \BigO{\varepsilon^2} \leq 1$, then  
\begin{equation*}
\inf\limits_{\nabla{F}(z^{k_{*}}) \in \partial{F}(z^{k_{*}})}\big\Vert \nabla{F}(z^{k_{*}})\big\Vert_2 \leq \BigO{\varepsilon}.
\end{equation*}
\end{compactitem}
\end{theorem}

\begin{proof}
(a)~From \eqref{eq:inexact_subp} we have $\nu^k - H(x^k)d^k - g(x^k) \in \partial{R}(z^k)$ with $d^k := z^k - x^k$.
Combining this expression, $x^{k+1} := (1-\alpha_k)x^k + \alpha_kz^k$, and the convexity of $R$, we can derive
\begin{equation}\label{eq:th9_proof1}
\begin{array}{ll}
R(x^{k+1}) &\leq (1-\alpha_k)R(x^k) + \alpha_kR(z^k) \vspace{1ex}\\
&\leq R(x^k) + \alpha_k\iprods{\nu^k - H(x^k)d^k - g(x^k), d^k} \vspace{1ex}\\
&= R(x^k) + \alpha_k\iprods{\nu^k, d^k} - \alpha_k\iprods{g(x^k), d^k} - \alpha_k\tnorm{d^k}_{x^k}^2.
\end{array}
\end{equation}
Note that $\tnorm{x^{k+1} - x^k}_{x^k} = \alpha_k\tnorm{d^k}_{x^k} = \alpha_k\lambda_k$.
We need to choose $\alpha_k \in (0, 1)$ such that $\alpha_k\lambda_k < \frac{1}{1+\delta_0}$.
Under this condition, using \eqref{eq:global_inexact_oracle} with $(y, x) = (x^{k+1}, x^k)$ and \eqref{eq:global_inexact_oracle_pro1}, we can derive that
\begin{eqnarray*}
f(x^{k+1}) & \overset{\tiny\eqref{eq:global_inexact_oracle}}{\leq} & \tilde{f}(x^k) + \iprods{g(x^k), x^{k+1} - x^k} + \omega_{\ast}\left( (1+\delta_0)\tnorm{x^{k+1}-x^k}_{x^k} \right) + \delta_1\\
& \overset{\tiny\eqref{eq:global_inexact_oracle_pro1}}{\leq} & f(x^k) + \alpha_k\iprods{g(x^k),d^k} + \omega_{\ast}\left( (1+\delta_0)\alpha_k\tnorm{d^k}_{x^k} \right) + \delta_1.
\end{eqnarray*}
Adding this  to \eqref{eq:th9_proof1} and using $\lambda_k := \tnorm{d^k}_{x^k}$ and the first condition of \eqref{eq:inexact_subp} as $\iprods{\nu^k,d^k} \leq \tnorm{\nu^k}_{x^k}^{\ast} \tnorm{d^k}_{x^k} \leq \delta_4\tnorm{d^k}_{x^k}^2 = \delta_4\lambda_k^2$, we can show that
\begin{equation}\label{eq:descent_subinexact}
F(x^{k+1}) \leq F(x^k) - \alpha_k(1-\delta_4)\lambda_k^2  +~ \omega_{\ast}\left( (1+\delta_0)\alpha_k\tnorm{d^k}_{x^k} \right) + \delta_1.
\end{equation}
Note that the function $s_k(\alpha) := \lambda_k^2(1-\delta_4)\alpha - \omega_{\ast}((1+\delta_0)\lambda_k\alpha) =  \lambda_k^2(1-\delta_4)\alpha + (1+\delta_0)\lambda_k\alpha + \ln\left(1 - (1+\delta_0)\lambda_k\alpha\right)$ is concave in $\alpha$.  
We can find its maximum over $[0, 1]$ by solving the equation $s_k'(\alpha) = \lambda_k^2(1-\delta_4) + (1+\delta_0)\lambda_k - \frac{(1+\delta_0)\lambda_k}{1- (1+\delta_0)\lambda_k\alpha} =  0$, leading to:
\begin{equation*}
\alpha_k := \frac{1-\delta_4}{(1+\delta_0)(1+\delta_0 + (1-\delta_4)\lambda_k)} \in [0, 1],
\end{equation*}
as in \eqref{eq:step_size2} with the optimal value $s_k(\alpha_k) =\omega\left(\frac{(1-\delta_4)\lambda_k}{1+\delta_0}\right)$.
Substituting $s_k(\alpha_k)$ into \eqref{eq:descent_subinexact}, we get
\begin{equation*} 
F(x^{k+1}) \leq F(x^k) - \omega\left(\tfrac{(1-\delta_4)\lambda_k}{1+\delta_0}\right) + \delta_1,
\end{equation*}
which is  exactly \eqref{eq:descent_property3}. 
In addition, we have 
\begin{equation*}
\begin{array}{ll}
(1+\delta_0)\alpha_k\lambda_k = \frac{(1-\delta_4)\lambda_k}{(1+\delta_0 + (1-\delta_4)\lambda_k)} < 1.
\end{array}
\end{equation*}
Therefore, the choice of $\alpha_k$ guarantees $\alpha_k\lambda_k < \frac{1}{1+\delta_0}$.

\noindent 
Now, summing up \eqref{eq:descent_property3} from $i=0$ to $i=k$ and noting that  $F(x^{k+1}) \geq F^{\star} > -\infty$, we obtain
\begin{equation*}
\min_{0 \leq i \leq k}\omega\left(\tfrac{(1-\delta_4)\lambda_i}{1+\delta_0}\right) \leq \frac{1}{k+1}\sum_{i=0}^k \omega\left(\tfrac{(1-\delta_4)\lambda_i}{1+\delta_0}\right)  \leq \frac{F(x^0) - F^{\ast}}{k+1} + \delta_1,
\end{equation*}
which is exactly \eqref{eq:convergence_rate}. 

(b)~If we define ${\lambda}_{k_{*}} := \min\set{ \lambda_i \mid 0 \leq i \leq k}$, then since $\omega$ is an increasing function, we have 
\begin{equation*}
\omega\left(\frac{(1-\delta_4) {\lambda}_{k_{*}}}{1+\delta_0}\right) \leq \min_{0 \leq i \leq k}\omega\left(\tfrac{(1-\delta_4)\lambda_i}{1+\delta_0}\right) \leq \frac{F(x^0) - F^{\ast}}{k+1} + \delta_1.
\end{equation*}
Assume that $\frac{F(x^0) - F^{\ast}}{k+1} \leq \frac{\hat{\varepsilon}}{2}$ and $\delta_1 \leq \frac{\hat{\varepsilon}}{2}$ for some $\hat{\varepsilon} > 0$.
We can overestimate the above inequality as $\omega\left(\frac{(1-\delta_4){\lambda}_{k_{*}}}{1+\delta_0}\right) \leq \hat{\varepsilon}$.
This implies that ${\lambda}_{k_{*}} \leq \frac{(1+\delta_0)\omega^{-1}(\hat{\varepsilon})}{1-\delta_4} := \varepsilon$.
Therefore, we obtain $\hat{\varepsilon} := \omega\left( \frac{(1-\delta_4)\varepsilon}{1+\delta_0}\right)$.
Moreover, the condition $\frac{F(x^0) - F^{\ast}}{k+1} \leq \frac{\hat{\varepsilon}}{2}$ shows that $k = \BigO{\frac{2(F(x^0) - F^{\star})}{\hat{\varepsilon}}}$.

(c)~Next, from the optimality condition of \eqref{eq:inexact_subp}, we have $\nu^i =  g(x^i)+H(x^i)(z^i - x^i) + r(z^i)$, where $r(z^i) \in \partial{R}(z^i)$ for any $0 \leq i \leq k$. 
Since ${\lambda}_{k_{*}} :=\tnorm{x_{k_{*}} - z_{k_{*}}}_{x_{k_{*}}} \leq \varepsilon$  for ${k_{*}} := \argmin\set{\lambda_i \mid 0 \leq i \leq k}$ and $\delta_0 \in [0, 1]$, to guarantee  $z^{k_{*}}  \in \dom{f}$ by Definition \ref{de:global_inexact_oracle}, we need to choose $\varepsilon$ such that $\varepsilon < \frac{1}{1+\delta_0}$.

The  optimality condition of  \eqref{eq:inexact_subp} also leads to
\begin{equation*}
\nabla f(z^{k_{*}})+r(z^{k_{*}}) = -H(x^{k_{*}})(z^{k_{*}} - x^{k_{*}}) + (\nabla f(z^{k_{*}}) - g(x^{k_{*}})) + \nu^{k_{*}}.
\end{equation*}
By the property of the local norm and the definition of the stopping criterion \eqref{eq:inexact_subp}, we have
\begin{equation}\label{eq:th9_proof2}
\begin{array}{ll}
\tnorm{\nabla f(z^{k_{*}}) + r(z^{k_{*}})}^{\ast}_{x^{k_{*}}} & \leq \tnorm{H(x^{k_{*}})(z^{k_{*}} - x^{k_{*}})}_{x^{k_{*}}}^{\ast} + \tnorm{\nu^{k_{*}}}_{x^{k_{*}}}^{\ast} + \tnorm{\nabla f(z^{k_{*}}) - g(x^{k_{*}})}_{x^{k_{*}}}^{\ast}\vspace{1ex}\\
&\leq (1 + \delta_4)\lambda_{k_{*}} + \tnorm{\nabla f(z^{k_{*}}) - g(x^{k_{*}})}_{x^{k_{*}}}^{\ast}.
\end{array}
\end{equation}
From \eqref{eq:global_inexact_oracle_pro6a} it follows  that $\omega\left(\frac{\tnorm{g(x^{k_{*}}) - \nabla{f}(z^{k_{*}})}^{\ast}_{x^{k_{*}}}}{1+\delta_0}\right) \leq \tnorm{g(x^{k_{*}}) - \nabla{f}(z^{k_{*}})}^{\ast}_{x^{k_{*}}}\lambda_{k_{*}} + \delta_1$.
However, since $\omega(\tau) \geq \frac{\tau^2}{2(1+\tau)}$ for $\tau \geq 0$, by denoting $s_{k_{*}} := \frac{1}{1+\delta_0}\tnorm{g(x^{k_{*}}) - \nabla{f}(z^{k_{*}})}^{\ast}_{x^{k_{*}}}$ and $\bar{\lambda}_{k_{*}} := (1+\delta_0)\lambda_{k_{*}}$, the last estimate implies that $\frac{s_{k_{*}}^2}{2(1+s_{k_{*}})}  \leq s_{k_{*}} \bar{\lambda}_{k_{*}} + \delta_1$ and  hence leads to
\begin{equation*} 
(1-2\bar{\lambda}_{k_{*}}) s_{k_{*}}^2 - 2(\bar{\lambda}_{k_{*}}+\delta_1) s_{k_{*}} - 2\delta_1 \leq 0.
\end{equation*}
This is a quadratic inequality of the form $as_{k_{*}}^2 + 2b s_{k_{*}} + c \leq 0$ in $s_{k_{*}} \geq 0$ with $a := (1-2\bar{\lambda}_{k_{*}})$, $b := -(\bar{\lambda}_{k_{*}}+\delta_1)$, and $c := - 2\delta_1$.
Solving this inequality  and noting that $s_{k_{*}} \geq 0$, we obtain 
\begin{equation*}
0 \leq s_{k_{*}} \leq \frac{\sqrt{b^2 - ac} - b}{a} = \frac{\sqrt{(\bar{\lambda}_{k_{*}}  - \delta_1)^2 + 2\delta_1} + (\bar{\lambda}_{k_{*}}  +\delta_1)}{(1 - 2\bar{\lambda}_{k_{*}})}.
\end{equation*}
Assume that $\bar{\lambda}_{k_{\ast}} \leq \frac{1}{4}$ (or equivalently $\lambda_{k_{*}} \leq \frac{1}{4(1+\delta_0)}$) and $\delta_1 \leq 1$.
Then, by upper bounding the numerator $\sqrt{(\bar{\lambda}_{k_{*}}  - \delta_1)^2 + 2\delta_1} + (\bar{\lambda}_{k_{*}}  +\delta_1) \leq 2\bar{\lambda}_{k_{*}} + 4\sqrt{\delta_1}$ and lower bounding the denominator $1 - 2\bar{\lambda}_{k_{*}} \geq \frac{1}{2}$ of the right-hand side in the last inequality, we can overestimate $\frac{\sqrt{(\bar{\lambda}_{k_{*}}  - \delta_1)^2 + 2\delta_1} + (\bar{\lambda}_{k_{*}}  +\delta_1)}{(1-2\bar{\lambda}_{k_{*}})} \leq 4(\bar{\lambda}_{k_{*}} + 2\sqrt{\delta_1})$.
Therefore, we have
\begin{equation}\label{eq:norm_bound2}
\tnorm{g(x^{k_{*}}) - \nabla{f}(z^{k_{*}})}^{\ast}_{x^{k_{*}}} \leq 4(1+\delta_0)(\bar{\lambda}_{k_{*}} + 2\sqrt{\delta_1}) = 4(1+\delta_0)^2\lambda_{k_{*}} + 8(1+\delta_0)\sqrt{\delta_1}.
\end{equation}
Combining  \eqref{eq:th9_proof2} and \eqref{eq:norm_bound2}, we arrive at
\begin{equation*}
\inf_{r(z^{k_{*}}) \in \partial{R}(z^{k_{*}})} \tnorm{\nabla f(z^{k_{*}}) + r(z^{k_{*}})}^{\ast}_{x^{k_{*}}} \leq \left[1 + \delta_4 + 4(1+\delta_0)^2\right]\lambda_{k_{*}} + 8(1+\delta_0)\sqrt{\delta_1} = \BigO{\lambda_{k_{*}} + \sqrt{\delta_1}}.
\end{equation*}
Clearly, if we have $\lambda_{k_{*}} \leq \BigO{\varepsilon} \leq \frac{1}{4(1+\delta_0)}$ and $\delta_1$ is chosen such that $\delta_1 = \BigO{\varepsilon^2} \leq 1$, then the last inequality leads to $\inf_{r(z^{k_{*}}) \in \partial{R}(z^{k_{*}})} \tnorm{\nabla f(z^{k_{*}}) + r(z^{k_{*}})}^{\ast}_{x^{k_{*}}} \leq \BigO{\varepsilon}$, which $\BigO{\varepsilon}$-approximately guarantees the optimality condition of \eqref{eq:c_sc_min} in the weighted norm $\tnorm{\cdot}_{x^{k_{*}}}^*$.

(d)~Finally, the last statement of this theorem is an immediate  consequence of the last estimate since $\frac{1}{\sqrt{L}} \norms{\nabla f(z^{k_{*}}) + r(z^{k_{*}})}_2 \leq \tnorm{\nabla f(z^{k_{*}}) + r(z^{k_{*}})}^{\ast}_{x^{k_{*}}}$.  \hfill$\square$
\end{proof}

\begin{remark}
Since $\lim_{\varepsilon\to 0} \lambda_{k_{*}} = \lim_{\varepsilon\to 0}\tnorm{z^{k_{*}} - x^{k_{*}}}_{x^{k_{*}}} = 0$ in Theorem \ref{th:global_convergence3}, we can see that if there exists $L \in [0, +\infty)$ such that $H(x^{k_{*}}) \preceq L\Id$, $\delta_1 \leq \varepsilon^2$, and $\lambda_{k_{*}} \leq \varepsilon$ for $x^{k_{*}} \in \dom{F}$, then $\lim_{\varepsilon\to 0} z^{k_{*}} = \lim_{\varepsilon\to 0}x^{k_{*}} = x^{\ast}$ if these limits exist $($at least via a subsequence$)$. Hence, by \cite[Theorem 24.4]{Rockafellar1970}, we have $\inf_{r^{\ast} \in \partial{R}(x^{\ast})} \norm{\nabla f(x^{\ast}) + r^{\ast}}_2 = 0$.  
\Eproof
\end{remark}

\begin{remark}
To guarantee only the property \eqref{eq:descent_property3}, one can use a weaker stopping criterion $\iprods{\nu^k,d^k} \leq \delta_4 \lambda_k^2$ along with $\delta_4 < 1$ instead of \eqref{eq:inexact_subp} to avoid the matrix inversion in computing $\tnorm{\nu^k}_{x^k}^{\ast}$. 
In addition, the proof of \eqref{eq:descent_property3} holds using this criterion even when $\delta_4$ is nonpositive. 
\Eproof
\end{remark}

\noindent Note that the statements of Theorem \ref{th:global_convergence3} are also valid  for an adaptive inexact oracle framework. 
More precisely, in contrast to Theorem \ref{th:global_convergence3}, where the accuracy level of the oracle is fixed according to the desired accuracy, we can also consider an adaptive oracle, where its accuracy   in the  early iterations  can be  rough and becomes finer in the last iterations. 
We state such a result in the next theorem, whose proof is similar to Theorem \ref{th:global_convergence3} and is omitted:

\begin{theorem}\label{th:adaptive_global_convergence}
Assume that $(\tilde{f}_k, g_k, H_k)$ is an adaptive  $(\delta_0^k,\delta_1^k)$-global inexact oracle of $f$. Let $\sets{x^k}$ be the sequence computed by \eqref{eq:prox_nt_scheme} starting from $\xb^0$, where $\alpha_k :=\frac{1-\delta_4^k}{(1 + \delta_0^k)(1+\delta_0^k + (1-\delta_4^k)\lambda_k)}$, with $\lambda_k := \tnorm{d^k}_{x^k}$ and $\delta_4^k \in [0, 1)$ given in \eqref{eq:inexact_subp}. 
Then:
\begin{compactitem}
\item[$\mathrm{(a)}$] 
The following descent property holds:
\begin{equation*}
F(x^{k+1}) \leq F(x^k) - \omega\left(\frac{(1-\delta_4^k)\lambda_k}{1+\delta_0^k}\right) + \delta_1^k.
\end{equation*}
\item[$\mathrm{(b)}$] 
Assume, in addition, that $\delta_0^k \in [0, 1]$, $\delta_1^k$ and $\delta_4^k$ are chosen such that $\sum_{k=0}^{\infty}\delta_1^k < +\infty$  and  $0 \leq \delta_4^k \leq {\delta}_4 <1$, then the inexact Newton decrement sequence $\sets{\lambda_k}$ converges to zero as $k\to\infty$. 
\item[$\mathrm{(c)}$] 
Consequently, the sequence $\sets{z^k}$ also satisfies 
\begin{equation*}
\lim_{k\to\infty} \inf_{\nabla{F}(z^k) \in \partial{F}(z^k)} \tnorm{\nabla{F}(z^k)}^{\ast}_{x^k} = 0.
\end{equation*} 
\end{compactitem}
\end{theorem}
\beforesubsec
\subsection{\bf \ref{eq:prox_nt_scheme}  with local inexact oracle: Local convergence under self-concordance}
\aftersubsec
In this subsection, we analyze local convergence of \eqref{eq:prox_nt_scheme} for solving \eqref{eq:c_sc_min} with local inexact oracle under the self-concordance of $f$.
The following lemma is key to our analysis, whose proof is deferred to Appendix \ref{apdx:le:key_est_for_prox_nt_scheme_ine}.

\begin{lemma}\label{le:key_est_for_prox_nt_scheme_ine}
Let $\sets{x^k}$ be the sequence generated by \eqref{eq:prox_nt_scheme} and  $\lambda_k$ be defined by \eqref{eq:step_size2}.
If $\alpha_k\lambda_k + \delta_3 < 1$ and $\delta_4 < 1$, then we have:
\begin{equation}\label{eq:key_local_estimate_com_ine}
\begin{array}{lcl}
\lambda_{k+1} & \leq & \frac{\delta_2}{1-\delta_4} + \frac{(1 + \delta_3)}{(1-\delta_4)\left( 1 - \delta_3 - \alpha_k\lambda_k\right)} \cdot \left[\delta_2 +  \delta_4\lambda_k + 3(1-\alpha_k)\lambda_k \right]  \vspace{1ex}\\
& & + {~} \frac{1}{(1-\delta_3)(1-\delta_4)\left( 1 - \delta_3 - \alpha_k\lambda_k\right)} \cdot \left[ \alpha_k(2+\delta_3)\delta_3\lambda_k   + \frac{\alpha_k^2\lambda_k^2}{1 - \delta_3 - \alpha_k\lambda_k} \right].
\end{array}
\end{equation}
\end{lemma}

Based on Lemma \ref{le:key_est_for_prox_nt_scheme_ine} and using  either a full step or a damped step scheme  we can prove  local convergence rates for  \eqref{eq:prox_nt_scheme}  in the following theorems.

\begin{theorem}\label{th:local_covergence1}
Let $\sets{x^k}$ be the sequence generated by \eqref{eq:prox_nt_scheme} using  the full step scheme, i.e., $\alpha_k := 1$ for all $k\geq 0$.
For $\delta_3$  in \eqref{eq:local_inexact_oracle} and $\delta_4$ in \eqref{eq:inexact_subp}, if $0 \leq \delta_3, \delta_4 \leq \delta \leq \frac{1}{100}$ and $0 \leq \lambda_k \leq \frac{1}{20}$, then
\begin{equation}\label{eq:main_bound_th1}
\lambda_{k+1} \leq \frac{1}{10}\big(12\lambda_k^2 + 33\delta\lambda_k + 21\delta_2\big).
\end{equation}
Moreover, for any $\varepsilon \in (0, \frac{1}{20})$ and $x^0\in\dom{F}$ such that $\lambda_0 \leq \frac{1}{20}$, the following statements hold:
\begin{compactitem}
\item[$\mathrm{(a)}$]
If we choose $0 \leq \delta_3, \delta_4 \leq \frac{1}{100}$ and $\delta_2 := \frac{5\varepsilon}{24}$, then after at most $k := \left\lfloor \frac{9}{20}\ln(\frac{1}{\varepsilon})\right\rfloor$ iterations, we have $\lambda_k \leq \varepsilon$ $($i.e., locally linear convergence rate$)$.

\vspace{1ex}
\item[$\mathrm{(b)}$]
If we choose $0 \leq \delta_3, \delta_4 \leq \frac{1}{200}\big(\frac{3}{5}\big)^{\ln(\frac{1}{\varepsilon})}$ and $0 \leq \delta_2 \leq  \frac{3}{5}\big(\frac{1}{5}\big)^{\ln(\frac{1}{\varepsilon})}$, then after at most $k := \left\lfloor \log_{3/2}(\ln(\frac{1}{\varepsilon}))\right\rfloor$ iterations, we get $\lambda_k \leq \varepsilon$ $($i.e., locally  superlinear convergence rate$)$.

\vspace{1ex}
\item[$\mathrm{(c)}$]
If we choose $0 \leq \delta_3, \delta_4\leq \frac{1}{10}(\frac{7}{20})^{\ln(\frac{1}{\varepsilon})}$ and $0 \leq \delta_2 \leq 28(\frac{3}{25})^{\ln(\frac{1}{\varepsilon})}$, then after at most $k := \left\lfloor \log_2(\ln(\frac{1}{\varepsilon}))\right\rfloor$ iterations, we obtain $\lambda_k \leq \varepsilon$ $($i.e., locally  quadratic convergence rate$)$.

\end{compactitem}
\vspace{1ex}
In addition, we have
\begin{equation*}
\inf_{\nabla{F}(z^k)\in\partial{F}(z^k)}\tnorm{\nabla{F}(z^k)}_{x^k}^{\ast} \leq \BigO{ \lambda_k + \sqrt{\delta_1} }.
\end{equation*}
Hence, if $\lambda_k \leq \BigO{\varepsilon}\leq \frac{1}{4(1+\delta_0)}$ and $\delta_1 \leq \BigO{\varepsilon^2} \leq 1$, then $\inf_{\nabla{F}(z^k) \in\partial{F}(z^k)}\tnorm{\nabla{F}(z^k)}_{x^k}^{\ast} \leq \BigO{\varepsilon}$.
\end{theorem}

\begin{proof}
By Lemma~\ref{le:apdx-monotone}, the right-hand side $H_1(\cdot)$ of  \eqref{eq:key_local_estimate_com_ine} with $\alpha_k = 1$ is monotonically increasing w.r.t. each variable, in particular, w.r.t. $\delta_3,\delta_4\in [0, 1)$.
If $0 \leq \delta_3 \leq \delta$, and $0 \leq \delta_4\leq\delta$ for some $\delta \in [0, 1)$, then \eqref{eq:key_local_estimate_com_ine} can be overestimated as
\begin{equation}\label{eq:th4_proof1} 
\begin{array}{ll}
\lambda_{k+1}  &\leq \frac{(2-\lambda_k)\delta_2}{(1-\delta)(1-\delta-\lambda_k)} +  \frac{(3 + \delta -\delta^2)\delta\lambda_k}{(1-\delta)^2(1-\delta-\lambda_k)} + \frac{\lambda_k^2}{(1-\delta)^2(1-\delta-\lambda_k)^2}.
\end{array}
\end{equation}
Now, assume that $\delta := \frac{1}{100}$ and $\lambda_k \leq \frac{1}{20}$.
Since the right-hand side of the last estimate is nondecreasing in $\lambda_k$, we have $\frac{(2-\lambda_k)\delta_2}{(1-\delta)(1-\delta-\lambda_k)}\leq 2.1\delta_2$, $\frac{(3 + \delta -\delta^2)\delta\lambda_k}{(1-\delta)^2(1-\delta-\lambda_k)}\leq 3.3\delta\lambda_k$, and $ \frac{\lambda_k^2}{(1-\delta)^2(1-\delta-\lambda_k)^2} \leq 1.2\lambda^2_k$.
Combining these three estimates, we can overestimate the above inequality as
\begin{equation*} 
\lambda_{k+1} \leq \frac{1}{10}\big(12\lambda_k^2 + 33\delta\lambda_k + 21\delta_2\big),
\end{equation*}
which is exactly \eqref{eq:main_bound_th1}.

$\mathrm{(a)}$~Since $\lambda_k \leq \frac{1}{20}$ and $\delta := \frac{1}{100}$, we can numerically overestimate \eqref{eq:main_bound_th1} as $\lambda_{k+1} \leq 0.1\lambda_k + 2.1\delta_2$.
Since $\lambda_0 \leq \frac{1}{20}$, by induction, we have $\lambda_k \leq (0.1)^k\lambda_0 + \frac{2.1\delta_2}{1-0.1} \leq \frac{(0.1)^k}{20} + 2.4\delta_2$.
Assume that $\delta_2 := \frac{\varepsilon}{4.8} = \frac{5\varepsilon}{24}$.
Then, from the last estimate, if we impose $\frac{(0.1)^k}{20} \leq \frac{\varepsilon}{2}$, then we have $\lambda_k \leq \varepsilon$.
This condition leads to $k \geq 0.45\ln(\frac{1}{\varepsilon}) - 1$.
Therefore, we can choose $k := \left\lfloor \frac{9}{20}\ln(\frac{1}{\varepsilon})\right\rfloor$. 

$\mathrm{(b)}$~
After $k$ iterations  of the full step scheme \eqref{eq:prox_nt_scheme}, we have $\delta \leq \frac{1}{10}\lambda_k^{\frac{1}{2}}$ and $\delta_2 \leq \lambda_k^{\frac{3}{2}}$.
Using these inequalities, and $\lambda_k \leq \frac{1}{20}$, we can numerically overestimate \eqref{eq:th4_proof1} as 
\begin{equation*}
\lambda_{k+1} \leq 2\lambda_k^{\frac{3}{2}}  +  0.31\lambda_k^{\frac{3}{2}} + 0.3\lambda_k^{\frac{3}{2}} \leq  \frac{27}{10}\lambda_k^{\frac{3}{2}}.
\end{equation*}
This is equivalent to  $c\lambda_{k+1} \leq (c\lambda_k)^{\frac{3}{2}}$, where $c := (\frac{27}{10})^2$.
By induction, we have $c\lambda_k \leq (c\lambda_0)^{(\frac{3}{2})^k} \leq (\frac{729}{2000})^{(\frac{3}{2})^k}$, which leads to $\lambda_k \leq \frac{100}{729}(\frac{729}{2000})^{(\frac{3}{2})^k}$.

Given $\varepsilon \in (0, \frac{1}{20})$, to obtain $\lambda_k \leq \varepsilon$, we need to impose $(\frac{3}{2})^k \geq \ln(\frac{1}{\varepsilon}) - 1$.
Hence, we can choose $k := \left\lfloor \log_{\frac{3}{2}}(\ln(\frac{1}{\varepsilon}))\right\rfloor$ as the maximum number of iterations.
Moreover, to bound $\delta$ and $\delta_2$ in terms of $\varepsilon$, we use $(\frac{3}{2})^k \geq \ln(\frac{1}{\varepsilon})-1$ to show that these quantities can be bounded by $\delta := \frac{1}{200}(\frac{3}{5})^{\ln(\frac{1}{\varepsilon})}$, and $\delta_2 := \frac{3}{5}(\frac{1}{5})^{\ln(\frac{1}{\varepsilon})}$, respectively.

$\mathrm{(c)}$~
Similar to $\mathrm{(b)}$, we assume that we run the full step scheme \eqref{eq:prox_nt_scheme} for $k$ iterations.
Then, after $k$ iterations, we have $\delta \leq \frac{\lambda_k}{5}$ and $\delta_2 \leq  5\lambda_k^2$.
Using these inequalities, and $\lambda_k \leq \frac{1}{10}$, we can numerically overestimate  \eqref{eq:th4_proof1} as 
\begin{equation*}
\lambda_{k+1} \leq  11.02\lambda_k^2 +  0.72\lambda_k^2 + 1.12\lambda_k^2 \leq 12.86\lambda_k^2.
\end{equation*}
This is equivalent to  $c\lambda_{k+1} \leq (c\lambda_k)^2$, where $c := 12.86$.
By induction, we have $c\lambda_k \leq (c\lambda_0)^{2^k} \leq (0.65)^{2^k}$, which leads to $\lambda_k \leq \frac{1}{12.86}(0.65)^{2^k}$.

Given $\varepsilon \in (0, \frac{1}{20})$, to obtain $\lambda_k \leq \varepsilon$, we need to impose $2^k \geq 2.4\ln(\frac{1}{\varepsilon})-5$.
Hence, we can choose $k := \left\lfloor \log_2(\ln(\frac{1}{\varepsilon}))\right\rfloor$ as the maximum number of iterations.
Moreover, to bound $\delta$ and $\delta_2$ in terms of $\varepsilon$, we use $2^k \geq 2.4\ln(\frac{1}{\varepsilon})-5$ to show that these quantities can be bounded by $\delta := \frac{1}{10}(\frac{7}{20})^{\ln(\frac{1}{\varepsilon})}$  and $\delta_2 := 28(\frac{3}{25})^{\ln(\frac{1}{\varepsilon})}$, respectively.

The last statement of this theorem follows from the same argument as  in Theorem~\ref{th:global_convergence3} but for the convergent sequence $\sets{z^k}$.
\Eproof
\end{proof}

As a concrete example, assume that we fix the target accuracy $\varepsilon := 10^{-4}$. 
Then $\delta_2$ in Statement $\mathrm{(a)}$ becomes $\delta_2 \leq 2.08\times 10^{-5}$ and $k = 5$.
For  Statement $\mathrm{(b)}$, we have $\delta_3, \delta_4 \leq 4.5\times 10^{-5}$, $\delta_2 \leq 2.2\times 10^{-7}$, and $k = 6$.
For  Statement $\mathrm{(c)}$, we have $\delta_3, \delta_4 \leq 6.3\times 10^{-6}$, $\delta_2 \leq 5\times 10^{-8}$, and $k = 4$.
In Statement $\mathrm{(c)}$,  the accuracy $\delta_2$ is  small due to the choice $\delta_2 \leq 5\lambda_k^2$.
Practically, we can always choose $\delta_2 \leq c\lambda_k^2$ for some large $c$ to relax this accuracy.

\begin{theorem}\label{th:local_covergence2}
Let $\sets{x^k}$ be the sequence generated by \eqref{eq:prox_nt_scheme} using the damped stepsize $\alpha_k$ in \eqref{eq:step_size2}.
For $\delta_0$ in \eqref{eq:global_inexact_oracle}, $\delta_3$ in \eqref{eq:local_inexact_oracle}, and $\delta_4$ in \eqref{eq:inexact_subp}, if  $\lambda_k \leq \frac{1}{20}$, $0 \leq \delta_3, \delta_4 \leq \delta \leq \frac{1}{100}$, and $\delta_0 = \delta$, then
\begin{equation}\label{eq:th3_key_est}
\lambda_{k+1} \leq \frac{1}{10}\big(41\lambda_k^2 + 125\delta\lambda_k + 21\delta_2\big).
\end{equation}
Moreover, for any $\varepsilon \in (0, \frac{1}{20})$ and $x^0\in\dom{F}$ such that $\lambda_0 \leq \frac{1}{10}$, the following statements hold:
\begin{compactitem}
\item[$\mathrm{(a)}$]
If we choose $0 \leq \delta_3, \delta_4 \leq \frac{1}{100}$ and $\delta_2 := \frac{5\varepsilon}{32}$, then after at most $k := \left\lfloor \ln(\frac{1}{\varepsilon})\right\rfloor$ iterations, we obtain $\lambda_k \leq \varepsilon$ $($i.e., locally linear convergence rate$)$.

\vspace{1ex}
\item[$\mathrm{(b)}$]
If we choose $0 \leq \delta_3, \delta_4 \leq \frac{1}{4}(\frac{3}{5})^{\ln(\frac{1}{\varepsilon})}$ and $0 \leq \delta_2 \leq \frac{37}{20}(\frac{11}{50})^{\ln(\frac{1}{\varepsilon})}$, then after at most $k := \left\lfloor \log_{3/2}(\ln(\frac{1}{\varepsilon}))\right\rfloor$ iterations, we get $\lambda_k \leq \varepsilon$ $($i.e., locally superlinear  convergence rate$)$.

\vspace{1ex}
\item[$\mathrm{(c)}$]
If we choose $0 \leq \delta_3, \delta_4\leq \frac{1}{10}(\frac{19}{50})^{\ln(\frac{1}{\varepsilon})}$ and $0 \leq \delta_2 \leq \frac{1}{5}(\frac{13}{50})^{\ln(\frac{1}{\varepsilon})}$, then after at most $k := \left\lfloor \log_2(\ln(\frac{1}{\varepsilon}))\right\rfloor$ iterations, we have $\lambda_k \leq \varepsilon$ $($i.e., locally  quadratic convergence rate$)$.

\end{compactitem}
\vspace{1ex}
In addition, we have
\begin{equation*}
\inf_{\nabla{F}(z^k)\in\partial{F}(z^k)}\tnorm{\nabla{F}(z^k)}_{x^k}^{\ast} \leq \BigO{ \lambda_k + \sqrt{\delta_1} }.
\end{equation*}
Hence, if $\lambda_k \leq \BigO{\varepsilon} \leq \frac{1}{4(1+\delta_0)}$ and $\delta_1 \leq \BigO{\varepsilon^2} \leq 1$, then $\inf_{\nabla{F}(z^k) \in\partial{F}(z^k)}\tnorm{\nabla{F}(z^k)}_{x^k}^{\ast} \leq \BigO{\varepsilon}$.
\end{theorem}

\begin{proof}
Assume that $\lambda_k \leq \frac{1}{20}$ and $\delta := \frac{1}{100}$.
Note that since $\delta_0 = \delta$ in \eqref{eq:step_size2} and $\delta_4 \in [0, 1)$, it leads to $\alpha_k = \frac{1-\delta_4}{(1+\delta_0)\left(1 + \delta_0 + (1 - \rvnn{\delta_4})\lambda_k\right)} \in (0, 1]$.
With this choice of $\alpha_k$, we can see that the right-hand side of  \eqref{eq:key_local_estimate_com_ine} is equal to $H_2$ defined by \eqref{eq:H_func2} in Appendix~\ref{app:detail_proofs}.
For $0 \leq \delta_3, \delta_4 \leq \delta  \leq \frac{1}{100}$, by using the upper bound $\frac{1}{100}$ of $\delta_3$ and $\delta_4$ and $\frac{1}{20}$ of $\lambda_k$  into $H_2$ of Lemma~\ref{le:apdx-monotone},  we can upper bound $\lambda_{k+1}$ from \eqref{eq:key_local_estimate_com_ine} as
\begin{equation*} 
\lambda_{k+1} \leq 4.1\lambda_k^2 + 12.5\delta\lambda_k + 2.1\delta_2,
\end{equation*}
which proves \eqref{eq:th3_key_est}.

$\mathrm{(a)}$~
Using $\lambda_k \leq \frac{1}{20}$ and $\delta := \frac{1}{100}$, we first numerically overestimate \eqref{eq:th3_key_est} as $\lambda_{k+1} \leq 0.33\lambda_k + 2.1\delta_2$.
If we define $c := 0.33$, then $\lambda_{k+1} \leq c\lambda_k + 2.1\delta_2$.
Since $\lambda_0 \leq \frac{1}{20}$ and $c = 0.33$, by induction, we have $\lambda_k \leq c^k\lambda_0 + \frac{2.1\delta_2}{1-c} \leq \frac{(0.33)^k}{20} + 3.2\delta_2$.
Assume that $\delta_2 := \frac{\varepsilon}{6.4} = \frac{5\varepsilon}{32}$.
Then, from the last estimate, if we impose $\frac{(0.33)^k}{20} \leq \frac{\varepsilon}{2}$, then we have $\lambda_k \leq \varepsilon$.
This condition leads to $k \geq \ln(\frac{1}{\varepsilon}) - 2.7$.
Therefore, we can choose $k := \left\lfloor \ln(\frac{1}{\varepsilon})\right\rfloor$. 

$\mathrm{(b)}$~
We assume that we run the damped  step scheme \eqref{eq:prox_nt_scheme} for $k$ iterations.
Then, after $k$ iterations, we have $\delta \leq \frac{1}{10}\lambda_k^{\frac{1}{2}}$ and $\delta_2 \leq \frac{1}{2}\lambda_k^{\frac{3}{2}}$.
Using these inequalities, and $\lambda_k \leq \frac{1}{20}$, we can numerically overestimate \eqref{eq:main_bound_th1} as 
\begin{equation*}
\lambda_{k+1} ~\leq \left(4.1\lambda_k^{\frac{1}{2}} + \frac{12.5}{10} + \frac{2.1}{2}\right)\lambda_k^{\frac{3}{2}} \leq \left(\frac{4.1}{\sqrt{20}}  + 2.3\right)\lambda_k^{\frac{3}{2}}~ \leq  3.22\lambda_k^{\frac{3}{2}}.
\end{equation*}
This is equivalent to  $c\lambda_{k+1} \leq (c\lambda_k)^{\frac{3}{2}}$, where $c := 3.22^2$.
By induction, we have $c\lambda_k \leq (c\lambda_0)^{(\frac{3}{2})^k} \leq (0.52)^{(\frac{3}{2})^k}$, which leads to $\lambda_k \leq \frac{1}{3.22^2}(0.52)^{(\frac{3}{2})^k}$.

Given $\varepsilon \in (0, \frac{1}{20})$, to obtain $\lambda_k \leq \varepsilon$, we need to impose $(\frac{3}{2})^k \geq 1.7\ln(\frac{1}{\varepsilon})-3$.
Hence, we can choose $k := \left\lfloor \log_{\frac{3}{2}}(\ln(\frac{1}{\varepsilon}))\right\rfloor$ as the maximum number of iterations.
Moreover, to bound $\delta$ and $\delta_2$ in terms of $\varepsilon$, we use $(\frac{3}{2})^k \geq 1.7\ln(\frac{1}{\varepsilon})-3$ to show that these quantities can be bounded by 
$\delta := \frac{1}{4}(\frac{3}{5})^{\ln(\frac{1}{\varepsilon})}$, and $\delta_2 := \frac{37}{20}(\frac{11}{50})^{\ln(\frac{1}{\varepsilon})}$, respectively.

$\mathrm{(c)}$~
Similar to $\mathrm{(b)}$, we assume that we run the damped step scheme \eqref{eq:prox_nt_scheme} for $k$ iterations.
Then, after $k$ iterations, we have $\delta \leq \frac{\lambda_k}{10}$ and $\delta_2 \leq \frac{\lambda_k^2}{5}$.
Using these inequalities, and $\lambda_k \leq \frac{1}{20}$, we can numerically overestimate \eqref{eq:main_bound_th1} as 
\begin{equation*}
\lambda_{k+1} ~\leq \left(4.1 + \frac{12.5}{10} + \frac{2.1}{5}  \right)\lambda_k^2 \leq~ 5.77\lambda_k^2.
\end{equation*}
This is equivalent to  $c\lambda_{k+1} \leq (c\lambda_k)^2$, where $c := 5.77$.
By induction, we have $c\lambda_k \leq (c\lambda_0)^{2^k} \leq (0.2885)^{2^k}$, which leads to $\lambda_k \leq \frac{1}{5.77}(0.2885)^{2^k}$.

Given $\varepsilon \in (0, \frac{1}{20})$, to obtain $\lambda_k \leq \varepsilon$, we need to impose $2^k \geq 0.9\ln(\frac{1}{\varepsilon})-1.4$.
Hence, we can choose $k := \left\lfloor \log_2(\ln(\frac{1}{\varepsilon}))\right\rfloor$ as the maximum number of iterations.
Moreover, to bound $\delta$ and $\delta_2$ in terms of $\varepsilon$, we use $2^k \geq 0.9\ln(\frac{1}{\varepsilon})-1.4$ to show that these quantities can be bounded by 
$\delta := \frac{1}{10}(\frac{1}{3})^{\ln(\frac{1}{\varepsilon})}$ and $\delta_2 := \frac{1}{5}(\frac{3}{25})^{\ln(\frac{1}{\varepsilon})}$, respectively.

The last statement of this theorem follows from the same argument as in Theorem~\ref{th:global_convergence3} for the convergent sequence $\sets{z^k}$.
\Eproof
\end{proof}

As an example,  assume that we fix the target accuracy $\varepsilon = 10^{-4}$, then $\delta_2$ in Statement $\mathrm{(a)}$ becomes $\delta_2 \leq 1.56\times 10^{-5}$ and $k = 10$.
For  Statement $\mathrm{(b)}$ we have $\delta_3, \delta_4 \leq 2.3\times 10^{-3}$, $\delta_2 \leq 1.6\times 10^{-6}$, and $k = 6$.
For  Statement $\mathrm{(c)}$ we have $\delta_3, \delta_4 \leq 0.5\times 10^{-5}$, $\delta_2 \leq 10^{-9}$, and $k = 4$.
In Statement $\mathrm{(c)}$  the accuracy $\delta_2$ is too small due to the choice $\delta_2 \leq \frac{\lambda_k^2}{5}$.
Again, we can always choose $\delta_2 \leq c\lambda_k^2$ for some positive value $c$ to relax this accuracy.

\begin{remark}\label{re:subgrad_convergence}
The last statement of Theorems \ref{th:local_covergence1} and \ref{th:local_covergence2} shows the bound on the subgradient sequence $\set{\inf_{\nabla{F}(z^k)\in\partial{F}(z^k)}\tnorm{\nabla{F}(z^k)}_{x^k}^{\ast}}$ of the objective function $F$.
This sequence decreases with the same rate as of  $\set{\lambda_k + \sqrt{\delta_1}}$.
\Eproof
\end{remark}

\begin{remark}\label{re:choice}
Due to the complexity of \eqref{eq:key_local_estimate_com_ine}, we only provide one explicit range of $\delta_i$ for $i=0,\ldots,4$ and $\lambda_k$ by numerically computing their upper bounds. 
However, we can choose different values than the ones we provided in Theorems \ref{th:local_covergence1} and \ref{th:local_covergence2}.   
Moreover, we can also adapt the choice of $\delta_i$ for $i=0,\ldots,4$ over the iteration $k$ instead of fixing them at given values according to the desired accuracy.  
In particular, the statements of Theorems \ref{th:local_covergence1} and \ref{th:local_covergence2} will also be valid  for an adaptive inexact oracle framework as in Theorem \ref{th:adaptive_global_convergence} where its accuracy in the  early iterations  can be  rough and becomes finer in the last iterations. 
For example, by considering an adaptive   local inexact oracle with varying accuracies $\delta_2^k \leq \rho^{k}$ and  $\delta_3^k \leq c$,  for some appropriate  constants  $c, \rho \in (0, \,1)$, at each iteration $k\geq 0$, we can obtain  linear convergence rate, that is $\lambda_k \leq  \rho^k$. 
Similarly, by choosing appropriate adaptive values for  $ \delta_2^k$ and $ \delta_3^k$, we can get superlinear or even quadratic convergence rates.    
\Eproof
\end{remark}

\beforesubsec
\subsection{\bf Relationship to other inexact methods}\label{subsec:special_cases}
\aftersubsec
We show that our \ref{eq:prox_nt_scheme} covers both inexact Newton methods in \cite{li2017inexact,zhang2015disco} and quasi-Newton method in \cite{gao2016quasi}. 
For these special cases, where convergence bounds are known, our theory allows one to recover or to reproduce the best known rates.

\beforeparagraph
\paragraph{\bf (a)~Inexact proximal Newton methods:}
In \cite{li2017inexact} a proximal Newton method was proposed, where the inexactness lies on the subproblem of computing proximal Newton directions.
This method can be viewed as a special case of our method by choosing $\delta_0 = \delta_1 = \delta_2 = \delta_3 = 0$ (i.e., no inexact oracle was considered in \cite{li2017inexact}). 
In this case, the subproblem \eqref{eq:inexact_subp} reduces to the following one by using $\delta_4 = 1 - \theta_k$ with $\theta_k$ defined in \cite{li2017inexact}:
\begin{equation}\label{inexact_subp_1}
\nu^k \in \nabla{f}(x^k) + \nabla^2{f}(x^k)(z^k-x^k)+\partial R(z^k),
\end{equation}
where $\norms{\nu^k}_{x^k}^{\ast}\leq (1-\theta_k)\norms{z^k-x^k}_{x^k}$.
For the damped step proximal Newton method, the corresponding stepsize reduces to $\alpha_k=\frac{1-\delta_4^k}{1+(1-\delta_4^k)\lambda_k} =\frac{\theta_k}{1+\theta_k\lambda_k}$, which is the same as the stepsize defined in \cite{li2017inexact}.
For the global convergence, \cite[Theorem 3]{li2017inexact} is a special case of our  Theorem \ref{th:global_convergence3} with exact Hessian, gradient, and function values. 
Furthermore, if we let $\alpha_k=1$ in Lemma \ref{le:key_est_for_prox_nt_scheme_ine}, then we get the same local convergence result as shown in \cite[Theorem 2]{li2017inexact}.

\beforeparagraph
\paragraph{\bf (b)~Quasi-Newton methods:}
In \cite{gao2016quasi}, a quasi-Newton method for self-concordant minimization is proposed  based on a curvature-adaptive stepsize that involves both inexact and the true Hessian.
Interestingly, we can reproduce the algorithms in \cite{gao2016quasi} using the same Lipschitz gradient and strong convexity assumptions.
We can classify the steps of this quasi-Newton method into our framework and routinely reproduce the same convergence results as in \cite{gao2016quasi}.

To avoid any notation ambiguity, we express related quantities in \cite{gao2016quasi} with a superscript ``$^G$'' (e.g.,  $\alpha_k^G$ means $t_k$ in \cite[Line 5 of Algorithm 1]{gao2016quasi}) and let $B_k^{\textrm{inv}}$ be the inverse inexact Hessian $B_k$ in \cite{gao2016quasi}.
Since $f$ is self-concordant, by using  $(\tilde{f}, g, H) = (f, \nabla f, \nabla^2 f)$, we obtain a $(0,0)$-global inexact oracle as in Definition \ref{de:global_inexact_oracle}. Since \cite{gao2016quasi} only deals with the non-composite form, we have $R(x)\equiv 0$ in our setting.
Therefore, our inexact proximal Newton scheme \eqref{eq:prox_nt_scheme} reduces to the following inexact Newton scheme with exact oracle:
\begin{equation}\label{eq:nt_scheme}
\left\{\begin{array}{ll}
z^k &:\approx \displaystyle x^k - \nabla^2f(x^k)^{-1}\nabla f(x^k) \vspace{1.5ex}\\
x^{k+1} &:= (1-\alpha_k)x^k + \alpha_kz^k = x^k + \alpha_kd^k,~~\text{where}~~d^k:= z^k - x^k.
\end{array}\right.\tag{\textrm{iNA}}
\end{equation}
To rewrite the quasi-Newton method in \cite{gao2016quasi} into \eqref{eq:nt_scheme}, we exactly compute $z^k := x^k - B_k^{\textrm{inv}}\nabla f(x^k)$.
In this case, $d^k = z^k - x^k = -B_k^{\textrm{inv}}\nabla f(x^k)$ is exactly the descent direction $d_k^G$ in  \cite{gao2016quasi}. 
Therefore, by choosing $\nu^k := \nabla f(x^k)+\nabla^2f(x^k)d^k$ in \eqref{eq:nt_scheme}, we can rewrite the quasi-Newton method in \cite{gao2016quasi} into \eqref{eq:nt_scheme}.
That is $\tnorm{\nabla f(x^k)+\nabla^2f(x^k)d^k}_{x^k}^{*} \leq \delta_4^k\tnorm{d^k}_{x^k}$.

Moreover, when $d^k$ is a descent direction, we have $\iprods{\nabla f(x^k),d^k} < 0$.
Hence, it implies that
\begin{equation*}
\delta_4^k := 1 - \alpha_k^G = 1 - \frac{\iprods{\nabla f(x^k),-d^k}}{\norms{d^k}_{x^k}^2} = 1 + \frac{\iprods{\nabla f(x^k),d^k}}{\norms{d^k}_{x^k}^2} = 1 + \frac{\iprods{\nabla f(x^k),d^k}}{\lambda_k^2} < 1.
\end{equation*}
The condition \eqref{eq:inexact_subp} becomes $\tnorm{\nu^k}^{*}_{x^k} = \tnorm{\nabla f(x^k)+\nabla^2f(x^k)d^k}_{x^k}^{*} \leq \left(\lambda_k + \frac{\iprods{\nabla f(x^k),d^k}}{\lambda_k}\right) = \delta_4^k\lambda_k$.

Now, using \eqref{eq:global_inexact_oracle} with $y := x^{k+1}$, $x := x^k$, and $\delta_1 = \delta_0 = 0$, we can show that
\begin{equation}\label{nt:desc_proof}
f(x^{k+1})  \leq f(x^k) + \alpha_k\iprods{\nabla f(x^k),d^k} + \omega_{\ast}(\alpha_k\lambda_k).
\end{equation}
Minimizing the right-hand side of this inequality over $\alpha_k\in [0,1]$, we obtain the optimal $\alpha_k$ as:
\begin{equation*}
\alpha_k = \frac{-\iprods{\nabla f(x^k),d^k}}{\lambda_k(\lambda_k-\iprods{\nabla f(x^k),d^k})} 
= \frac{-\iprods{\nabla f(x^k),d^k}/\lambda_k^2}{\lambda_k(\lambda_k-\iprods{\nabla f(x^k),d^k})/\lambda_k^2} 
= \frac{1-\delta_4^k}{1+(1-\delta_4^k)\lambda_k}.
\end{equation*}
Substituting this $\alpha_k$ into \eqref{nt:desc_proof} we obtain
\begin{equation}\label{nt:desc}
f(x^{k+1}) \leq f(x^k) - \omega((1-\delta_4^k)\lambda_k).
\end{equation}
Using this estimate and the above stepsize $\alpha_k$, the conclusion of Theorem \ref{th:global_convergence3} still holds.

To see the similarity between \cite{gao2016quasi} and \eqref{eq:nt_scheme}, we rearrange  our stepsize $\alpha_k$ above to get
\begin{equation*}
\alpha_k = \frac{1-\delta_4^k}{1+(1-\delta_4^k)\lambda_k} = \frac{\alpha_k^G}{1+\alpha_k^G\delta_k^G} = t_k^G,
\end{equation*}
which is exactly the stepsize $t_k^G$ used in \cite{gao2016quasi}.
For the descent property, the conclusion in \cite[Lemma 4.1]{gao2016quasi} is $f(x^{k+1}) \leq f(x^k) - \omega(\eta_k^G)$.
Comparing this with \eqref{nt:desc}, we have
\begin{equation*}
(1-\delta_4^k)\lambda_k = \frac{\iprods{\nabla f(x^k),-d^k}}{\norms{d^k}_{x^k}^2}\cdot\norms{d^k}_{x^k} = \frac{\iprods{\nabla f(x^k),B_k^{-1}\nabla f(x^k)}}{\lambda_k} = \frac{\rho_k^G}{\delta_k^G}=\eta_k^G.
\end{equation*}
\noindent 
Therefore, we have reproduced the main result of \cite[Section 4]{gao2016quasi} by using our framework and Theorem \ref{th:global_convergence3}.
Furthermore, \cite[Section 5]{gao2016quasi} analyzes the convergence behavior of the quasi-Newton method but relying on the condition $\underline{\lambda} \Id \preceq B_k^{\textrm{inv}} \preceq \bar{\lambda}\Id$ for either $\underline{\lambda} = \bar{\lambda} = 1$ (gradient descent) or for $\underline{\lambda}$ and $\bar{\lambda}$ being chosen as in \cite[Theorem 5.5]{gao2016quasi} (L-BFGS).
We emphasize that \cite[Section 6]{gao2016quasi} derives similar results for $B_k^{\textrm{inv}}$ based on BFGS updates.
Since  \cite[Sections 5 and 6]{gao2016quasi} are just two particular choices for $B_k^{\textrm{inv}}$ based on the scheme of \cite[Section 4]{gao2016quasi}, from our previous discussion, it follows immediately that we can reproduce   the local and global convergence results in \cite{gao2016quasi} under the Lipschitz gradient and strong convexity assumptions as considered in \cite{gao2016quasi}.

\beforesec
\section{Application to Primal-Dual Methods}\label{sec:primal_dual_method}
\aftersec
We have shown in Subsection \ref{subsec:Fen_conj} that inexact oracles of a convex function can be controlled by approximately evaluating its Fenchel conjugate.
In this section, we show how to apply this theory to design an inexact primal-dual method for solving composite minimization problem of a self-concordant objective term and a nonsmooth convex regularizer.

We consider the following composite convex problem:
\begin{equation}\label{eq:prob_main}
G^{\star}:=\min_{y\in\R^n}\Big\{ G(y) := \varphi(A^{\top}y) + \psi(y) \Big\},
\end{equation}
where $\varphi : \R^p \to \R\cup\set{+\infty}$ is proper, closed, and convex, and $\psi : \R^n\to\R\cup\set{+\infty}$ is a smooth convex function. 
We assume that $\psi$ is self-concordant, $\varphi$ is proximally tractable, and  $A\in\R^{n\times p}$ is not diagonal.
Problem \eqref{eq:prob_main} covers many applications in the literature such as image denoising and restoration \cite{Chambolle2011}, sparse inverse covariance estimation \cite{Friedman2008}, distance weighted discrimination \cite{marron2007distance}, robust PCA, and fused lasso problems. 

Since $\varphi$ is nonsmooth, and $A$ is not diagonal, the proximal operator of $\varphi(A^{\top}(\cdot))$ is not proximally tractable in general.
We instead consider the dual problem of \eqref{eq:prob_main}.
Using Fenchel conjugate, the dual problem of \eqref{eq:prob_main} can be written as
\begin{equation}\label{eq:prob_dual_ori2}
F^{\star} := \min_{x\in\R^p}\Big\{ F(x) := f(x) + R(x) \equiv \psi^{\ast}(Ax) + \varphi^{\ast}(-x) \Big\},
\end{equation}
which is exactly of the form \eqref{eq:c_sc_min}, where $f(x) := \psi^{\ast}(Ax)$ and $R(x) := \varphi^{\ast}(-x)$.
For our primal-dual framework to be well-defined, we impose the following assumption:
\begin{assumption}\label{as:A4}
The primal and dual solution sets of \eqref{eq:prob_main} and \eqref{eq:prob_dual_ori2} are nonempty.
Moreover, the strong duality between \eqref{eq:prob_main} and \eqref{eq:prob_dual_ori2} holds, i.e., $G^{\star} + F^{\star} = 0$.
\end{assumption}
Assumption~\ref{as:A4} is standard in primal-dual methods for convex optimization and can be stated in different ways, see, e.g.,  \cite{Chambolle2011,Necoara2009,TranDinh2016c}.
The optimality condition of \eqref{eq:prob_main} and \eqref{eq:prob_dual_ori2} becomes
\begin{equation}\label{eq:opt_cond3}
A x^{\star} = \nabla{\psi}(y^{\star})~~\text{and}~~ -x^{\star} \in \partial{\varphi}(A^{\top}y^{\star}) ~\Leftrightarrow~ 0 \in  -A^{\top}y^{\star} + \partial{\varphi}^{\ast}(-x^{\star}).
\end{equation}
Let $y^{\ast}(x) \in \displaystyle\mathrm{arg}{\!\!\!\!\!\!\!}\max_{y\in\dom{\psi}}\set{\iprods{x, A^{\top}y} - \psi(y)}$. 
Since the optimal set of \eqref{eq:prob_main} is nonempty and $\psi$ is self-concordant, $y^{\ast}(x)$ exists and is unique under mild conditions \cite[Theorem 4.1.11]{Nesterov2004}.

Moreover, we can show that the exact gradient and Hessian mappings of $f$ are $\nabla{f}(x) = A^{\top}y^{\ast}(x)$ and $\nabla^2{f}(x) = A^{\top}\nabla^2{\psi}(y^{\ast}(x))^{-1}A$, respectively.
However, in practice, we can only evaluate an inexact oracle of $f$ as
\begin{equation}\label{g,H:pd}
g(x) := A^{\top}\tilde{y}^{\ast}(x),~~~\text{and}~~~H(x) := A^{\top}\nabla^2{\psi}(\tilde{y}^{\ast}(x))^{-1}A,
\end{equation}
that approximate $\nabla{f}(x)$ and $\nabla^2{f}(x)$, respectively, 
where $\tilde{y}^{\ast}(x)$ is an approximate solution of $y^{\ast}(x)$ such that $\norms{Ax - \nabla{\psi}(\tilde{y}^{\ast}(x))}_{\tilde{y}^{\ast}(x)} \leq  \frac{\delta}{1 + \delta}$ as suggested by Lemma \ref{le:est_A}.
Here, the local dual norm $\norms{\cdot}_{y}^{\ast}$ is defined as $\norms{u}_{y}^{\ast} := \big(u^{\top}\nabla^2{\psi}(y)^{-1}u\big)^{1/2}$ based on the self-concordant function $\psi$.

Now, we can develop an inexact primal-dual method to solve \eqref{eq:prob_main} as follows.
Starting from an initial point $x^0\in\dom{F}$, at each iteration $k\geq 0$, we perform the following steps:
\begin{enumerate}
\item Approximately compute $\tilde{y}^{\ast}(x^k)$ such that $\norms{Ax^k - \nabla{\psi}(\tilde{y}^{\ast}(x^k))}_{\tilde{y}^{\ast}(x^k)}^{\ast} \leq \frac{\delta}{1 + \delta}$,  where $\delta$ is chosen according to Lemma \ref{le:est_A} and Theorem \ref{th:global_convergence3}.
\item Form an inexact oracle $g(x^k) := A^{\top}\tilde{y}^{\ast}(x^k)$ and $H(x^k) := A^{\top}\nabla^2{\psi}(\tilde{y}^{\ast}(x^k))^{-1}A$ of $f$ at $x^k$.
\item Approximately solve $z^k\approx \bar{z}^k:=\arg\min\big\{ \mathcal{Q}(x;x^k) + R(x) \big\}$ as in \eqref{eq:prox_nt_scheme}.
\item Compute a stepsize $\alpha_k$ as in \eqref{eq:step_size2}.
\item Update $x^{k+1} := (1 -  \alpha_k)x^k +  \alpha_kz^k$.
\end{enumerate}
Finally, we recover an approximate solution $y^k := \tilde{y}^{\ast}(x^k)$ of $y^{\star}$ for \eqref{eq:prob_main}.

The following theorem shows that $y^k$ is indeed an approximate solution of \eqref{eq:prob_main}.

\begin{theorem}\label{th:approx_sol}
Let $\sets{(z^k, y^k)}$ be the sequence generated by our primal-dual scheme above.
Then
\begin{equation}\label{eq:primal_dual_aprox_sol}
\Vert Az^k - \nabla{\psi}(y^k)\Vert_{y^k}^{\ast} \leq \frac{\delta}{1 + \delta} + \lambda_k~~~\text{and}~~r^k \in A^{\top}y^k  -\partial \varphi^{\ast}(-z^k) ~\text{with}~\tnorm{r^k}^{\ast}_{x^k} \leq (1+\delta_4)\lambda_k.
\end{equation}
Consequently, if we compute $\lambda_k$ and choose $\delta$ such that $\delta   \leq \frac{\varepsilon}{2 - \varepsilon}$ and $\lambda_k \leq \frac{\varepsilon}{2}$, then
$(z^k, y^k)$ is an $\varepsilon$-solution of the primal problem \eqref{eq:prob_main} and its dual \eqref{eq:prob_dual_ori2}, i.e., $\Vert Az^k - \nabla{\psi}(y^k)\Vert^{\ast}_{y^k} \leq \varepsilon$ and $\tnorm{r^k}^{\ast}_{x^k} \leq \varepsilon$ such that $r^k \in  A^{\top}y^k  -\partial \varphi^{\ast}(-z^k)$, which approximate \eqref{eq:opt_cond3}.
\end{theorem}

\begin{proof}
Since we define $y^k:=\tilde{y}^{\ast}(x^k)$, from \eqref{eq:prox_nt_scheme}, \eqref{eq:inexact_subp}, and \eqref{g,H:pd}, we have 
\begin{equation*}
\nu^k\in A^{\top}y^k + A^{\top}\nabla^2{\psi}(y^k)^{-1}A(z^k-x^k)-\partial \varphi^{\ast}(-z^k).
\end{equation*} 
Let us define $r^k := \nu^k - A^{\top}\nabla^2{\psi}(y^k)^{-1}A(z^k-x^k)$.
Then, the last condition  leads to $r^k  \in A^{\top}y^k  -\partial \varphi^{\ast}(-z^k)$.
Hence, we can estimate $\tnorm{r^k}_{x^k}$ as follows:
\begin{equation*}
\tnorm{r^k}_{x^k}^{\ast} \overset{\tiny{(i)}}{\leq} \tnorm{\nu^k}_{x^k}^{\ast} + \tnorm{A^{\top}\nabla^2{\psi}(y^k)^{-1}A(z^k-x^k)}_{x^k}^{\ast}  \overset{\tiny{(ii)}}{\leq} \delta_4^k\lambda_k + \lambda_k = (1 + \delta_4^k)\lambda_k.
\end{equation*}
Here, we use the triangle inequality in (i), and $H(x^k) = A^{\top}\nabla^2{\psi}(y^k)^{-1}A$, $d^k := z^k - x^k$, and $\lambda_k$ defined by \eqref{eq:step_size2} into (ii).
Therefore, we get the second part of \eqref{eq:primal_dual_aprox_sol}.

We also have $\big[\Vert A(z^k - x^k)\Vert_{y^k}^{*}\big]^2 = (z^k-x^k)^{\top}A\nabla^2{\psi}(y^k)^{-1}A(z^k - x^k)  = (z^k - x^k)^{\top}H(x^k)(z^k - x^k)  = \tnorm{z^k - x^k}_{x^k}^2$, which implies $\Vert A(z^k - x^k)\Vert_{y^k}^{*} = \tnorm{z^k - x^k}_{x^k}$.
Using this relation and $\norms{Ax^k - \nabla{\psi}(y^k)}_{y^k}^{\ast} \leq \frac{\delta_k}{1 + \delta_k}$ we can show that $\norms{Az^k - \nabla{\psi}(y^k)}_{y^k}^{\ast} \leq \frac{\delta_k}{1 + \delta_k} + \norms{A(z^k - x^k)}_{y^k}^{\ast} = \frac{\delta_k}{1 + \delta_k} + \tnorm{z^k - x^k}_{x^k} = \frac{\delta_k}{1 + \delta_k} + \lambda_k$, which proves the first part of \eqref{eq:primal_dual_aprox_sol}.

Finally, since $\delta_4 \in [0,1]$, if we impose $\lambda_k \leq \frac{\varepsilon}{2}$, then we have $\tnorm{r^k}^{\ast}_{x^k} \leq 2\lambda_k \leq \varepsilon$.
Moreover, if $\delta   \leq \frac{\varepsilon}{2 - \varepsilon}$, then since $\lambda_k \leq \frac{\varepsilon}{2}$, from \eqref{eq:primal_dual_aprox_sol} we have $\Vert Az^k - \nabla{\psi}(y^k)\Vert_{y^k}^{\ast} \leq \frac{\delta}{1+\delta} + \frac{\varepsilon}{2} \leq \varepsilon$.
\Eproof
\end{proof}

By Theorems \ref{th:local_covergence1} and \ref{th:local_covergence2}, we can also prove locally linear, superlinear, and quadratic convergence rates of the two residual sequences $\big\{\Vert Az^k - \nabla{\psi}(y^k)\Vert^{\ast}_{y^k}\big\}$ and $\big\{\tnorm{r^k}^{\ast}_{x^k}\big\}$. 
However, we skip the details to avoid overloading the paper.

\beforesec
\section{Preliminary Numerical Experiments}\label{sec:Numerical}
\aftersec
We provide two numerical examples to verify several aspects of our theoretical results and also compare our algorithms with some state-of-the-art methods. 
We have implemented the proposed algorithms and competing methods (if needed) in Matlab 2018a running on a Lenovo Thinkpad 2.60GHz Intel Core i7 Laptop with 8Gb memory.

\beforesec
\subsection{\bf Composite convex optimization models involving log-barriers}
\aftersec
This example aims at illustrating several theoretical aspects of our theory developed in the previous sections. For this purpose, we consider the following composite convex minimization model involving a log-barrier as a special case of \eqref{eq:prob_main}:
\begin{equation}\label{ex_2:l1tv}
G^{\star} := \min_{y\in\R^p}\Big\{ G(y) := \varphi(A^{\top}y) + \psi(y) \Big\},
\end{equation}
where $\varphi : \R^n\to\Rext$ is a proper, closed, and convex function, $\psi(y) := -\sum_{i=1}^m w_i \ln(d_i - c_i^{\top}y)$, which can be viewed as a barrier function of a polyhedron $\Pc := \big\{ y \in \R^p \mid C^{\top}y \leq d\big\}$, $A\in\R^{p\times n}$, and $w\in\R^m_{+}$ is a given weight vector.
We emphasize that using log-barriers allows us to handle ``hard" constraints (i.e., the constraints that can not be violated).

In our experiments, we focus on the case $\varphi$ is a finite sum of $\ell_p$-norms. Problem \eqref{ex_2:l1tv} has concrete applications, including solving systems of linear equations and inequations, Poisson image processing \cite{Harmany2012,Lefkimmiatis2013}, and robust optimization \cite{Ben-Tal2009}.

Unlike several existing models, the linear operator $A$ in \eqref{ex_2:l1tv} is composited into a nonsmooth term $\varphi$ making proximal gradient-type methods to be inefficient due to the expensive proximal operator of $\varphi(A^{\top}(\cdot))$.
Instead of solving the primal problem \eqref{ex_2:l1tv} directly, we consider its dual formulation as in Section \ref{sec:primal_dual_method}:
\begin{equation}\label{ex_2:l1tvd}
F^{\star} := \min_{x}\Big\{ F(x) := \varphi^{\ast}(-x) + \psi^{\ast}(Ax)\Big\},
\end{equation}
where $\varphi^{\ast}$ and $\psi^{\ast}$ are the Fenchel conjugates of $\varphi$ and $\psi$, respectively.
Clearly, since $\psi$ is smooth, one can evaluate its conjugate $\psi^{\ast}$ as well as the derivatives of $\psi^{\ast}$ by solving
\begin{equation}\label{ex_2:fconj}
\psi^{\ast}(Ax) := \max_{u \in\R^n } \Big\{ h(u) :=  \iprod{Ax, u} + \sum_{i=1}^mw_i\ln(d_i - c_i^{\top}u) \Big\}.
\end{equation}
Let us denote by $u^{\ast}(x)$ the solution of this problem.
Since the underlying function is self-concordant, one can apply a Newton method to compute $u^{\ast}(x)$ \cite{Nesterov2004}.
However, we can only approximately compute $u^{\ast}(x)$, which leads to inexact oracle of $\psi^{\ast}$.
Hence, our theory, in particular the results developed in Section \ref{sec:primal_dual_method}, can be applied to solve \eqref{ex_2:fconj}.

\beforesubsubsec
\subsubsection{\textbf{Application to network allocation problems}}\label{eg:Network}
\aftersubsubsec
We first apply the composite setting \eqref{ex_2:l1tv} to model the following network allocation problem. 

Assume that we have $K$ cities described by polytopes as their possible area:
\begin{equation*}
\Pc_{[i]} := \set{ y \in\R^p \mid C^{[i]}y \leq d^{[i]}}~~~\text{for}~~i=1,\cdots, K.
\end{equation*}
These cities are connected by a delivery network describing the routes between each pair of cities.
Our goal is to locate a delivery center $y^{[i]} \in \Pc_{[i]}$ such that the total distance (or the total delivery costs) between these cities is minimized. 

In order to accurately guarantee $y^{[i]} \in \Pc_{[i]}$, we use a log-barrier function to handle this hard constraint.
Therefore, one way to model this problem is as in \eqref{ex_2:l1tv}, where
\begin{equation*}
\varphi(Ay) := \mu\sum_{(i,j) \in \Ec}c_{ij}\norms{y^{[i]} - y^{[j]}}_2=\mu\sum_{(i,j)\in\Ec} c_{ij}\sqrt{(y^{[i]}_1 - y^{[j]}_1)^2 + (y^{[i]}_2 - y^{[j]}_2)^2},
\end{equation*}
where $c_{ij} \geq 0$ is the cost that is proportional to the distance between the $i$-th and the $j$-th city, and $\Ec$ is the set of edges of the graph describing this network, $\mu > 0$ is a penalty parameter in the barrier formulation \eqref{ex_2:l1tv}, and $A$ is a matrix describing the difference operator. 

\beforeparagraph
\paragraph{\textbf{$(a)$~Geometric illustrations}:}
We illustrate our algorithm for solving this model by creating a UNC shape as a toy example and finding optimal site allocation solution of a US network.
Here, the network shape is downloaded from \href{http://esciencecommons.blogspot.com/2015/06/how-flu-viruses-use-transportation.html}{http://esciencecommons.blogspot.com/2015/06/how-flu-viruses-use-transportation.html}.

Let us focus on minimizing the total distance between cities in this experiment.
Hence, the cost $c_{ij} = 1$ for all $(i,j) \in\Ec$.
We implement our \ref{eq:prox_nt_scheme} to solve \eqref{ex_2:l1tvd} using these two network configurations to generate input data.
The results are visualized in Figure~\ref{fig:UNC_and_STOR}.
\begin{figure}[htp!]
\begin{center}
\includegraphics[width=0.9\linewidth]{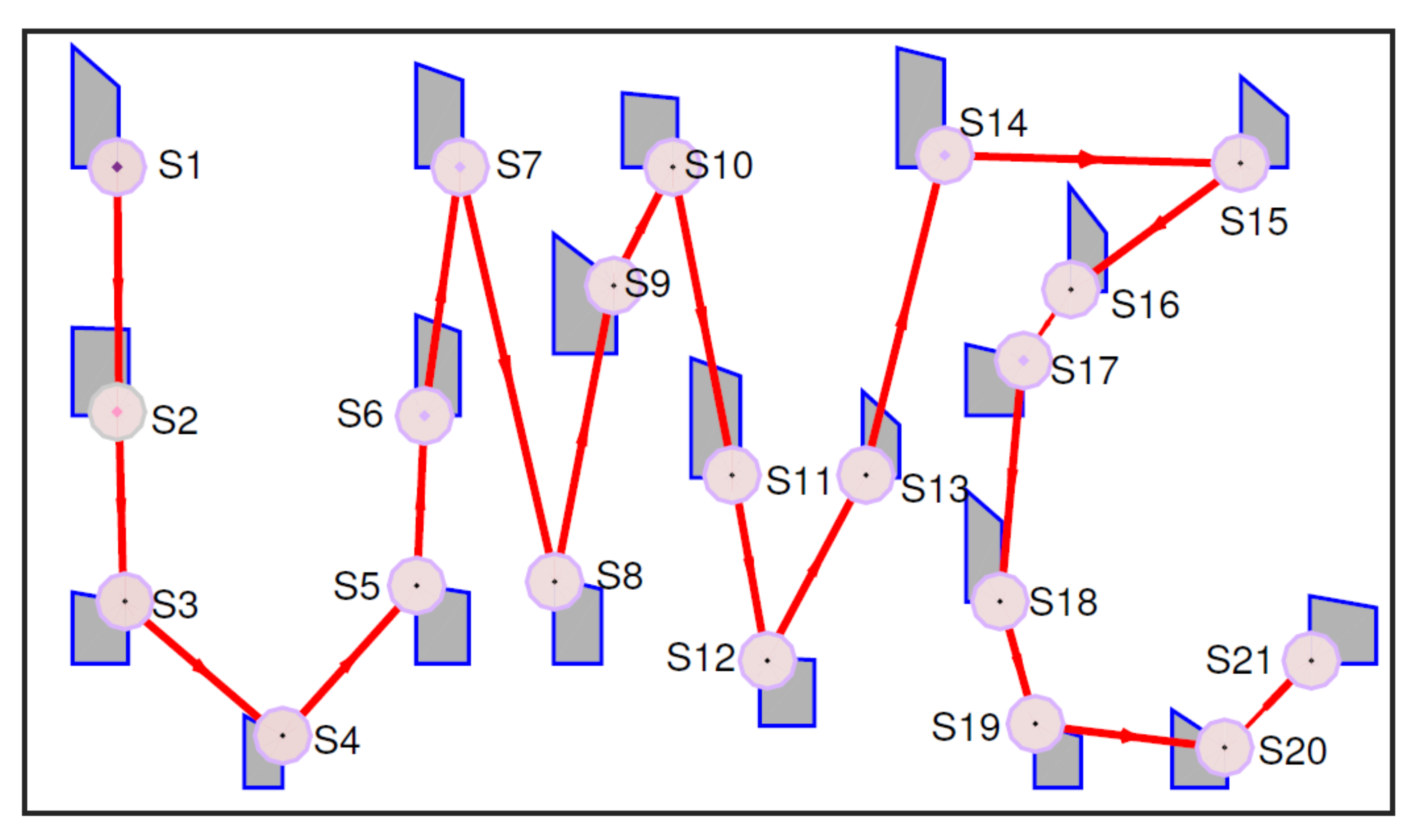}
\includegraphics[width=0.9\linewidth]{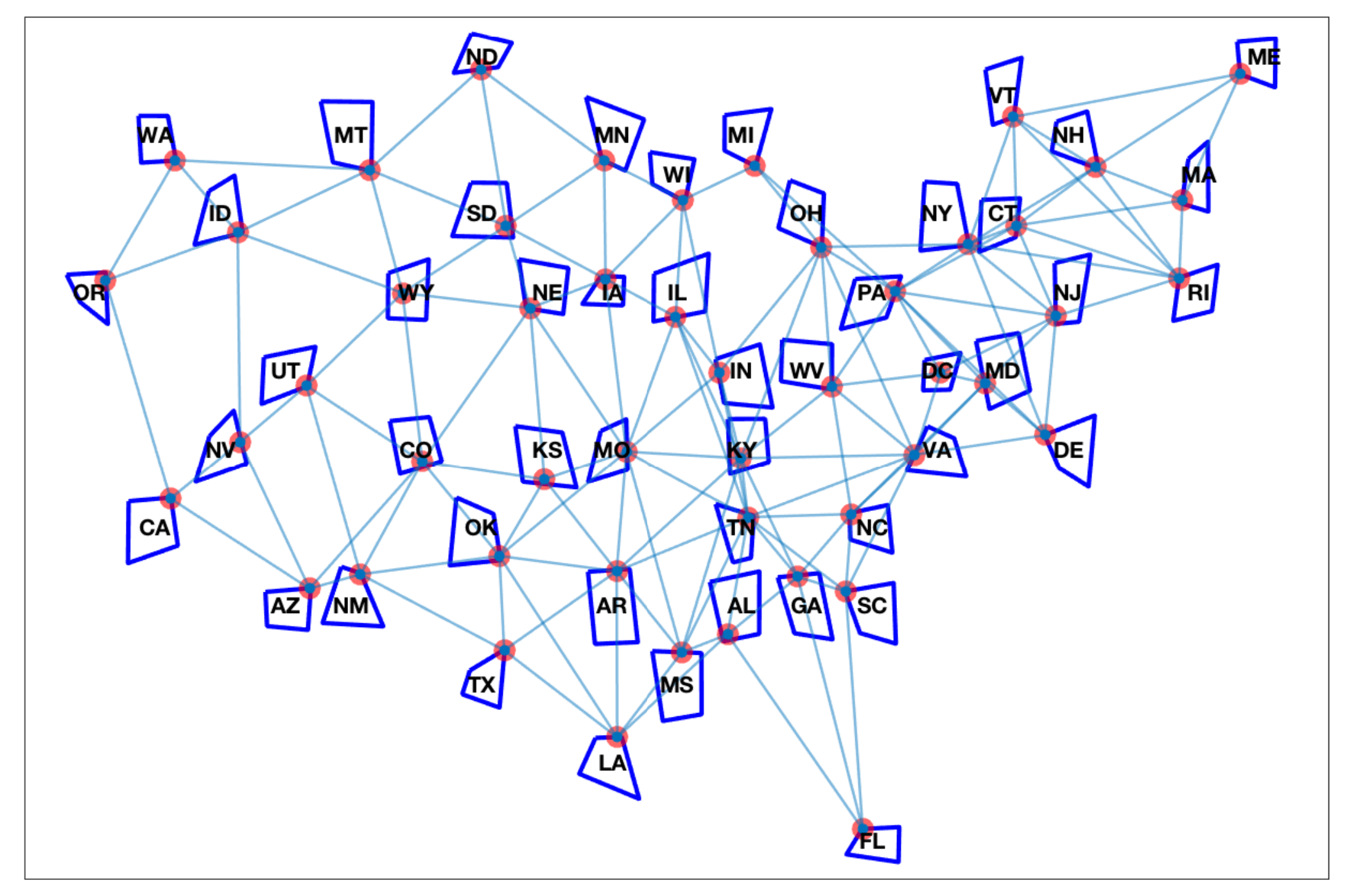} 
\vspace{-2ex}
\caption{Optimal site allocation for routes using a UNC shape (top) and  a US network (bottom).}\label{fig:UNC_and_STOR}
\vspace{-7ex}
\end{center}
\end{figure}
It is clear from this figure that our method automatically selects the delivery locations to minimize the total of distances between cities.

\beforeparagraph
\paragraph{\textbf{$(b)$~Comparison to other methods:}}
Next, we test our method on a collection of problems generated synthetically. 
The data is generated based on the geometric illustration above and is described as follows.

We simulate the data by generating $17$ problems with sparse network (the density $\rho=0.04$) and $13$ problems with dense network (the density $\rho=0.15$). 
For problems of size $2p$, we generate an $l$-by-$n$ rectangle area with $l=10$ and $n = \lfloor p/5\rfloor$ in our case with each area of size $10\times 10$ units. 
We randomly select $p$ positions from the $2p$ square. 
For each chosen position $i$ with the center point being the origin, we again randomly generates one point as a vertex in each quadrant of the square, 
and then link them together as the feasible region of site $i$, where $i=1,2,\cdots,p$, and the matrix $C$ and vector $d$ are generated from all feasible regions. 
We also generate a random adjacency matrix of size $p\times p$ with density $\rho=0.04$ (sparse) and $0.15$ (dense) as the network, which corresponds to the linear operator $A$ in the model setting \eqref{ex_2:fconj}. 

In our test, we choose $\mu=10$, which is appropriate to guarantee that the optimal points are sufficiently close to the boundary of the feasible regions. 
We again choose $c_{ij} = 1$ for all $(i,j) \in \Ec$. 
One can also use different $c_{ij}$ in order to reflect different practical situations.

We solve this problem using \ref{eq:prox_nt_scheme} as before. 
Since the problem shares a sparse structure of matrix $A$, we set the tolerance of the main loop to be $\texttt{tol}_{gap}:=10^{-10}$ and $\texttt{tol}_{sol}:=10^{-8}$, which measure the relative primal-dual gap defined by $r_{gap}:=\frac{\abs{F(x^k) +G(y^k)}}{1+\abs{F(x^k)}+\abs{G(y^k)}}$ and the maximum relative solution difference of the primal and dual solutions defined by
\begin{equation*}
r_{sol}:=\max\left\{\frac{\norms{x^{k+1}-x^k}_2}{\max\set{1,\norms{x^k}_2}},\frac{\norms{y^{k+1}-y^k}_2}{\max\set{1,\norms{y^k}_2}}\right\}.
\end{equation*}
We terminate our algorithm when both (1)~$r_{gap} \leq \texttt{tol}_{gap}$ and (2) $r_{sol} \leq \texttt{tol}_{sol}$ hold.

In this test, we show the advantages of our \ref{eq:prox_nt_scheme} to  state-of-the-art solvers such as SDPT3: a well-established interior-point solver to solve \eqref{ex_2:l1tv} \cite{Toh2010}, 
ADMM: the alternating direction method of multipliers \cite{Boyd2011}, and CP: Chambolle-Pock's primal-dual first-order algorithm \cite{Chambolle2011}.
We terminate all methods when both $\texttt{tol}_{gap}$ and $\texttt{tol}_{sol}$ are met.

For the first-order methods ADMM and CP, we lower $\texttt{tol}_{sol}$ to $10^{-6}$ instead of $10^{-8}$ in our implementation since these methods rarely achieve such a high accuracy. 
We run CP for $10,000$ iterations to get a solution with a very high accuracy as the ground truth, and compare the relative difference from the primal solution of all algorithms to the ground truth, which is denoted by $\texttt{qsol}$. 
Since there is no convergence rate guarantee at the first phase of \ref{eq:prox_nt_scheme} algorithm, we use ``$t_0/t$'' to represent the ratio between the number of iterations starting from $x^0$ until reaching the local quadratic convergence region (measured by $\lambda_k \leq 0.1$, where we start to apply Theorem \ref{th:local_covergence1}, \ref{th:local_covergence2}) and the total number of iterations. 
The results are listed in Table \ref{tbl:ex3}.

\begin{scriptsize}
\vspace{-4ex}
\begin{table}[hpt!]
\newcommand{\cell}[1]{{\!\!}#1{\!\!}}
\newcommand{\cellf}[1]{{\!\!\!}#1{\!\!}}
\newcommand{\cellbf}[1]{{\!\!}{#1}{\!\!\!}}
\begin{center}
\caption{The performance of $4$ algorithms for $\ell_{1,2}$-log barrier over $30$ problem instances.}\label{tbl:ex3}
\vspace{-1ex}
\begin{tabular}{ l | rrr | rrr | rrr | rrr}
\toprule
\multicolumn{1}{c|}{\cell{Problem}} & \multicolumn{3}{c|}{\cell{\ref{eq:prox_nt_scheme}}} & \multicolumn{3}{c|}{\cell{SDPT3}} & \multicolumn{3}{c|}{\cell{ADMM}} & \multicolumn{3}{c}{\cell{Chambolle-Pock}}\\  \midrule
\cell{Name} & \cell{$t_0/t$} & \cell{\!\texttt{t[s]}\!} &\cell{\texttt{qsol}} & \cell{\!\texttt{iter}\!} & \cell{\!\texttt{t[s]}} & \cell{\texttt{qsol}} & \cell{\!\texttt{iter}\!}  & \cell{\!\texttt{t[s]}} & \cell{\texttt{qsol}}  & \cell{\!\texttt{iter}\!} & \cell{\!\texttt{t[s]}} & \cell{\texttt{qsol}}\\
\midrule
\multicolumn{13}{c}{Sparse networks} \\ \midrule
\cell{p004120} & \cell{16/72} & \cell{2.8} & \cell{9.0e-06} & \cell{22} & \cell{2.4} & \cell{3.0e-05} & \cell{207} & \cell{0.4} & \cell{9.3e-07} & \cell{644} & \cell{2.4} & \cell{6.9e-05} \\
\cell{p004160} & \cell{16/79} & \cell{3.1} & \cell{4.4e-05} & \cell{28} & \cell{4.9} & \cell{5.8e-06} & \cell{253} & \cell{0.7} & \cell{1.2e-06} & \cell{681} & \cell{3.5} & \cell{3.7e-04} \\
\cell{p004200} & \cell{16/91} & \cell{6.0} & \cell{4.3e-05} & \cell{31} & \cell{8.0} & \cell{9.9e-06} & \cell{329} & \cell{1.3} & \cell{7.6e-07} & \cell{701} & \cell{5.5} & \cell{3.1e-04} \\
\cell{p004240} & \cell{17/98} & \cell{6.9} & \cell{1.1e-05} & \cell{29} & \cell{10.6} & \cell{5.9e-06} & \cell{336} & \cell{5.2} & \cell{1.4e-06} & \cell{789} & \cell{9.1} & \cell{9.6e-05} \\
\cell{p004280} & \cell{16/105} & \cell{8.7} & \cell{8.2e-05} & \cell{34} & \cell{18.6} & \cell{5.1e-06} & \cell{397} & \cell{13.9} & \cell{3.3e-06} & \cell{776} & \cell{12.0} & \cell{2.1e-04} \\
\cell{p004320} & \cell{18/114} & \cell{9.0} & \cell{1.4e-05} & \cell{34} & \cell{21.5} & \cell{6.5e-06} & \cell{375} & \cell{17.6} & \cell{1.7e-06} & \cell{733} & \cell{14.4} & \cell{7.4e-05} \\
\cell{p004360} & \cell{16/118} & \cell{10.0} & \cell{4.1e-05} & \cell{36} & \cell{32.4} & \cell{3.8e-06} & \cell{308} & \cell{20.6} & \cell{1.6e-06} & \cell{813} & \cell{21.9} & \cell{1.5e-04} \\
\cell{p004400} & \cell{18/131} & \cell{20.2} & \cell{2.1e-05} & \cell{41} & \cell{50.9} & \cell{2.9e-06} & \cell{677} & \cell{64.9} & \cell{4.2e-06} & \cell{866} & \cell{30.0} & \cell{6.3e-05} \\
\cell{p004440} & \cell{18/132} & \cell{18.7} & \cell{7.7e-05} & \cell{35} & \cell{60.4} & \cell{5.2e-06} & \cell{524} & \cell{59.7} & \cell{2.6e-06} & \cell{843} & \cell{39.5} & \cell{1.4e-04} \\
\cell{p004480} & \cell{20/146} & \cell{26.1} & \cell{1.5e-05} & \cell{42} & \cell{103.8} & \cell{1.7e-06} & \cell{584} & \cell{84.8} & \cell{5.9e-07} & \cell{790} & \cell{60.7} & \cell{9.5e-05} \\
\cell{p004520} & \cell{17/146} & \cell{29.3} & \cell{3.1e-05} & \cell{34} & \cell{99.1} & \cell{3.2e-06} & \cell{577} & \cell{102.7} & \cell{2.0e-06} & \cell{848} & \cell{96.4} & \cell{1.6e-04} \\
\cell{p004560} & \cell{17/150} & \cell{29.2} & \cell{2.6e-05} & \cell{31} & \cell{98.5} & \cell{4.2e-06} & \cell{447} & \cell{89.9} & \cell{6.7e-07} & \cell{815} & \cell{127.1} & \cell{1.2e-04} \\
\cell{p004600} & \cell{20/158} & \cell{42.5} & \cell{3.6e-05} & \cell{37} & \cell{264.6} & \cell{2.4e-06} & \cell{564} & \cell{141.3} & \cell{2.0e-06} & \cell{974} & \cell{197.3} & \cell{1.3e-04} \\
\cell{p004640} & \cell{18/172} & \cell{54.0} & \cell{2.8e-05} & \cell{36} & \cell{317.5} & \cell{2.4e-06} & \cell{649} & \cell{184.3} & \cell{1.1e-06} & \cell{889} & \cell{197.4} & \cell{9.7e-05} \\
\cell{p004680} & \cell{19/172} & \cell{61.2} & \cell{3.4e-05} & \cell{34} & \cell{380.9} & \cell{1.8e-06} & \cell{688} & \cell{230.6} & \cell{1.0e-06} & \cell{1042} & \cell{267.9} & \cell{9.2e-05} \\
\cell{p004720} & \cell{17/177} & \cell{68.5} & \cell{1.4e-05} & \cell{38} & \cell{539.5} & \cell{2.8e-06} & \cell{659} & \cell{269.0} & \cell{4.4e-07} & \cell{844} & \cell{290.1} & \cell{7.0e-05} \\
\cell{p004760} & \cell{20/190} & \cell{84.5} & \cell{3.7e-05} & \cell{40} & \cell{742.7} & \cell{1.5e-06} & \cell{780} & \cell{374.2} & \cell{7.4e-06} & \cell{1311} & \cell{1544.9} & \cell{8.6e-05} \\
\midrule
\multicolumn{13}{c}{Dense networks} \\ \midrule
\cell{p01580} & \cell{17/75} & \cell{1.7} & \cell{2.7e-05} & \cell{20} & \cell{3.4} & \cell{1.7e-05} & \cell{356} & \cell{0.5} & \cell{1.0e-06} & \cell{1107} & \cell{3.4} & \cell{3.1e-04} \\
\cell{p015120} & \cell{18/86} & \cell{2.9} & \cell{3.1e-06} & \cell{22} & \cell{8.0} & \cell{1.2e-05} & \cell{372} & \cell{0.9} & \cell{1.1e-07} & \cell{491} & \cell{2.6} & \cell{2.8e-05} \\
\cell{p015160} & \cell{17/97} & \cell{3.9} & \cell{3.9e-06} & \cell{22} & \cell{15.7} & \cell{5.8e-06} & \cell{501} & \cell{6.3} & \cell{5.2e-07} & \cell{640} & \cell{6.4} & \cell{4.0e-05} \\
\cell{p015200} & \cell{16/109} & \cell{5.7} & \cell{8.5e-06} & \cell{28} & \cell{37.1} & \cell{1.0e-05} & \cell{580} & \cell{16.1} & \cell{4.5e-07} & \cell{901} & \cell{12.6} & \cell{8.7e-05} \\
\cell{p015240} & \cell{19/121} & \cell{8.7} & \cell{4.9e-06} & \cell{29} & \cell{59.3} & \cell{6.3e-06} & \cell{469} & \cell{20.4} & \cell{3.5e-07} & \cell{613} & \cell{16.3} & \cell{3.9e-05} \\
\cell{p015280} & \cell{21/135} & \cell{13.4} & \cell{8.2e-06} & \cell{32} & \cell{193.6} & \cell{6.5e-06} & \cell{599} & \cell{46.3} & \cell{4.4e-07} & \cell{861} & \cell{25.1} & \cell{7.8e-05} \\
\cell{p015320} & \cell{20/152} & \cell{27.4} & \cell{6.4e-06} & \cell{33} & \cell{333.0} & \cell{4.8e-06} & \cell{736} & \cell{81.2} & \cell{5.1e-07} & \cell{1070} & \cell{44.9} & \cell{6.0e-05} \\
\cell{p015360} & \cell{19/161} & \cell{33.8} & \cell{2.6e-06} & \cell{32} & \cell{543.1} & \cell{3.0e-06} & \cell{694} & \cell{107.8} & \cell{3.8e-07} & \cell{805} & \cell{46.2} & \cell{1.9e-05} \\
\cell{p015400} & \cell{20/164} & \cell{34.2} & \cell{1.1e-05} & \cell{33} & \cell{991.1} & \cell{4.9e-06} & \cell{1042} & \cell{205.0} & \cell{1.9e-06} & \cell{946} & \cell{78.5} & \cell{7.9e-05} \\
\cell{p015440} & \cell{23/167} & \cell{41.7} & \cell{5.5e-06} & \cell{33} & \cell{1598.9} & \cell{4.8e-06} & \cell{755} & \cell{225.1} & \cell{8.0e-07} & \cell{997} & \cell{118.6} & \cell{5.9e-05} \\
\cell{p015480} & \cell{20/188} & \cell{82.0} & \cell{2.0e-05} & \cell{36} & \cell{2380.3} & \cell{4.0e-06} & \cell{854} & \cell{300.2} & \cell{9.8e-07} & \cell{872} & \cell{213.6} & \cell{8.6e-05} \\
\cell{p015520} & \cell{24/203} & \cell{103.8} & \cell{1.2e-05} & \cell{40} & \cell{3571.9} & \cell{2.5e-06} & \cell{820} & \cell{353.1} & \cell{3.2e-07} & \cell{922} & \cell{539.1} & \cell{7.1e-05} \\
\cell{p015560} & \cell{19/206} & \cell{121.1} & \cell{5.8e-06} & \cell{42} & \cell{5365.4} & \cell{1.8e-06} & \cell{823} & \cell{412.2} & \cell{1.1e-06} & \cell{1003} & \cell{776.1} & \cell{4.5e-05} \\
\bottomrule
\end{tabular}
\end{center}
\vspace{-3ex}
\end{table}
\end{scriptsize}

The results in Table~\ref{tbl:ex3} show that \ref{eq:prox_nt_scheme} outperforms all other methods when the problem size increases.
If the network is dense, \ref{eq:prox_nt_scheme} is much more efficient than SDPT3 compared to the sparse case.
It is not surprised that \ref{eq:prox_nt_scheme} also beats  both ADMM and Chambolle-Pock methods since we require a high accurate solution, which is often a disadvantage of first-order methods.

The performance profile can be considered as a standard way to compare different optimization algorithms.
A performance profile is built based on a set $\mathcal{S}$ of $n_s$ algorithms (solvers) and a collection $\mathcal{P}$ of $n_p$ problems. We  build a profile based on computational time.
We denote by $T_{ij} := \textit{computational time required to solve problem $i$ by solver $j$}$.
We compare the performance of solver $j$ on problem $i$ with the best performance of any algorithm on this problem; that is, we compute the performance ratio
$r_{ij} := \frac{T_{ij}}{\min\{T_{ik} \mid k \in \mathcal{S}\}}$.
Now, let $\tilde{\rho}_j(\tilde{\tau}) := \frac{1}{n_p}\mathrm{size}\left\{i\in\mathcal{P} \mid r_{ij}\leq \tilde{\tau}\right\}$ for
$\tilde{\tau} \in\R_{+}$. The function $\tilde{\rho}_j:\R\to [0,1]$ represents a probability for solver $j$ that a performance ratio is within a
factor $\tilde{\tau}$ of the best possible ratio. We use the term ``performance profile'' for the distribution function $\tilde{\rho}_j$ as a performance
metric.
In the following numerical examples, we plot the performance profiles in $\log_2$-scale, i.e., $\rho_j(\tau) :=
\frac{1}{n_p}\mathrm{size}\left\{i\in\mathcal{P} \mid \log_2(r_{i,j})\leq \tau := \log_2\tilde{\tau}\right\}$.

\begin{figure}[hpt!]
\vspace{-4ex}
\centering
\includegraphics[width=0.70\linewidth]{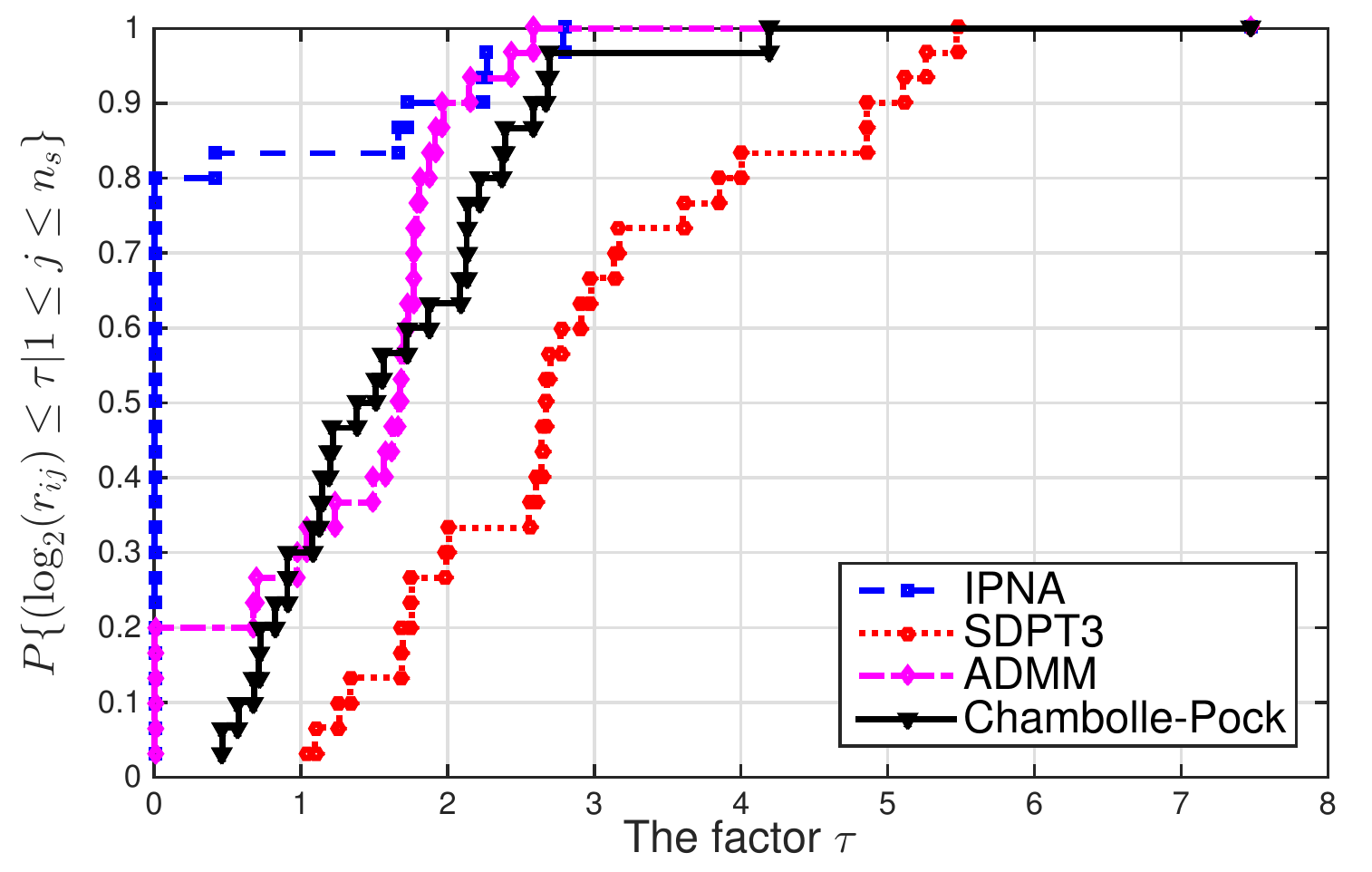}
\vspace{-1ex}
\caption{{\small Performance Profile in time[s] of $4$ methods over $30$ problem instances}}\label{fig:Profile}
\vspace{-4ex}
\end{figure}

Figure~\ref{fig:Profile} shows the performance profile of the four algorithms on the set of $30$ problem instances. 
\ref{eq:prox_nt_scheme} achieves $24/30$ ($80\%$) with the best performance, while ADMM obtains $6/30$ ($20\%$) with the best performance.
In terms of time, both inexact proximal Newton method and first-order methods outperform SDPT3 in this experiment. 
We can also see from Table~\ref{tbl:ex3} that ADMM gives the best solution quality in most cases, while CP gives the worst solution quality.

\beforesubsubsec
\subsubsection{\textbf{The effect of inexactness to the convergence of \ref{eq:prox_nt_scheme}}}
\aftersubsubsec
Now, we show how the accuracy of inexact oracles affects the overall convergence of \ref{eq:prox_nt_scheme} when solving \eqref{ex_2:l1tvd}.
As indicated by Theorems \ref{th:global_convergence3}, \ref{th:local_covergence1}, and \ref{th:local_covergence2}, \ref{eq:prox_nt_scheme} can achieve different local convergence rates, or can diverge. In this experiment, we  analyze the convergence or divergence of \ref{eq:prox_nt_scheme} under different accuracy levels of inexact oracles. 

We use the same model \eqref{ex_2:l1tvd} and generate data according to Subsection \ref{eg:Network} but using $A := \texttt{rand}(p, 0.1p)$, where $p = 500$. 
For configuration of the experiment, we set the maximum number of iterations at $100$ as a safeguard, but also terminate the algorithm if $\lambda_k \leq 10^{-9}$ and the relative objective value satisfies $F(x^k) - F^{\star} \leq \varepsilon\max\set{1, \vert F^{\star}\vert}$, where $\varepsilon = 10^{-11}$ for the linear convergence rate, and $\varepsilon = 10^{-12}$ for the quadratic convergence rate, respectively.
The optimal value $F^{\star}$ is computed by running SDPT3 to the best accuracy.
We also choose $\delta := \frac{1}{100}$ as suggested by Theorem~\ref{th:local_covergence2}(a).

The global convergence of \ref{eq:prox_nt_scheme} is reflected in Figure \ref{fig:global_convergence}.
\begin{figure}[htp!]
\begin{center}
\includegraphics[width=0.95\linewidth]{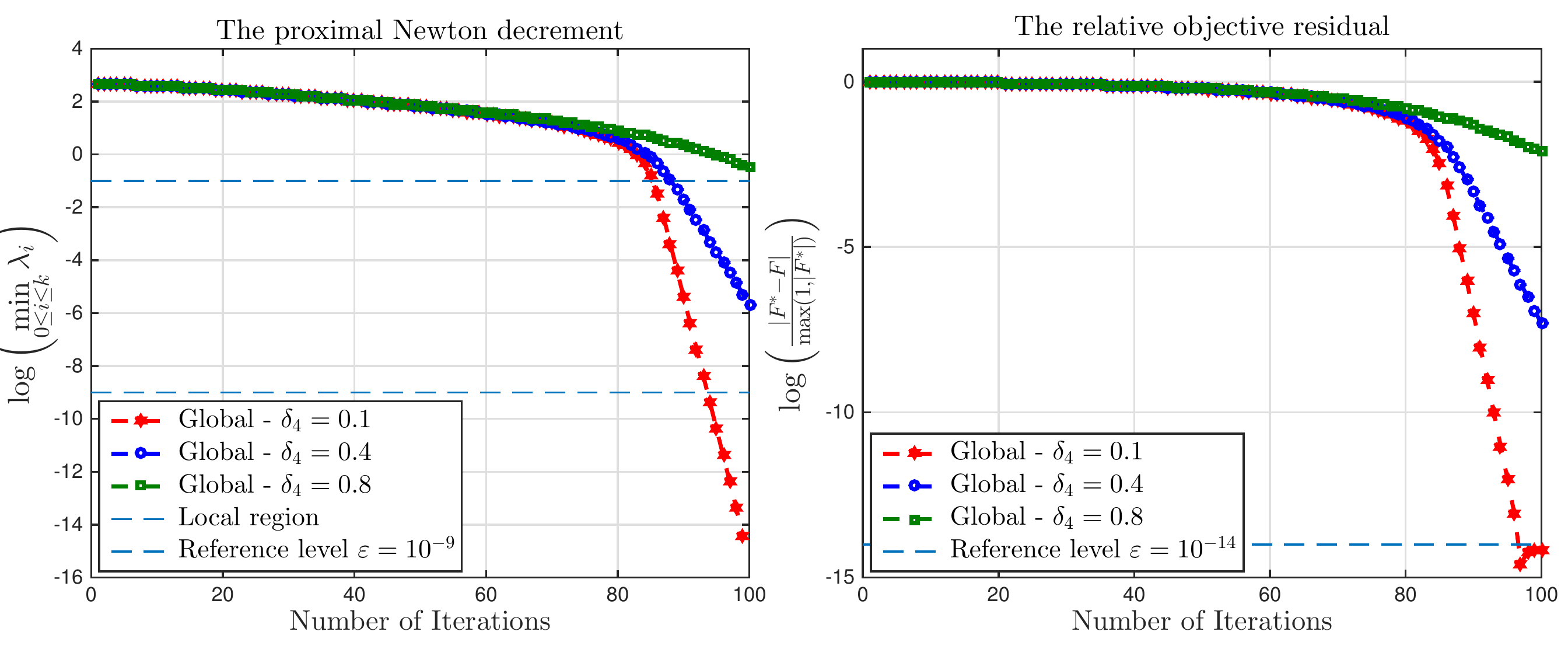}
\vspace{-1ex}
\caption{{\small Global convergence behavior of \ref{eq:prox_nt_scheme} in Theorem \ref{th:global_convergence3}.}}\label{fig:global_convergence}
\vspace{-6ex}
\end{center}
\end{figure}
The left plot reveals the inexact proximal Newton decrement $\lambda_{k_{\ast}} := \min\set{\lambda_i \mid 1 \leq i \leq k}$ computed from different accuracy levels of the subproblem in \eqref{eq:inexact_subp} as presented in Theorem \ref{th:global_convergence3}. 
Clearly, the more accurate is \eqref{eq:inexact_subp}, the faster convergence in $\lambda_k$ is achieved.
If $\delta_4$ is too rough, e.g.,  $\delta_4 = 0.8$, then both global and local behaviors of the algorithm become worse.
The right plot provides the convergence of the relative objective residuals $\frac{F(x^k) - F^{\star}}{\max\set{\vert F^{\star}\vert, 1}}$ under different accuracy level $\delta_4$ of the subproblem.
We can see  that the global convergence rate is sublinear, i.e., $\BigO{\frac{1}{k}}$, while its local convergence can range from linear to quadratic. 

Our next step is to verify the local convergence of Theorem \ref{th:local_covergence2}, and how inexact oracles affect the local convergence of  \ref{eq:prox_nt_scheme}. 
By choosing different values of $\delta$ we obtain different levels of inexact oracles in $\psi^{\ast}$.
Figure \ref{fig:local_ln_convergence}, Figure \ref{fig:local_sl_convergence}, and Figure \ref{fig:local_qr_convergence} show an R-linear, R-superlinear, and R-quadratic convergence rate of \ref{eq:prox_nt_scheme}, respectively. Here, the reference level $\varepsilon$ representing the desired accuracy of the solution is given in the legend of these figures.
\begin{figure}[htp!]
\vspace{-3ex}
\begin{center}
\includegraphics[width=0.9\linewidth]{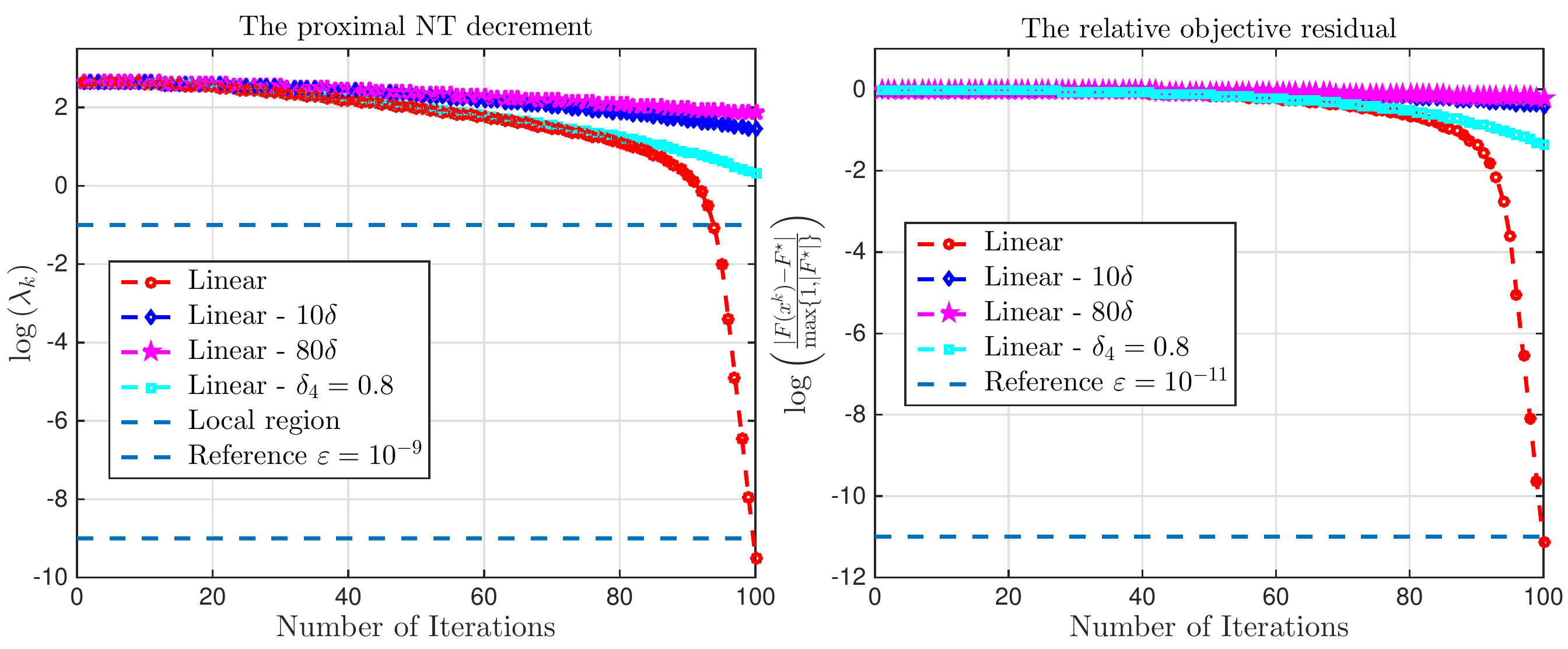}
\vspace{-1ex}
\caption{Local linear convergence of \ref{eq:prox_nt_scheme} under the effect of inexact oracles.}\label{fig:local_ln_convergence}
\vspace{-6ex}
\end{center}
\end{figure}

\begin{figure}[htp!]
\vspace{-4ex}
\begin{center}
\includegraphics[width=0.9\linewidth]{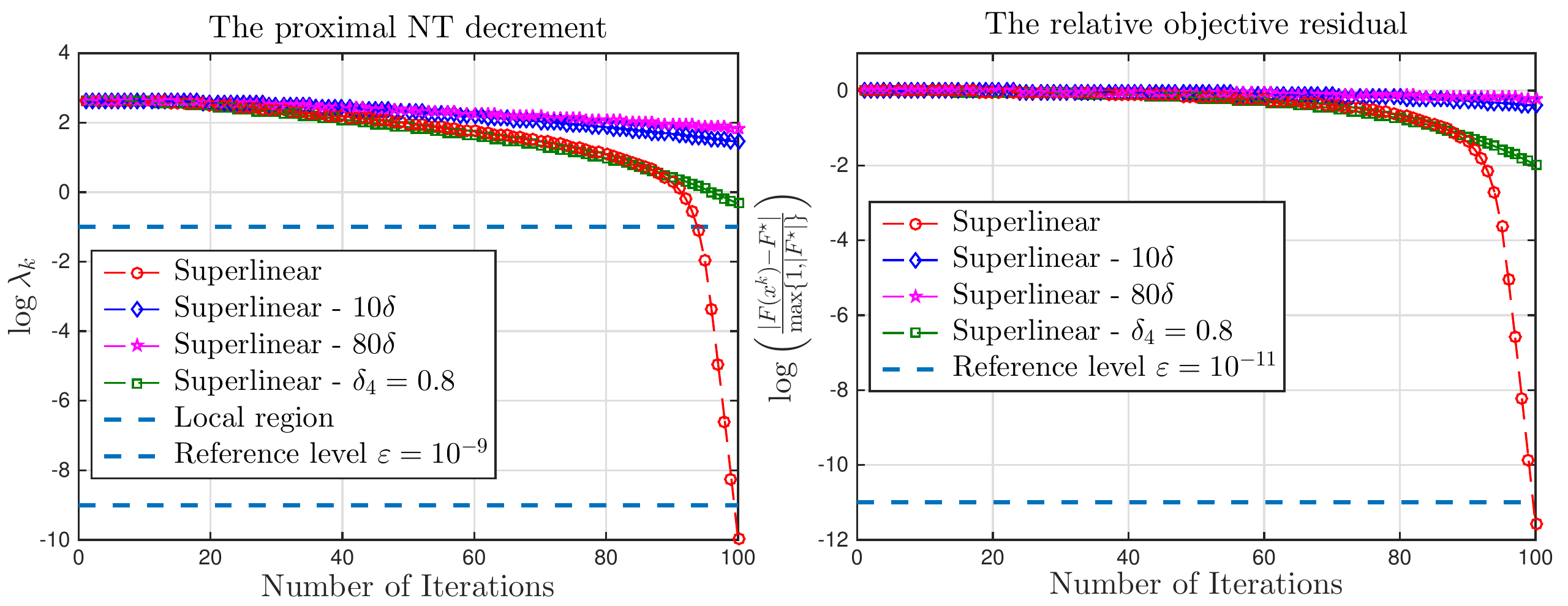}
\vspace{-1ex}
\caption{Local superlinear convergence of \ref{eq:prox_nt_scheme} under the effect of inexact oracles.}\label{fig:local_sl_convergence}
\vspace{-6ex}
\end{center}
\end{figure}

\begin{figure}[htp!]
\begin{center}
\includegraphics[width=0.9\linewidth]{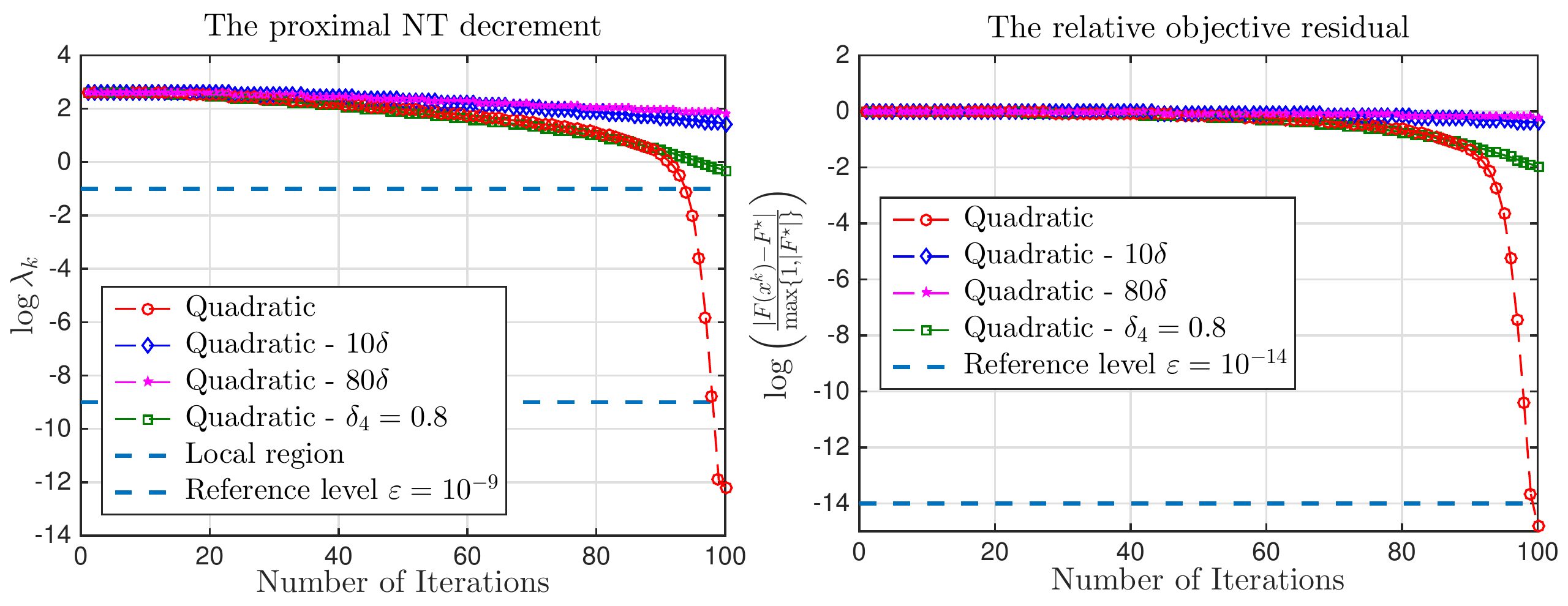}
\vspace{-1ex}
\caption{Local quadratic convergence of \ref{eq:prox_nt_scheme} under the effect of inexact oracles. }
\label{fig:local_qr_convergence}
\vspace{-5ex}
\end{center}
\end{figure}

As we can see from Figure~\ref{fig:local_ln_convergence}, if we choose the parameters as in Theorem \ref{th:local_covergence2}(a) to reflect a local linear convergence rate, then we observe a sublinear convergence in a few dozen of iterations due to slow global convergence rate, but  a fast local convergence at the last iterations.
If we multiply the accuracy $\delta$ by $10$ and $80$, respectively, we can see from this figure that the linear convergence is lost and the method tends to diverge. 
If we choose the inexact level $\delta_4 = 0.8$ in the subproblem  \eqref{eq:inexact_subp}, then we also get a significantly slow linear convergence rate, thus confirming our theory. 

The superlinear and quadratic convergence rates are reflected in Figure \ref{fig:local_sl_convergence} and  Figure \ref{fig:local_qr_convergence}, respectively. 
Both figures look very similar, but the quadratic convergence case achieves much higher accuracy up to $10^{-12}$ after around $100$ iterations.
If we increase the inexactness of the inexact oracle by multiplying $\delta$ by $10$ and $80$, respectively, then \ref{eq:prox_nt_scheme} shows a slow convergence or even divergence. 
If we increase the inexactness $\delta_4$ of the subproblem in \eqref{eq:inexact_subp} to $0.8$, we again obtain a much slower convergence rate.

\beforesubsubsec
\subsection{\bf \ref{eq:prox_nt_scheme} for Graphical Lasso with inexact oracles}
\aftersubsubsec
Proximal Newton-type methods have been proven to be efficient for graphical LASSO \cite{Friedman2008,Hsieh2011,Olsen2012}.
In this example, we also show that our theory can be useful for this problem.

Consider a recent graphical LASSO model in \cite{zhang2018linear}. 
Assume that the data matrix has a sparse structure $G$, then the original model can be written as
\begin{equation}\label{eq:ex1_gr_p}
F^{\star} := \min_{Y\succ 0} \Big\{ F(Y) := \iprods{C_{\lambda}, Y}-\log\det(Y) ~~~~\mathrm{s.t.}~~~ Y_{ij}=0,~\forall (i,j)\notin G \Big\},
\end{equation}
where $C_{\lambda}$ is a soft-threshold operator which serves as the penalty item, that can recover the sparse graph $G$. 
This problem can be cast into \eqref{eq:prob_main} with $\varphi(A^{\top}Y) := \delta_{\set{Y_{ij}=0,~\forall\ (i,j)\notin G}}(Y)$ the indicator function of the feasible set and $\psi(Y) := \iprods{C_{\lambda}, Y} - \log\det(Y)$.
The dual problem of \eqref{eq:ex1_gr_p} becomes \eqref{eq:prob_dual_ori2}, which has the same form as problem (15) in \cite{zhang2018linear}.

We focus on aspects of inexactness: (i) the inexactness of solutions of \eqref{eq:inexact_subp} and (ii) the Hessian and the Newton decrement measurement reflected by Cholesky decomposition. 
Instead of using linesearch, we use the stepsize given by \eqref{eq:step_size2} based on the self-concordance theory.

For (i), we compute the Newton direction inexactly by controlling the tolerance of the pre-conditioned conjugate gradient (PCG) method. 
For (ii), we use an incomplete Cholesky decomposition instead of the exact Cholesky decomposition.
More concretely, when we approximately compute the lower triangular matrix $\tilde{L}$ of $C_{\lambda} - AX$ in the dual problem  \eqref{eq:prob_dual_ori2} such that $C_{\lambda} - AX \approx \tilde{L}\tilde{L}^{\top}$, we fill all other off-diagonal elements by zeros, if the corresponding original entries of the input matrix are zeros.
In this way, we further take advantages of sparsity structure of the original method and bring the inexactness to the Hessian-related quantity indirectly.
Due to the expensive computation, we do not exactly calculate $\delta_2$ and $\delta_3$ in this example since they do not enter in our stepsize as well as the subproblem \eqref{eq:inexact_subp}.

For data, we use both the real biology dataset from \cite{Li2010} and the synthetic data with sample covariance matrices and the threshold parameter generated from real sparse matrix collection in \cite{zhang2018linear} (\href{https://sparse.tamu.edu/}{https://sparse.tamu.edu/}).
Since the Newton-CG (NCG) method in the latest paper \cite{zhang2018linear} already compared and beat QUIC \cite{Hsieh2011} in their experiments, we make use of the chordal property and only compare our algorithm with the one in \cite{zhang2018linear}.
More concretely, the parameter configuration of our algorithm is specified as follows:
\begin{compactitem}
\item We choose $\delta_0 := 0$ and $\delta_4 := 10^{-3}$ leading to $\alpha_k = \frac{1-\delta_4}{1 + (1-\delta_4)\lambda_k}$ as our stepsize.
\item We terminate the PCG for computing the search direction if the default error is below $10^{-3}$.
\item Since we do not use $\delta_2$ and $\delta_3$ in our implementation, we leave them unspecified since evaluating these quantities are expensive in this example.
\item Following \cite{zhang2018linear}, we terminate both algorithms: \ref{eq:prox_nt_scheme} and NCG by $\lambda_k \leq 10^{-6}$.
\item The maximum number of iterations in \ref{eq:prox_nt_scheme} is set to $100$.
\end{compactitem}
The results are listed in Table \ref{tbl:portf_cp1}, where $p$ is the dimension of the original graph/data, ``\texttt{iter}'' is the number of iterations, ``$\lambda_{e}$'' is the weighted norm $\lambda_k$ which is used by NCG when the algorithm is terminated, ``\texttt{soldiff}'' measures the relative solution difference of two methods for primal solutions, and ``$\texttt{t}_{\textrm{ratio}}$'' represents the time ratio of NCG over \ref{eq:prox_nt_scheme} (i.e., the ratio $\frac{t_{\mathrm{NCG}}}{t_{\mathrm{\ref{eq:prox_nt_scheme}}}}$).

\begin{table}[hpt!]
\vspace{-3ex}
\newcommand{\cell}[1]{{\!\!}#1{\!\!}}
\newcommand{\cellf}[1]{{\!\!\!}#1{\!\!}}
\newcommand{\cellbf}[1]{{\!\!}{#1}{\!\!\!}}
\begin{center}
\caption{The performance of NCG and \ref{eq:prox_nt_scheme} for solving the graphical lasso problem \eqref{eq:ex1_gr_p}.}\label{tbl:portf_cp1}
\vspace{-1ex}
\begin{tabular}{ ll | rrr | rrr | rr}
\toprule
\multicolumn{2}{c|}{Problem} & \multicolumn{3}{c|}{ \cell{\ref{eq:prox_nt_scheme}} } & \multicolumn{3}{c|}{\cell{NCG with linesearch}} & \multicolumn{2}{c}{\cell{Comparison}}\\  \midrule
\cell{Name}  & \cell{$p$} & \cell{\texttt{iter}} & \cell{\!\texttt{time[s]}\!}  & \cell{$\lambda_{e}$} & \cell{\texttt{iter}} & \cell{\!\texttt{time[s]}\!}  & \cell{$\lambda_{e}$} & \cell{\texttt{soldiff}} & \cell{\!$\texttt{t}_{\textrm{ratio}}$\!}\\ \midrule
\multicolumn{10}{c}{Synthetic Data} \\ \midrule
\cell{3eltdual} & \cell{9000}& \cell{4} & \cell{11.45} & \cell{3.0e-07} & \cell{3} & \cell{13.15} & \cell{2.7e-07}& \cell{3.0e-12} & \cell{1.15}  \\
\cell{bcsstm38} & \cell{8032}& \cell{3} & \cell{2.84} & \cell{6.1e-07} & \cell{3} & \cell{4.36} & \cell{5.2e-10}& \cell{4.5e-12} & \cell{1.54}  \\
\cell{cage8} & \cell{1015}& \cell{7} & \cell{62.99} & \cell{3.2e-07} & \cell{4} & \cell{116.64} & \cell{3.1e-10}& \cell{1.1e-09} & \cell{1.85}  \\
\cell{cryg10000} & \cell{10000}& \cell{6} & \cell{543.31} & \cell{4.1e-08} & \cell{4} & \cell{634.06} & \cell{2.8e-10}& \cell{5.6e-12} & \cell{1.17}  \\
\cell{FlyingRobot1} & \cell{798}& \cell{6} & \cell{4.23} & \cell{5.3e-07} & \cell{4} & \cell{9.98} & \cell{5.2e-11}& \cell{5.7e-10} & \cell{2.36}  \\
\cell{G32} & \cell{2000}& \cell{4} & \cell{2.79} & \cell{7.3e-07} & \cell{4} & \cell{5.17} & \cell{7.5e-12}& \cell{4.9e-11} & \cell{1.85}  \\
\cell{G50} & \cell{3000}& \cell{5} & \cell{5.49} & \cell{3.9e-09} & \cell{4} & \cell{7.75} & \cell{6.1e-11}& \cell{2.6e-13} & \cell{1.41}  \\
\cell{G57} & \cell{5000}& \cell{5} & \cell{9.10} & \cell{2.1e-07} & \cell{4} & \cell{12.75} & \cell{2.6e-10}& \cell{5.5e-12} & \cell{1.40}  \\
\cell{lshp2614} & \cell{2614}& \cell{6} & \cell{108.29} & \cell{1.8e-07} & \cell{4} & \cell{162.54} & \cell{7.4e-11}& \cell{1.3e-10} & \cell{1.50}  \\
\cell{lshp3025} & \cell{3025}& \cell{6} & \cell{137.24} & \cell{3.8e-07} & \cell{4} & \cell{215.66} & \cell{6.7e-11}& \cell{3.2e-10} & \cell{1.57}  \\
\cell{NotreDamey} & \cell{2114}& \cell{3} & \cell{1.57} & \cell{7.6e-08} & \cell{3} & \cell{2.19} & \cell{4.1e-11}& \cell{1.8e-12} & \cell{1.40}  \\
\cell{orsirr2} & \cell{886}& \cell{6} & \cell{7.17} & \cell{1.7e-07} & \cell{4} & \cell{13.35} & \cell{2.2e-10}& \cell{4.8e-10} & \cell{1.86}  \\
\cell{sherman3} & \cell{5005}& \cell{5} & \cell{56.11} & \cell{5.1e-07} & \cell{4} & \cell{99.77} & \cell{1.3e-11}& \cell{3.0e-10} & \cell{1.78}  \\
\cell{ukerbe1} & \cell{5981}& \cell{3} & \cell{5.23} & \cell{5.8e-07} & \cell{3} & \cell{8.63} & \cell{1.2e-10}& \cell{1.2e-11} & \cell{1.65}  \\
\cell{USpowerGrid} & \cell{4941}& \cell{3} & \cell{4.66} & \cell{4.9e-07} & \cell{3} & \cell{7.09} & \cell{6.7e-09}& \cell{5.2e-12} & \cell{1.52}  \\
\midrule
\multicolumn{10}{c}{Real Data} \\ \midrule
\cell{Arabidopsis} & \cell{834}& \cell{4} & \cell{1.27} & \cell{1.2e-07} & \cell{4} & \cell{1.41} & \cell{5.0e-09}& \cell{2.8e-12} & \cell{1.11}  \\
\cell{ER} & \cell{692}& \cell{4} & \cell{0.89} & \cell{1.5e-08} & \cell{4} & \cell{1.25} & \cell{5.8e-11}& \cell{8.8e-14} & \cell{1.40}  \\
\cell{hereditarybc} & \cell{1869}& \cell{4} & \cell{21.06} & \cell{2.9e-07} & \cell{4} & \cell{35.39} & \cell{1.7e-07}& \cell{7.3e-12} & \cell{1.68}  \\
\cell{Leukemia} & \cell{1255}& \cell{3} & \cell{0.60} & \cell{7.6e-08} & \cell{3} & \cell{0.76} & \cell{2.7e-09}& \cell{8.6e-13} & \cell{1.25}  \\
\cell{Lymph} & \cell{587}& \cell{4} & \cell{0.24} & \cell{8.5e-10} & \cell{3} & \cell{0.25} & \cell{9.1e-07}& \cell{2.4e-14} & \cell{1.03}  \\
\bottomrule
\end{tabular}
\vspace{-4ex}
\end{center}
\end{table}

From Table~\ref{tbl:portf_cp1}, we can see that our algorithm performs better than the state-of-the-art NCG algorithm with linesearch for both collections of datasets. 
Although for some graphs we cannot significantly accelerate, we point out that NCG has already taken the advantages of chordal structure and used linesearch, while we specify a stepsize, and the acceleration is highly related to the sparsity and the shape of the graph. 
Besides, we need slightly more iterations and have a greater $\lambda_e$ value, because we did not solve the subproblem up to a high accuracy, which leads to a smaller descent. 
However, we still met the stopping criterion and obtained almost the same solution as NCG (see the \texttt{soldiff} column in Table~\ref{tbl:portf_cp1}).
\vspace{-3ex}

\appendix
\section{The Proof of Technical Results in The Main Text}\label{sec:apendix1}
\aftersec
This appendix provides the proofs of technical results and missing concepts in the main text.

\beforesubsec
\subsection{\bf The proof of Lemma \ref{le:global_properties}: Properties of global inexact oracle}\label{apdx:le:global_properties}
\aftersubsec
(a)~Substituting $x = y$ into \eqref{eq:global_inexact_oracle}, we obtain  \eqref{eq:global_inexact_oracle_pro1} directly for all $x\in\dom{f}$.

\vspace{1ex}
\noindent (b)~Clearly, if $\iprods{g(\bar{x}), y - \bar{x}} \geq 0$ for all $y\in\dom{f}$, then $\iprods{g(\bar{x}), x^{\star} - \bar{x}} \geq 0$ for a minimizer $x^{\star}$ of $f$.
Using this relation into \eqref{eq:global_inexact_oracle}, we have
\begin{equation*}
\begin{array}{llr}
f^{\star} &= f(x^{\star}) \overset{\tiny\eqref{eq:global_inexact_oracle}}{\geq} \tilde{f}(\bar{x}) + \iprods{g(\bar{x}), x^{\star} - \bar{x}} + \omega((1-\delta_0)\tnorm{x^{\star} - \bar{x}}_{\bar{x}})& \vspace{1ex}\\
& \geq \tilde{f}(\bar{x}) + \omega((1-\delta_0)\tnorm{x^{\star} - \bar{x}}_{\bar{x}}) &~~~\text{since}~~\iprods{g(\bar{x}), x^{\star} - \bar{x}} \geq 0\vspace{1ex}\\
& \geq \tilde{f}(\bar{x}) \overset{\tiny\eqref{eq:global_inexact_oracle_pro1}}{\geq} f(\bar{x}) - \delta_1, 
\end{array}
\end{equation*}
which implies $f^{\star} \leq f(\bar{x}) \leq f^{\star} + \delta_1$.

\vspace{1ex}
\noindent(c)~Let $\nabla f(x)$ be a (sub)gradient of $f$ at $x \in\intx{\dom{f}}$.
For $y\in\dom{f}$, it follows from \eqref{eq:global_inexact_oracle} and \eqref{eq:global_inexact_oracle_pro1} that
\begin{equation*}
f(y)\geq f(x) + \iprod{\nabla f(x),y-x} \overset{\tiny\eqref{eq:global_inexact_oracle_pro1}}{\geq} \tilde{f}(x)+\iprod{\nabla f(x),y-x}.
\end{equation*}
Subtracting this estimate from the second inequality of \eqref{eq:global_inexact_oracle}, we have
\begin{equation}\label{eq:bd_prev}
\iprod{\nabla f(x)-g(x),y-x}\leq \omega_{\ast}\left((1+\delta_0)\tnorm{y - x}_x\right) + \delta_1,
\end{equation}
provided that $\tnorm{y - x}_x < \frac{1}{1+\delta_0}$. Let us consider  an arbitrary $z\in\R^p$ such that
\begin{equation*}
\tnorm{\nabla f(x)-g(x)}_x^{\ast}=\abs{\iprod{\nabla f(x)-g(x),z}}~~~\textrm{and}~~~\tnorm{z}_x=1.
\end{equation*}
Let us choose $y = y_{\tau}(x) := x + \tau\textrm{sign}(\iprod{\nabla f(x)-g(x),z})z$ for some $\tau > 0$.
Since $x\in\intx{\dom{f}}$, for sufficiently small $\tau$, $y\in\dom{f}$.
Moreover, \eqref{eq:bd_prev} becomes $\tau\tnorm{\nabla f(x)-g(x)}_x^{\ast}\leq\omega_{\ast}\left((1+\delta_0)\tau\right)+\delta_1$,
which is equivalent to
\begin{equation}\label{eq:bd_grad}
\tnorm{\nabla f(x)-g(x)}_x^{\ast}\leq s(\tau;\delta_0,\delta_1) := \tfrac{\omega_{\ast}\left((1+\delta_0)\tau\right)+\delta_1}{\tau}.
\end{equation}
Let us take $\tau := \frac{\delta_2}{(1+\delta_0 + \delta_2)(1+\delta_0)}$ for some sufficiently small $\delta_2 > 0$.
Then, we can easily check that $\tnorm{y - x}_x = \tau < \frac{1}{1+\delta_0}$.
In this case, the right-hand side of \eqref{eq:bd_grad} becomes
\begin{equation}\label{eq:bd_grad:c}
s(\tau;\delta_0,\delta_1) = \tfrac{(1+\delta_0)(1+\delta_0 + \delta_2)}{\delta_2}\left[\delta_1 + \ln\left(1 + \tfrac{\delta_2}{1+\delta_0}\right)\right] - (1+\delta_0),
\end{equation}
for any $\delta_2 >0$.
Minimizing  the right-hand side of \eqref{eq:bd_grad:c} w.r.t. $\delta_2 > 0$, we can show that  the minimum is attained at $\delta_2 := \delta_2(\delta_0,\delta_1) > 0$ which is the unique solution $\delta_2 := (1+\delta_0)\omega^{-1}(\delta_1)$ of $\omega \big( \tfrac{\delta_2}{1+\delta_0} \big) = \delta_1$ in $\delta_2$, where $\omega^{-1}$ is the inverse function of $\omega$ (note that $\omega(\tau) = \tau - \ln(1 + \tau)$).

Now, substituting $\delta_2 = \delta_2(\delta_0,\delta_1)$ back into $s(\tau;\delta_0,\delta_1)$, we can see that the minimum value of \eqref{eq:bd_grad:c} is $\delta_2 := \delta_2(\delta_0,\delta_1) = (1+\delta_0)\omega^{-1}(\delta_1)$.
By directly using the definition of $\omega$, it is obvious that if $\delta_0 \to 0$ and $\delta_1\to 0$, then $\delta_2 := (1+\delta_0)\omega^{-1}(\delta_1) \to 0$.

\vspace{1ex}
\noindent(d)~Let us consider the function $\varphi(y) := f(y) - \iprods{\nabla{f}(x^0), y}$ for some $x^0 \in \dom{f}$.
It is clear that $\nabla{\varphi}(x^0) = 0$, which shows that $x^0$ is a minimizer of $\varphi$.
Hence, we have $\varphi(x^0) \leq \varphi(x - tH(x)^{-1}h(x))$ for some $t > 0$ such that $x\in\intx{\dom{f}}$ and $x - tH(x)^{-1}h(x) \in\dom{f}$.
If we define $\tilde{\varphi}(x) := \tilde{f}(x) - \iprods{\nabla{f}(x^0), x}$, and $h(x) := g(x) - \nabla{f}(x^0)$, then, by using \eqref{eq:global_inexact_oracle}, we can further derive
\begin{equation*}
\varphi(x^0) \leq \varphi(x - tH(x)^{-1}h(x)) \leq \tilde{\varphi}(x) - t(\tnorm{h(x)}_{x}^{\ast})^2 + \omega_{\ast}\left( (1+\delta_0)t \tnorm{h(x)}_{x}^{\ast}) \right) + \delta_1.
\end{equation*}
Minimizing the right-hand side of the last estimate w.r.t. $t > 0$, we obtain
\begin{equation*}
\varphi(x^0) \leq   \tilde{\varphi}(x) - \omega\left(\tfrac{\tnorm{h(x)}_{x}^{\ast}}{1+\delta_0}\right) + \delta_1,
\end{equation*}
given $t = \frac{1}{(1+\delta_0)(1+\delta_0+\tnorm{h(x)}_x^{\ast})}$.
Using the definition of $\varphi$ and the Cauchy-Schwarz inequality, we have
\begin{equation*}
\begin{array}{ll}
\omega\left(\frac{\tnorm{h(x)}_{x}^{\ast}}{1+\delta_0}\right) &\leq \tilde{f}(x) - f(x^0) - \iprods{\nabla{f}(x^0), x - x^0} + \delta_1 \vspace{-0.5ex}\\
&\overset{\tiny\eqref{eq:global_inexact_oracle}}{\leq} -\omega((1-\delta_0)\tnorm{x-x^0}_x)+\iprods{g(x)-\nabla f(x^0),x-x^0}+\delta_1 \vspace{1ex}\\
& \leq \tnorm{h(x)}_{x}^{\ast}\tnorm{x-x^0}_x + \delta_1.
\end{array}
\end{equation*}
By letting $x^0 = y$ into this inequality, we obtain exactly \eqref{eq:global_inexact_oracle_pro6a}.
\Eproof

\beforesubsec
\subsection{\bf The proof of Lemma \ref{le:norm_relation}: Properties of local inexact oracle}\label{apdx:le:norm_relation}
\aftersubsec
From the second line of \eqref{eq:local_inexact_oracle}, for any $u\in\R^p$, we have 
\begin{equation*}
(1-\delta_3)^2\norms{u}_x^2  \leq  \tnorm{u}_x^2 \leq (1+\delta_3)^2\norms{u}_x^2,
\end{equation*}
which implies the first expression of \eqref{eq:norm_relation}.

Using again the second line of \eqref{eq:local_inexact_oracle}, we have $\frac{1}{(1+\delta_3)^2}\nabla^2{f}(x)^{-1} \preceq H(x)^{-1} \preceq \frac{1}{(1-\delta_3)^2}\nabla^2{f}(x)^{-1}$.
Hence, for any $v\in\R^p$, one has
\begin{equation*}
\frac{1}{(1+\delta_3)^2}(\norms{v}_x^{*})^2   \leq  (\tnorm{v}_x^{*})^2\leq  \frac{1}{(1-\delta_3)^2}(\norms{v}_x^{*})^2,
\end{equation*}
which implies the second expression of  \eqref{eq:norm_relation}.

Now, we prove \eqref{eq:key_bound3}.
For any $x, y \in \Xc$, using \eqref{eq:norm_relation} with $u := y - x$, we have
\begin{equation}\label{eq:proof_keybound3a}
\frac{\tnorm{y-x}_x}{(1 + \delta_3)} \leq \norm{y-x}_x \leq \frac{\tnorm{y-x}_x}{(1 - \delta_3)}.
\end{equation}
For  $x, y \in \Xc$ such that $\tnorm{y-x}_x < 1-\delta_3$ for $\delta_3 \in [0, 1)$, the estimate \eqref{eq:proof_keybound3a} implies that 
\begin{equation}\label{eq:proof_keybound3}
(1 - \norms{y-x}_x)^2 \geq \left(1 - \frac{\tnorm{y-x}_x}{1-\delta_3}\right)^2 = \frac{\left(1 - \delta_3 - \tnorm{y-x}_x\right)^2}{(1-\delta_3)^2}.
\end{equation}
Since $\tnorm{y-x}_x < 1-\delta_3$, by \eqref{eq:proof_keybound3a}, we have $\norm{y-x}_x < 1$.
Hence, by \cite[Theorem 4.1.6]{Nesterov2004}, we can show that
\begin{equation}\label{eq:proof_keybound3d}
\left(1 - \norm{y-x}_x\right)^2\nabla^2f(x) \preceq \nabla^2f(y) \preceq \frac{1}{\left(1 - \norm{y-x}_x\right)^{2}}\nabla^2f(x).
\end{equation}
Combining \eqref{eq:proof_keybound3d} and \eqref{eq:proof_keybound3}, and using again \eqref{eq:local_inexact_oracle}, we can further derive
\begin{equation*}
\begin{array}{ll}
H(y) &\overset{\tiny\eqref{eq:local_inexact_oracle}}{\succeq} (1-\delta_3)^2\nabla^2{f}(y) \overset{\tiny\eqref{eq:proof_keybound3d}}{\succeq} (1-\delta_3)^2 \left(1 - \norm{y-x}_x\right)^2\nabla^2{f}(x)  \overset{\tiny\eqref{eq:proof_keybound3}}{\succeq}  \left(1 - \delta_3 - \tnorm{y-x}_x\right)^2\nabla^2{f}(x)  \vspace{1ex}\\
& \overset{\tiny\eqref{eq:local_inexact_oracle}}{\succeq}  \Big[\frac{1 - \delta_3 - \tnorm{y-x}_x}{1+\delta_3}\Big]^2H(x),
\end{array}
\end{equation*}
and
\begin{align*}
H(y) \overset{\tiny\eqref{eq:local_inexact_oracle}}{\preceq} (1+\delta_3)^2\nabla^2{f}(y) \overset{\tiny\eqref{eq:proof_keybound3d}}{\preceq} \tfrac{(1+\delta_3)^2}{\left(1 - \norm{y-x}_x\right)^2} \nabla^2f(x) \overset{\tiny\eqref{eq:proof_keybound3}}{\preceq} \left[ \tfrac{(1-\delta_3)(1+\delta_3)}{1 - \delta_3 - \tnorm{y-x}_x}\right]^2\nabla^2f(x) \overset{\tiny\eqref{eq:local_inexact_oracle}}{\preceq} \left[ \tfrac{1+\delta_3}{1 + \delta_3 - \tnorm{y-x}_x}\right]^2H(x).
\end{align*}
Therefore, we obtain the first estimate of \eqref{eq:key_bound3} from these expressions.

From the second line of  \eqref{eq:local_inexact_oracle}, we have $-(2\delta_3 - \delta_3^2)\nabla^2{f}(x) \preceq H(x) - \nabla^2 f(x) \preceq (2\delta_3 + \delta_3^2)\nabla^2 f(x)$.
If we define $G_x := [\nabla^2 f(x)]^{-1/2}(\nabla^2 f(x) - H(x))[\nabla^2 f(x)]^{-1/2}$, then the last estimate implies that
\begin{equation}\label{eq:proof_keybound3e}
\norm{G_x} \leq 2\delta_3+\delta_3^2.
\end{equation}
Moreover, by \eqref{eq:proof_keybound3a}, \eqref{eq:proof_keybound3d},  \eqref{eq:proof_keybound3e}, and the Cauchy--Schwarz inequality in (\textit{i}), we can further derive
\begin{equation*}
\begin{array}{lcl}
 \tnorm{(\nabla^2 f(x)-H(x))v}_y^{\ast} & \overset{\tiny\eqref{eq:proof_keybound3a}}{\leq}  & \tfrac{1}{1-\delta_3} \norms{(\nabla^2 f(x)-H(x))v}_y^{\ast} \vspace{1ex}\\
& \overset{\tiny\eqref{eq:proof_keybound3d}}{\leq} & \tfrac{1}{1-\delta_3} \Big( v^{\top} (\nabla^2 f(x) - H(x))  \tfrac{1}{(1 - \norms{y-x}_x)^2}\nabla^2 f(x)^{-1} (\nabla^2 f(x) - H(x)) v \Big)^{1/2} \vspace{1ex}\\
&=&  \tfrac{1}{(1-\delta_3)(1 - \norm{y-x}_x)} \norms{G_x [\nabla^2 f(x)]^{1/2}v} \vspace{1ex}\\
& \overset{\tiny(i)}{\leq} &  \tfrac{1}{(1-\delta_3)(1 - \norm{y-x}_x)} \norm{G_x} \norm{v}_x \vspace{1ex}\\
& \overset{\tiny\eqref{eq:proof_keybound3e},\eqref{eq:proof_keybound3a}}{\leq} & \tfrac{2\delta_3 + \delta_3^2}{(1-\delta_3)^2\left(1 - (1-\delta_3)^{-1} \tnorm{y-x}_x\right)} \tnorm{v}_x \vspace{1ex}\\
& = & \tfrac{2\delta_3 + \delta_3^2}{(1-\delta_3)(1 -\delta_3 - \tnorm{y-x}_x)} \tnorm{v}_x,
\end{array}
\end{equation*}
which is exactly the second estimate of \eqref{eq:key_bound3}.
\Eproof

\beforesubsec
\subsection{\bf The proof of Lemma \ref{le:inexact_oracle_bounds}: Computational inexact oracle}\label{apdx:le:inexact_oracle_bounds}
\aftersubsec
(a)~We first prove the left-hand side inequality of \eqref{eq:global_inexact_oracle}.
Since $f$ is standard self-concordant, for any $x, y\in\dom{f}$ and $\alpha \in [0, 1]$, we have
\begin{equation}\label{eq:proof_key1a}
\begin{array}{lcl}
f(y) &\overset{(i)}{\geq} & f(x) + \iprods{\nabla{f}(x), y-x} + \omega(\norm{y-x}_x) \vspace{1ex}\\
&\overset{\tiny\eqref{eq:inexact_oracle2}}{\geq} & \hat{f}(x) + \iprods{g(x), y-x}  - \varepsilon + \iprods{\nabla{f}(x) - g(x), y-x} + \omega(\norm{y-x}_x)  \vspace{1ex}\\
&\overset{\tiny(ii),\eqref{eq:norm_relation}}{\geq} & \hat{f}(x) + \iprods{g(x), y-x} - \varepsilon - \tnorm{\nabla{f}(x) - g(x)}^{\ast}_{x}\tnorm{y-x}_x + \omega((1-\delta_3)\tnorm{y-x}_x) \vspace{1ex}\\
&\geq& \hat{f}(x) + \iprods{g(x), y-x} + \omega(\alpha(1-\delta_3)\tnorm{y-x}_x)  - \varepsilon \vspace{1ex}\\
&& - {~} \delta_2\tnorm{y-x}_x + \omega((1-\delta_3)\tnorm{y-x}_x) - \omega(\alpha(1-\delta_3)\tnorm{y-x}_x),
\end{array}
\end{equation}
where (\textit{i}) follows from \cite[Theorem 4.1.7]{Nesterov2004} and (\textit{ii}) follows from the Cauchy-Schwarz inequality. 

Let $\gamma := 1 - \delta_3 \in (0, 1]$. 
We consider the function 
\begin{equation*}
\begin{array}{ll}
\underline{\psi}(t) &:= -\delta_2 t + \omega(\gamma t) - \omega(\alpha\gamma t) \vspace{1ex}\\
& = \gamma t - \ln(1 + \gamma t)  - \delta_2 t  - \alpha\gamma t + \ln(1 + \alpha\gamma t) \vspace{1ex}\\
& = \left[\gamma(1-\alpha) - \delta_2\right]t - \ln\left(1 + \frac{\gamma(1-\alpha)t}{1+\alpha\gamma t}\right).
\end{array}
\end{equation*}
The first and second derivatives of $\underline{\psi}$ are given respectively by
\begin{equation*}
\underline{\psi}'(t) = (1-\alpha)\gamma - \delta_2 - \frac{\gamma}{1+\gamma t} + \frac{\alpha\gamma}{1 + \alpha\gamma t}~~~\text{and}~~~\underline{\psi}''(t) = \frac{\gamma^2}{(1 + \gamma t)^2} - \frac{(\alpha\gamma)^2}{(1 + \alpha\gamma t)^2}.
\end{equation*}
Since $\alpha \in [0, 1]$, it is easy to check that $\underline{\psi}''(t) \geq 0$ for all $t \geq 0$. 
Hence, $\underline{\psi}$ is convex.

If $(1-\alpha)\gamma > \delta_2$, then $\underline{\psi}$ attains the minimum at $\underline{t}^{\ast} > 0$ as the positive solution of $\underline{\psi}'(t) = (1-\alpha)\gamma - \delta_2 - \frac{\gamma}{1+\gamma t} + \frac{\alpha\gamma}{1 + \alpha\gamma t}  = 0$.
Solving this equation for a positive solution, we get
\begin{equation*}
\underline{t}^{\ast} := \tfrac{1}{2\alpha\gamma}\left[ \sqrt{(1+\alpha)^2  + \tfrac{4\alpha\delta_2}{(1-\alpha)\gamma - \delta_2}} - (1+\alpha)\right] = \frac{2\delta_2}{\gamma\left[(1-\alpha)\gamma - \delta_2\right]\left[ \sqrt{(1+\alpha)^2  + \tfrac{4\alpha\delta_2}{(1-\alpha)\gamma - \delta_2}} + (1+\alpha)\right]} > 0.
\end{equation*}
Let choose $\alpha := 1 - \frac{2\delta_2}{1-\delta_3} = \frac{1-2\delta_2-\delta_3}{1-\delta_3}$.
To guarantee $\alpha \in [0, 1]$, we impose $2\delta_2 + \delta_3 \in [0, 1)$.
Moreover, $(1-\alpha)\gamma - \delta_2 = \delta_2 \geq 0$.
Substituting $\alpha$ and  $\gamma = 1 - \delta_3$ into $\underline{t}^{*}$, we eventually obtain
\begin{equation*}
\underline{t}^{*} = \frac{1}{(1-\delta_2 - \delta_3) + \sqrt{(1-\delta_2 - \delta_3)^2 + (1-\delta_3)(1-2\delta_2 - \delta_3)}}.
\end{equation*}
As a result, we can directly compute
\begin{equation*}
\frac{\gamma(1-\alpha)\underline{t}^{*}}{1+\alpha\gamma \underline{t}^{*}} = \frac{2\delta_2}{(2- 3\delta_2 - 2\delta_3) + \sqrt{(1-\delta_2 - \delta_3)^2 + (1-\delta_3)(1-2\delta_2 - \delta_3)}}.
\end{equation*}
In this case, we can write the minimum value $\underline{\psi}(\underline{t}^{*})$ of $\underline{\psi}$ explicitly as
\begin{equation*}
\underline{\psi}^{*}(\delta_2,\delta_3) := \underline{\psi}(\underline{t}^{*}) = \frac{\delta_2}{ \underline{c}_{23} + (1-\delta_2-\delta_3)} - \ln\left(1 + \frac{2\delta_2}{ \underline{c}_{23} + (2- 3\delta_2 - 2\delta_3)}\right) \geq 0,
\end{equation*}
where $\underline{c}_{23} := \left[(1-\delta_2 - \delta_3)^2 + (1-\delta_3)(1-2\delta_2 - \delta_3)\right]^{1/2} \geq 0$.
This is exactly third line of \eqref{eq:psi_stars}.

Now, substituting this lower bound $\underline{\psi}^{*}(\delta_2,\delta_3)$ of $\underline{\psi}$ into \eqref{eq:proof_key1a} and noting that $\alpha(1-\delta_3) = 1 - 2\delta_2 - \delta_3$, we obtain
\begin{equation*}
f(y) \geq \hat{f}(x) + \iprods{g(x), y-x} + \omega\left((1-2\delta_2-\delta_3)\tnorm{y - x}_x\right) - \varepsilon + \underline{\psi}^{*}(\delta_2,\delta_3).
\end{equation*}
Clearly, if we define $\tilde{f}(x) := \hat{f}(x) - \varepsilon ~+~ \underline{\psi}^{*}(\delta_2,\delta_3)$ and $\delta_0 := 2\delta_2 + \delta_3 \in [0, 1)$, then the last inequality is exactly the left-hand side inequality of \eqref{eq:global_inexact_oracle}.
 
(b)~To prove the right-hand side inequality of \eqref{eq:global_inexact_oracle}, we first derive
\begin{equation}\label{eq:proof_key1b}
\begin{array}{lcl}
f(y) &\overset{\tiny\textit{(a)}}{\leq} & f(x) + \iprods{\nabla{f}(x), y - x} + \omega_{\ast}(\norms{y-x}_x)  \vspace{1ex}\\
& \overset{\tiny\eqref{eq:inexact_oracle2},\eqref{eq:norm_relation},\textit{(b)}}{\leq} &  \hat{f}(x) +  \iprods{g(x), y - x} + \omega_{\ast}(\beta(1+\delta_3)\tnorm{y-x}_x) \vspace{1.5ex}\\
& & + {~} \varepsilon + \tnorm{g(x) - \nabla{f}(x)}_x^{\ast}\tnorm{y - x}_x + \omega_{\ast}((1+\delta_3)\tnorm{y-x}_x) - \omega_{\ast}(\beta(1+\delta_3)\tnorm{y-x}_x),
\end{array}
\end{equation}
where \textit{(a)} follows from \cite[Theorem 4.1.8]{Nesterov2004} and \textit{(b)} holds due to the Cauchy-Schwarz inequality.

Let $\beta \geq 1$ and $\bar{\gamma} := 1 + \delta_3 \geq 1$.
We consider the following function
\begin{equation*}
\begin{array}{ll}
\bar{\psi}(t) & := \delta_2t +\omega_{\ast}(\bar{\gamma} t) - \omega_{\ast}(\beta\bar{\gamma} t) \vspace{1ex}\\
& = \delta_2 t -\bar{\gamma} t - \ln(1-\bar{\gamma} t)+\beta\bar{\gamma} t + \ln(1-\beta\bar{\gamma} t) \vspace{1ex}\\
& =  \left[(\beta - 1)\bar{\gamma} + \delta_2\right] t  - \ln\left(1 + \frac{\bar{\gamma}(\beta-1)t}{1-\beta\bar{\gamma} t}\right).
\end{array}
\end{equation*}
First, we compute the first and second derivatives of $\bar{\psi}$ respectively as 
\begin{equation*}
\bar{\psi}'(t)=(\beta-1)\bar{\gamma}+\delta_2-\frac{\beta\bar{\gamma}}{1-\beta\bar{\gamma}t}+\frac{\bar{\gamma}}{1-\bar{\gamma}t}~~\text{and}~~\bar{\psi}''(t)=-\frac{(\beta\bar{\gamma})^2}{(1-\beta\bar{\gamma}t)^2}+\frac{\bar{\gamma}^2}{(1-\bar{\gamma}t)^2}.
\end{equation*}
Clearly, $\bar{\psi}''(t) \leq 0$ for all $0 \leq t < \frac{1}{\bar{\gamma}\beta}$.
Hence, $\bar{\psi}$ is concave in $t$.
To find the maximum value of $\bar{\psi}$, we need to solve $\bar{\psi}'(t) = 0$ for $t > 0$, and obtain
\begin{equation*}
\bar{t}^{\ast} = \tfrac{1}{2\beta\bar{\gamma}}\left(1+\beta - \sqrt{(1+\beta)^2 - \tfrac{4\beta\delta_2}{(\beta-1)\bar{\gamma}+\delta_2}}\right) = \frac{2\delta_2}{\bar{\gamma}\left[(\beta-1)\bar{\gamma}+\delta_2\right]\left[\sqrt{(1+\beta)^2 - \tfrac{4\beta\delta_2}{(\beta-1)\bar{\gamma}+\delta_2}} + 1+\beta\right]} > 0.
\end{equation*}
Let us choose $\beta := 1 + \frac{2\delta_2}{1+\delta_3} \geq 1$.
Then we can explicitly compute $\bar{t}^{*}$ as
\begin{equation*}
\bar{t}^{\ast} = \frac{1}{3(1+\delta_2+\delta_3) + \sqrt{3(1+\delta_2+\delta_3)^2 - (1+\delta_3)(1+2\delta_2+\delta_3)}} > 0.
\end{equation*}
To evaluate $\bar{\psi}(\bar{t}^{*})$, we first compute
\begin{equation*}
\frac{\bar{\gamma}(\beta-1)\bar{t}^{*}}{1 - \beta\bar{\gamma}\bar{t}^{*}} = \frac{2\delta_2}{2(1+\delta_2+\delta_3) + \sqrt{3(1+\delta_2+\delta_3)^2 - (1+\delta_3)(1+2\delta_2+\delta_3)}}.
\end{equation*}
Using this expression, we can explicitly compute the maximum value $\bar{\psi}(\bar{t}^{*})$ of $\bar{\psi}$ as
\begin{equation*}
\bar{\psi}^{*}(\delta_2,\delta_3) := \bar{\psi}(\bar{t}^{*}) = \frac{3\delta_2}{3(1+\delta_2+\delta_3) + \bar{c}_{23}} - \ln\left(1 + \frac{2\delta_2}{2(1+\delta_2+\delta_3) + \bar{c}_{23}}\right) \geq 0,
\end{equation*}
where $\bar{c}_{23} :=  \sqrt{3(1+\delta_2+\delta_3)^2 - (1+\delta_3)(1+2\delta_2+\delta_3)} \geq 0$.
Plugging this expression into \eqref{eq:proof_key1b}, and noting that $\beta(1+\delta_3) = 1 + 2\delta_2 + \delta_3 = 1 + \delta_0$, we can show that
\begin{equation*} 
f(y) \leq  \hat{f}(x)  +  \iprods{g(x), y - x} + \omega_{\ast}\left((1+\delta_0)\tnorm{y-x}_x\right)  + \varepsilon + \bar{\psi}^{*}(\delta_2,\delta_3).
\end{equation*}
Finally, by defining $\delta_1 := \max\set{0, 2\varepsilon +  \bar{\psi}^{*}(\delta_2,\delta_3) - \underline{\psi}^{*}(\delta_2,\delta_3)} \geq 0$ and noting that $\tilde{f}(x) = \hat{f}(x) - \varepsilon + \underline{\psi}^{*}$, we obtain
\begin{equation*}
f(y) \leq  \tilde{f}(x)  +  \iprods{g(x), y - x} + \omega_{\ast}\left((1+\delta_0)\tnorm{y-x}_x\right) + \delta_1,
\end{equation*}
which proves the right-hand side inequality of \eqref{eq:global_inexact_oracle}.
\Eproof

\beforesubsec
\subsection{\bf The proof of Lemma \ref{le:est_A}: Inexact oracle of the dual problem}\label{apdx:le:est_A}
\aftersubsec
Since $\varphi$ is self-concordant, by \cite[Theorem 4.1.6]{Nesterov2004} and $\delta(x) := \norm{\tilde{u}^{\ast}(x) - u^{\ast}(x)}_{\tilde{u}^{\ast}(x)}$, we have
\begin{equation*} 
(1-\delta(x))^2 [\nabla^2 \varphi(u^{\ast}(x))]^{-1} \preceq [\nabla^2 \varphi(\tilde{u}^{\ast}(x))]^{-1} \preceq (1- \delta(x))^{-2} [\nabla^2 \varphi(u^{\ast}(x))]^{-1}.
\end{equation*}
Multiplying this estimate by  $A$ and $A^{\top}$ on the left and right, respectively we obtain
\begin{equation*} 
(1-\delta(x))^2 A[\nabla^2 \varphi(u^{\ast}(x))]^{-1}A^{\top} \preceq A[\nabla^2 \varphi(\tilde{u}^{\ast}(x))]^{-1}A^{\top} \preceq (1- \delta(x))^{-2} A[\nabla^2 \varphi(u^{\ast}(x))]^{-1}A^{\top}.
\end{equation*}
Using \eqref{eq:exact_derivs} and \eqref{eq:inexact_oracle3}, this estimate leads to
\begin{equation}\label{eq:lm3_proof1}
(1-\delta(x))^2\nabla^2f(x) \preceq H(x) \preceq (1- \delta(x))^{-2}\nabla^2f(x).
\end{equation}
Since $\delta(x) \leq \delta$ and $\delta_3 :=  \frac{\delta}{1-\delta} \in [0, 1)$, we have $(1-\delta(x))^2 \geq (1-\delta)^2 \geq (1-\delta_3)^2$ and $\frac{1}{(1-\delta(x))^2} \leq \frac{1}{(1-\delta)^2} = (1+\delta_3)^2$.
Using these inequalities in \eqref{eq:lm3_proof1}, we obtain the second bound of \eqref{eq:bounds_exam1}.

Next, by the definition of $g(x)$ and $\nabla{f}(x)$, we can derive that
\begin{equation*}
\begin{array}{ll}
\left[\tnorm{ g(x) - \nabla f(x)}_x^{\ast}\right]^2  & =  (\tilde{u}^{\ast}(x) \!-\! u^{\ast}(x))^{\top}{\!\!}A^{\top}\left(A\nabla^2{\varphi}(\tilde{u}^{\ast}(x))^{-1}{\!\!}A^{\top}\right)^{-1}{\!\!}A(\tilde{u}^{\ast}(x) \!-\! u^{\ast}(x)) \vspace{1ex}\\
&\overset{\tiny{(i)}}{\leq}  (\tilde{u}^{\ast}(x)  - u^{\ast}(x))^{\top}\nabla^2{\varphi}(\tilde{u}^{\ast}(x))(\tilde{u}^{\ast}(x)  - u^{\ast}(x)) \vspace{1ex}\\
&= \norm{\tilde{u}^{\ast}(x)  - u^{\ast}(x)}_{\tilde{u}^{\ast}(x)}^2 \leq \delta^2(x)\leq\delta^2,
\end{array}
\end{equation*}
where we use $A^{\top}(AQ^{-1}A^{\top})^{-1}A \preceq Q$ for $Q = \nabla^2{\varphi}(u^{\ast}(x))\succ 0$ in (\textit{i}) (see \cite{TranDinh2012c} for a detailed proof of this inequality).
This expression implies $\tnorm{ g(x) - \nabla f(x)}_x^{\ast} \leq \delta$, the first estimate of \eqref{eq:bounds_exam1}. 

Now, by the definition of $f$ in \eqref{eq:max_func} and of $\tilde{f}$ in \eqref{eq:inexact_oracle3}, respectively, and the optimality condition $\nabla{\varphi}(u^{\ast}(x)) = A^{\top}x$ in \eqref{eq:max_func}, we have
\begin{equation*}
\begin{array}{lcl}
f(x) - \tilde f(x)  & \overset{\tiny \eqref{eq:max_func}, \eqref{eq:inexact_oracle3}}{=} & \big[\iprods{u^{\ast}(x), A^{\top}x} - \varphi(u^{\ast}(x)) \big] -  \big[\iprods{\tilde{u}^{\ast}(x), A^{\top}x} - \varphi(\tilde{u}^{\ast}(x)) \big]  \vspace{1ex}\\
& = & \varphi(\tilde{u}^{\ast}(x)) - \varphi(u^{\ast}(x)) - \iprods{A^{\top}x,  \tilde{u}^{\ast}(x) - u^{\ast}(x)} \vspace{1ex}\\
& = & \varphi(\tilde{u}^{\ast}(x)) - \varphi(u^{\ast}(x)) - \iprod{\nabla \varphi(u^{\ast}(x)), \tilde{u}^{\ast}(x) - u^{\ast}(x)}.
\end{array}
\end{equation*}
Since $\varphi$ is standard self-concordant,  using \cite[Theorem 4.1.7, 4.1.8]{Nesterov2004} we obtain from the last expression that
\begin{align*}
\omega(\norms{\tilde{u}^{\ast}(x) - u^{\ast}(x)}_{u^{\ast}(x)})  \leq f(x) - \tilde f(x) \leq  \omega_{\ast}(\norms{\tilde{u}^{\ast}(x) - u^{\ast}(x)}_{u^{\ast}(x)}),
\end{align*}
which leads to 
\begin{equation*}
0\leq\omega\left(\tfrac{\delta(x)}{1+\delta(x)}\right)  \leq f(x) - \tilde f(x) \leq  \omega_{\ast}\left(\tfrac{\delta(x)}{1-\delta(x)}\right)\leq \omega_{\ast}\left(\tfrac{\delta}{1-\delta}\right), 
\end{equation*}
provided that $\delta(x)<1$. 
This condition leads to $\vert  f(x) - \tilde{f}(x)\vert \leq  \omega_{\ast}\left(\tfrac{\delta}{1-\delta}\right) =: \varepsilon$.

Using Lemma \ref{le:inexact_oracle_bounds} with $\varepsilon:=\omega_{\ast}\left(\frac{\delta}{1-\delta}\right)$ and $\delta_2 := \delta$, and $\delta_3$ defined above, we conclude that $(\tilde{f}, g, H)$ given by \eqref{eq:inexact_oracle3} is a $(\delta_0,\delta_1)$-global inexact oracle of $f$, where $\delta_0$ and $\delta_2$ are computed from Lemma~\ref{le:inexact_oracle_bounds}. 
Since $2\delta_2+\delta_3<1$ is required in Lemma \ref{le:inexact_oracle_bounds}, by a direct numerical calculation, we obtain $\delta\in [0,0.292]$.

From the optimality condition of \eqref{eq:max_func_arg} we have $\nabla{\varphi}(u^{\ast}(x)) - A^{\top}x = 0$.
Let $r(x) := \nabla{\varphi}(\tilde{u}^{\ast}(x)) - A^{\top}x$.
Then, using the self-concordance of $\varphi$, by \cite[Theorem 4.1.7]{Nesterov2004}, we have
\begin{equation*}
\tfrac{\norms{\tilde{u}^{\ast}(x) - u^{\ast}(x)}_{u^{\ast}(x)}^2}{1 + \norms{\tilde{u}^{\ast}(x) - u^{\ast}(x)}_{u^{\ast}(x)}} \overset{\tiny{(a)}}{\leq} \iprods{\nabla{\varphi}(\tilde{u}^{\ast}(x)) - \nabla{\varphi}(u^{\ast}(x)), \tilde{u}^{\ast}(x) - u^{\ast}(x)} = \iprods{r(x), \tilde{u}^{\ast}(x) - u^{\ast}(x)},
\end{equation*}
where we use \cite[Theorem 4.1.7]{Nesterov2004} in (\textit{a}).
Since $\delta(x) := \norms{\tilde{u}^{\ast}(x) - u^{\ast}(x)}_{\tilde{u}^{\ast}(x)}$, by the Cauchy-Schwarz inequality, we can show that $\frac{\delta(x)^2}{1 + \delta(x)} \leq \norms{r(x)}^{\ast}_{\tilde{u}^{\ast}(x)}\delta(x)$, which leads to 
$\frac{\delta(x)}{1+\delta(x)} \leq \Vert r(x)\Vert^{*}_{\tilde{u}^{*}(x)}$.

Finally, we assume that $\Vert r(x)\Vert^{*}_{\tilde{u}^{*}(x)} \leq \frac{\delta}{1+\delta}$ for some $\delta > 0$ as stated in Lemma \ref{le:est_A}.
Using this condition and the last inequality  $\frac{\delta(x)}{1+\delta(x)} \leq \Vert r(x)\Vert^{*}_{\tilde{u}^{*}(x)}$ we have $\frac{\delta(x)}{1+\delta(x)} \leq \frac{\delta}{1+\delta}$, which implies that $\delta(x) \leq \delta$.
\Eproof

\beforesubsec
\subsection{\bf The proof of Lemma \ref{le:key_est_for_prox_nt_scheme_ine}: Key estimate for local convergence analysis}\label{apdx:le:key_est_for_prox_nt_scheme_ine}
\aftersubsec
First, recall that $\nu^k\in g(x^k)+H(x^k)(z^k - x^k)+\partial R(z^k)$ from \eqref{eq:inexact_subp}.
Using the definition of $\Pc_x$ from \eqref{eq:scaled_prox_oper}, this expression leads to
\begin{equation*}
H(x^k)x^k + \nu^k - g(x^k) \in \partial R(z^k)+H(x^k)z^k~~~~{\iff}~~~~z^k = \mathcal{P}_{x^k}(x^k + [H(x^k)]^{-1}(\nu^k-g(x^k))).
\end{equation*}
Shifting the index from $k$ to $k+1$, the last expression leads to
\begin{equation}\label{opt:quad3}
z^{k+1} = \mathcal{P}_{x^{k+1}}(x^{k+1} + [H(x^{k+1})]^{-1}(\nu^{k+1} - g(x^{k+1}))).
\end{equation}
Next, if we denote by $r_{x^k}(z^k):=g(x^k)+H(x^k)(z^k-x^k)$, then again from \eqref{eq:inexact_subp} and \eqref{eq:scaled_prox_oper}, we can rewrite
\begin{equation}\label{opt:quad4}
{\!\!\!\!}\begin{array}{rrl}
\nu^k-r_{x^k}(z^k) \in \partial R(z^k)  &~~~ \iff ~~~& z^k+[H(x^{k+1})]^{-1}(\nu^k-r_{x^k}(z^k)) \in z^k+[H(x^{k+1})]^{-1}\partial R(z^k)\vspace{1.25ex}\\
&~~~ \iff ~~~& z^k = \mathcal{P}_{x^{k+1}}(z^k+[H(x^{k+1})]^{-1}(\nu^k-r_{x^k}(z^k))).
\end{array}{\!\!\!\!}
\end{equation}
Denote $H_k := H(x^k)$, $f_k':=\nabla f(x^k)$, and $g_k := g(x^k)$ for simplicity. 
By the triangle inequality, we have
\begin{equation}\label{triangle_subinexact}
\lambda_{k+1}=\tnorm{z^{k+1} - x^{k+1}}_{x^{k+1}} \leq  \tnorm{x^{k+1} - z^{k}}_{x^{k+1}} + \tnorm{z^{k+1} - z^{k}}_{x^{k+1}}.
\end{equation}
To upper bound $\lambda_{k+1}$, we upper bound each term of \eqref{triangle_subinexact} as follows.

(a)~For the first term $\tnorm{x^{k+1}-z^k}_{x^{k+1}}$ of \eqref{triangle_subinexact}, since $f$ is standard self-concordant, by \eqref{eq:norm_relation} and \cite[Theorem 4.1.5]{Nesterov2004}, we have
\begin{equation*} 
\begin{array}{lcl}
\tnorm{x^{k+1}-z^k}_{x^{k+1}} &\overset{\tiny\eqref{eq:norm_relation}}{\leq} & (1+\delta_3)\norms{x^{k+1}-z^k}_{x^{k+1}}  \vspace{1ex}\\
&\overset{\tiny{\text{\cite[Theorem 4.1.5]{Nesterov2004}}}}{\leq}& \frac{1}{1-\norms{x^{k+1}-x^k}_{x^k}}\cdot (1+\delta_3)\norms{x^{k+1}-z^k}_{x^{k}} \vspace{1ex}\\
&\overset{\tiny{\eqref{eq:norm_relation}}}{\leq}&  \frac{1}{1-\tfrac{1}{1-\delta_3}\tnorm{x^{k+1}-x^k}_{x^k}} \cdot \frac{(1+\delta_3)\tnorm{x^{k+1}-z^k}_{x^{k}}}{1-\delta_3}   \vspace{1ex}\\
&= & \frac{(1+\delta_3)\tnorm{x^{k+1}-z^k}_{x^{k}}}{1-\delta_3 - \tnorm{x^{k+1}-x^k}_{x^k}}.
\end{array}
\end{equation*}
Since $\alpha_k \in [0, 1]$, $\lambda_k := \tnorm{d^k}_{x^k}$, and $x^{k+1} := x^k + \alpha_k(z^k - x^k) = x^k + \alpha_kd^k$ due to \eqref{eq:prox_nt_scheme}, we have
\begin{equation}\label{eq:norm_lbd}
\left\{\begin{array}{ll}
\tnorm{x^{k+1}-z^k}_{x^k} &= \tnorm{(1-\alpha_k)(z^k - x^k) }_{x^k} = (1-\alpha_k)\tnorm{d^k}_{x^k} =  (1-\alpha_k)\lambda_k, \vspace{1.2ex}\\
\tnorm{x^{k+1} - x^k}_{x^k} &= \tnorm{x^k + \alpha_k(z^k - x^k) - x^k}_{x^k} = \alpha_k\tnorm{z^k - x^k}_{x^k} = \alpha_k\tnorm{d^k}_{x^k} = \alpha_k\lambda_k.
\end{array}\right.
\end{equation}
Substituting \eqref{eq:norm_lbd} into the last estimate, we obtain
\begin{equation}\label{inexact_subest0}
\tnorm{x^{k+1}-z^k}_{x^{k+1}} \leq \frac{(1+\delta_3)(1-\alpha_k)\lambda_k}{1-\delta_3 - \alpha_k\lambda_k}.
\end{equation}

(b)~For the second term $\tnorm{z^{k+1} - z^{k}}_{x^{k+1}}$ of \eqref{triangle_subinexact}, using \eqref{opt:quad3}, \eqref{opt:quad4}, the triangle inequality in (\textit{i}), and the nonexpansiveness of the scaled proximal operator $\mathcal{P}_{x}$ from \eqref{eq:nonexpansiveness}, we can show that
\rvnn{
\begin{equation}\label{inexact_subest1}
{\!\!\!\!\!\!\!}\begin{array}{lcl}
\tnorm{z^{k+1} - z^{k}}_{x^{k+1}} {\!\!\!\!}& \overset{\tiny\eqref{opt:quad3},\eqref{opt:quad4}}{=} {\!\!\!}& \tnorm{\mathcal{P}_{x^{k+1}}(x^{k+1} + H_{k+1}^{-1}(\nu^{k+1}-g_{k+1}))-\mathcal{P}_{x^{k+1}}(z^k+H_{k+1}^{-1}(\nu^k-r_{x^k}(z^k)))}_{x^{k+1}}\vspace{1ex}\\
& \overset{\tiny\eqref{eq:nonexpansiveness}}{\leq} & \tnorm{(x^{k+1} + H_{k+1}^{-1}(\nu^{k+1}-g_{k+1})-(z^k+H_{k+1}^{-1}(\nu^k - r_{x^k}(z^k)))}_{x^{k+1}}\vspace{1.25ex}\\
& \overset{\tiny\eqref{eq:weighted_norm}}{=} & \tnorm{(H_{k+1}-H_k)(x^{k+1} - z^k)-(g_{k+1}-g_k-H_k(x^{k+1}-x^k)) + (\nu^{k+1}-\nu^k)}_{x^{k+1}}^{\ast} \vspace{1ex}\\
& \overset{\tiny(i)}{\leq}  & \underbrace{ \tnorm{(H_{k+1}-H_k)(x^{k+1} - z^k)-(g_{k+1}-g_k-H_k(x^{k+1}-x^k))}_{x^{k+1}}^{\ast} }_{[\Tc_1]} \vspace{1ex}\\
& & + {~} \underbrace{ \tnorm{\nu^{k+1}-\nu^k}_{x^{k+1}}^{\ast} }_{[\Tc_2]}.
\end{array}{\!\!\!\!}
\end{equation}}
To further estimate the last term $[\Tc_2]$ of \eqref{inexact_subest1}, we have
\begin{equation*}
\begin{array}{ll}
\tnorm{\nu^{k}}_{x^{k+1}}^{\ast} &\overset{\tiny{\eqref{eq:norm_relation}}}{\leq} \frac{1}{1-\delta_3}\norms{\nu^k}_{x^{k+1}}^{\ast} \overset{\text{\tiny{\cite[Theorem 4.1.6]{Nesterov2004}}}}{\leq} \frac{\norms{\nu^k}_{x^{k}}^{\ast} }{(1-\delta_3)(1-\norms{x^{k+1}-x^k}_{x^k})}\vspace{1ex}\\
&\overset{\tiny{\eqref{eq:norm_relation}}}{\leq} \frac{(1+\delta_3)\tnorm{\nu^k}_{x^k}^{\ast}}{(1 - \delta_3)\left(1 - \frac{1}{1-\delta_3}\tnorm{x^{k+1}-x^k}_{x^k}\right) } \vspace{1ex}\\
&\overset{\tiny\eqref{eq:norm_lbd}}{=} \frac{(1 + \delta_3)}{(1-\delta_3 - \alpha_k\lambda_k)}\tnorm{\nu^k}_{x^k}^{\ast}.
\end{array}
\end{equation*}
Utilizing this estimate and the triangle inequality, we can estimate the  term $[\Tc_2]$ of \eqref{inexact_subest1} as
\begin{equation}\label{inexact_subest2}
{\!\!\!\!}\begin{array}{lll}
[\Tc_2] &:= \tnorm{\nu^{k+1} \!-\! \nu^{k}}_{x^{k+1}}^{\ast}  &\leq  \tnorm{\nu^{k+1}}_{x^{k+1}}^{\ast} + \tnorm{\nu^{k}}_{x^{k+1}}^{\ast} \vspace{1ex}\\
& \leq \tnorm{\nu^{k+1}}_{x^{k+1}}^{\ast} + \frac{(1 + \delta_3)}{(1-\delta_3 - \alpha_k\lambda_k)}\tnorm{\nu^k}_{x^k}^{\ast} &\overset{\eqref{eq:inexact_subp}}{\leq} \delta_4\lambda_{k+1}+\frac{(1 + \delta_3)\delta_4\lambda_k}{(1-\delta_3 - \alpha_k\lambda_k)}.
\end{array}{\!\!\!\!}
\end{equation}
Now, using triangle inequality, we can split the term $[\Tc_1]$ of  \eqref{inexact_subest1} as
\begin{equation}\label{inexact_subest3}
{\!\!\!\!}\begin{array}{ll}
[\Tc_1] &:= \tnorm{(H_{k+1}-H_k)(x^{k+1} - z^k)-(g_{k+1}-g_k-H_k(x^{k+1}-x^k))}_{x^{k+1}}^{\ast}\vspace{1.5ex}\\
&\leq  \tnorm{H_{k+1}(x^{k+1}-z^k)}_{x^{k+1}}^{\ast} + \tnorm{H_k(x^{k+1} \!-\! z^k)}_{x^{k+1}}^{\ast} + \tnorm{f_{k+1}' - g_{k+1}}^{\ast}_{x^{k+1}} + \tnorm{{f_k'-g_k}}^{\ast}_{x^{k+1}} \vspace{1.5ex}\\
&+~ \tnorm{f_{k+1}'-f_k'-\nabla^2 f(x^k)(x^{k+1}-x^k)}^{\ast}_{x^{k+1}} + \tnorm{(H_k-\nabla^2 f(x^k))(x^{k+1}-x^k)}^{\ast}_{x^{k+1}}.
\end{array}{\!\!\!\!}
\end{equation}
In addition, using the left-hand side inequality in the first line of \eqref{eq:key_bound3} with $x := x^k$ and $y := x^{k+1}$, we have
\begin{equation}\label{eq:key_bound3_n}
\tnorm{\cdot}_{x^{k+1}}^{*} \overset{\tiny\eqref{eq:key_bound3}}{\leq} \frac{(1+\delta_3)}{(1-\delta_3 - \tnorm{x^{k+1}-x^k}_{x^k})}\tnorm{\cdot}_{x^k}^{*} \overset{\tiny\eqref{eq:norm_lbd}}{=} \frac{(1+\delta_3)}{(1-\delta_3 - \alpha_k\lambda_k)}\tnorm{\cdot}_{x^k}^{*}.
\end{equation}
To estimate each term of \eqref{inexact_subest3}, we first note that
\begin{equation}\label{eq:s3_term1}
 \tnorm{H_{k+1}(x^{k+1}-z^k)}_{x^{k+1}}^{\ast} \overset{\tiny\eqref{eq:weighted_norm}}{=}  \tnorm{x^{k+1}-z^k}_{x^{k+1}} \overset{\tiny\eqref{inexact_subest0}}{\leq} \frac{(1+\delta_3)(1-\alpha_k)\lambda_k}{1-\delta_3 - \alpha_k\lambda_k}.
\end{equation}
Second, using \eqref{eq:key_bound3_n} and $\tnorm{H_kd^k}_{x^k}^{\ast} = \tnorm{d^k}_{x^k} = \lambda_k$, we can show that
\begin{equation}\label{eq:s3_term2}
\begin{array}{ll}
\tnorm{H_k(x^{k+1} - z^k)}_{x^{k+1}}^{\ast} &= (1-\alpha_k)\tnorm{H_kd^k}_{x^{k+1}}^{\ast} \vspace{1ex}\\
& \overset{\tiny\eqref{eq:key_bound3_n}}{\leq} \frac{(1+\delta_3)(1-\alpha_k)}{(1- \delta_3 - \tnorm{x^{k+1} - x^k}_{x^k})}\tnorm{H_kd^k}_{x^k}^{\ast} \vspace{1ex}\\
&\overset{\tiny\eqref{eq:s3_term1}}{\leq}  \frac{(1+\delta_3)(1-\alpha_k)\lambda_k}{(1- \delta_3 - \alpha_k\lambda_k)}.
\end{array}
\end{equation}
Third, by \eqref{eq:local_inexact_oracle} and \eqref{eq:key_bound3_n}, we have
\begin{equation}\label{eq:s3_term3}
\left\{\begin{array}{ll}
\tnorm{f_{k+1}' - g_{k+1}}^{\ast}_{x^{k+1}} &\overset{\tiny\eqref{eq:local_inexact_oracle}}{\leq} \delta_2\vspace{1ex}\\
\tnorm{{f_k'-g_k}}^{\ast}_{x^{k+1}} &\overset{\tiny\eqref{eq:key_bound3_n}}{\leq} \frac{(1+\delta_3)}{(1- \delta_3 - \alpha_k\lambda_k)}\tnorm{{f_k'-g_k}}^{\ast}_{x^{k}}
\overset{\tiny\eqref{eq:local_inexact_oracle}}{\leq} \frac{(1+\delta_3)\delta_2}{(1- \delta_3 - \alpha_k\lambda_k)}.
\end{array}\right.
\end{equation}
Fourth, utilizing \eqref{eq:key_bound3_n} and \cite[Theorem 4.1.14]{Nesterov2004}, we can show that
\begin{equation}\label{eq:s3_term4}
{\!\!\!\!}\begin{array}{lcl}
\tnorm{f_{k+1}'-f_k'-\nabla^2 f(x^k)(x^{k+1}-x^k)}^{\ast}_{x^{k+1}} 
&\overset{\tiny\eqref{eq:norm_relation}}{\leq} &  \frac{1}{(1-\delta_3)}\norms{f_{k+1}'-f_k'-\nabla^2 f(x^k)(x^{k+1}-x^k)}^{\ast}_{x^{k+1}} \vspace{1ex}\\
&{\!\!\!\!\!\!\!\!\!\!\!\!\!}\overset{\tiny\text{\cite[Theorem 4.1.14]{Nesterov2004}}}{\leq} {\!\!\!\!}&  \frac{1}{(1-\delta_3)}\left(\frac{\norms{x^{k+1}-x^k}_{x^k}}{1-\norms{x^{k+1}-x^k}_{x_k}}\right)^2 \vspace{1ex}\\
&\overset{\tiny\eqref{eq:norm_relation}}{\leq} &   \frac{1}{(1-\delta_3)}\left(\frac{\tnorm{x^{k+1}-x^k}_{x^k}}{1-\delta_3 - \tnorm{x^{k+1}-x^k}_{x_k}}\right)^2 \vspace{1ex}\\
&\overset{\tiny\eqref{eq:norm_lbd}}{=}  &  \frac{1}{(1-\delta_3)}\left(\frac{\alpha_k\lambda_k}{1-\delta_3 - \alpha_k\lambda_k}\right)^2.
\end{array}
\end{equation}
Fifth, employing the second inequality of \eqref{eq:key_bound3} with $x := x^k$, $y := x^{k+1}$, and $v := x^{k+1} - x^k$, we can show that
\begin{equation}\label{eq:s3_term5}
\begin{array}{ll}
\tnorm{(H_k-\nabla^2 f(x^k))(x^{k+1}-x^k)}^{\ast}_{x^{k+1}} 
&\overset{\tiny\eqref{eq:key_bound3}}{\leq} \frac{(2\delta_3 + \delta_3^2)}{(1-\delta_3)(1 -\delta_3 - \tnorm{x^{k+1} - x^k}_{x^k})} \tnorm{x^{k+1} - x^k}_{x^k} \vspace{1ex}\\
&\overset{\tiny\eqref{eq:norm_lbd}}{=} \frac{(2 + \delta_3)\delta_3\alpha_k\lambda_k}{(1-\delta_3)(1 -\delta_3 -  \alpha_k\lambda_k)}.
\end{array}
\end{equation}
Finally, substituting \eqref{eq:s3_term1}, \eqref{eq:s3_term2}, \eqref{eq:s3_term3}, \eqref{eq:s3_term4}, and \eqref{eq:s3_term5} into \eqref{inexact_subest3}, we can upper bound $[\Tc_1]$ as
\begin{equation*}
\begin{array}{ll}
[\Tc_1] &\leq  \frac{(1+\delta_3)(1-\alpha_k)}{1-\delta_3 - \alpha_k\lambda_k}\lambda_k 
+ \frac{(1 + \delta_3)(1-\alpha_k)}{(1-\delta_3 - \alpha_k\lambda_k)}\lambda_k + \delta_2 
+ \frac{(1 + \delta_3)\delta_2}{(1-\delta_3 - \alpha_k\lambda_k)}\vspace{1.5ex}\\
&+~ \frac{1}{1-\delta_3}\cdot\left(\frac{\alpha_k\lambda_k}{1-\delta_3 - \alpha_k\lambda_k}\right)^2 
+ \frac{(2+\delta_3)\delta_3}{1-\delta_3}\cdot \frac{\alpha_k\lambda_k}{1-\delta_3 - \alpha_k\lambda_k}.
\end{array}
\end{equation*}
Plugging this upper bound of $[\Tc_1]$ and the upper bound of $[\Tc_2]$ from \eqref{inexact_subest2} into \eqref{inexact_subest1}, we obtain
\begin{equation*}
\begin{array}{ll}
\tnorm{z^{k+1} - z^{k}}_{x^{k+1}} &\leq  \frac{2(1+\delta_3)(1-\alpha_k)}{1-\delta_3 - \alpha_k\lambda_k}\lambda_k 
+ \delta_2 
+ \frac{(1 + \delta_3)\delta_2}{(1-\delta_3 - \alpha_k\lambda_k)}  + \frac{1}{1-\delta_3} \cdot \left(\frac{\alpha_k\lambda_k}{1-\delta_3 - \alpha_k\lambda_k}\right)^2 
+ \frac{(2+\delta_3)\delta_3}{1-\delta_3} \cdot \frac{\alpha_k\lambda_k}{1-\delta_3 - \alpha_k\lambda_k} \vspace{1ex}\\
&+ {~}\delta_4\lambda_{k+1} 
+ \frac{(1 + \delta_3)\delta_4\lambda_k}{(1-\delta_3 - \alpha_k\lambda_k)}.
\end{array}
\end{equation*}
Substituting this estimate and \eqref{inexact_subest0} back into \eqref{triangle_subinexact} we get
\begin{equation*}
\begin{array}{ll}
\lambda_{k+1} &\leq \delta_4\lambda_{k+1} + \delta_2 
+ \frac{3(1+\delta_3)(1-\alpha_k)\lambda_k}{1-\delta_3 - \alpha_k\lambda_k} 
+ \frac{(1 + \delta_3)\delta_2}{(1-\delta_3 - \alpha_k\lambda_k)} 
+ \frac{1}{1-\delta_3} \cdot \left(\frac{\alpha_k\lambda_k}{1-\delta_3 - \alpha_k\lambda_k}\right)^2  \vspace{1.5ex}\\
&+{~} \frac{(2+\delta_3)\delta_3}{1-\delta_3} \cdot \frac{\alpha_k\lambda_k}{1-\delta_3 - \alpha_k\lambda_k} 
+ \frac{(1 + \delta_3)\delta_4\lambda_k}{(1-\delta_3 - \alpha_k\lambda_k)}.
\end{array}
\end{equation*}
Since $0 < 1 - \delta_4 < 1$, rearranging this estimate, we obtain \eqref{eq:key_local_estimate_com_ine}.
\Eproof

\beforesubsec
\subsection{\bf Detailed proofs of the missing technical results in the main text}\label{app:detail_proofs}
\aftersubsec
In this subsection, we provide more details of some missing proofs in the main text.

\vspace{1ex}
\noindent\textbf{(a)~Technical details in the proof of Theorems~\ref{th:local_covergence1} and \ref{th:local_covergence2}:}
Let  us denote the right-hand side of \eqref{eq:key_local_estimate_com_ine} by
\begin{equation*} 
\begin{array}{ll}
H(\alpha_k, \lambda_k, \theta) &
:= \frac{\delta_2}{1-\delta_4}  + \frac{(1+\delta_3)\left[\delta_2 
+ \left(\delta_4 + 3(1-\alpha_k)\right)\lambda_k\right]}{(1-\delta_4)\left( 1 - \delta_3 - \alpha_k\lambda_k\right)}  
+ \frac{\alpha_k(2+\delta_3)\delta_3\lambda_k}{(1-\delta_4)(1-\delta_3)\left( 1 - \delta_3 - \alpha_k\lambda_k\right)} \vspace{1ex}\\
& + {~} \frac{\alpha_k^2\lambda_k^2}{(1-\delta_4)(1-\delta_3)\left( 1 - \delta_3 - \alpha_k\lambda_k\right)^2},
\end{array}
\end{equation*}
where $\lambda_k, \delta_2 \geq 0$, $\alpha_k \in [0, 1]$, $\delta_3, \delta_4 \in [0, 1)$, $\alpha_k\lambda_k + \delta_3 < 1$, and $\theta := (\delta_2, \delta_3, \delta_4)$.

\noindent If $\alpha_k = 1$, then $H(\cdot)$ reduces to
\begin{equation}\label{eq:H_func1}
\begin{array}{ll}
H_1(\lambda_k, \theta) &:= \frac{\delta_2}{1-\delta_4}  
+ \frac{(1+\delta_3)\left(\delta_2 + \delta_4\lambda_k\right)}{(1-\delta_4)\left( 1 - \delta_3 - \lambda_k\right)}  
+ \frac{(2+\delta_3)\delta_3\lambda_k}{(1-\delta_4)(1-\delta_3)\left( 1 - \delta_3 - \lambda_k\right)}  
+ \frac{\lambda_k^2}{(1-\delta_4)(1-\delta_3)\left( 1 - \delta_3 - \lambda_k\right)^2}.
\end{array}
\end{equation}
If $\alpha_k = \frac{1-\delta_4}{(1+\delta)(1+\delta + (1-\delta)\lambda_k)} \in [0, 1]$, then $H(\cdot)$ can be rewritten as
\begin{equation}\label{eq:H_func2}
\begin{array}{ll}
H_2(\alpha_k,\lambda_k, \delta, \theta) &:= \frac{\delta_2}{1-\delta_4}  
+ \frac{(1+\delta_3)\left(\delta_2 + \delta_4\lambda_k\right)}{(1-\delta_4)\left( 1 - \delta_3 - \alpha_k\lambda_k\right)}  
+  \frac{3(1+\delta_3)\lambda_k}{\left( 1 - \delta_3 - \alpha_k\lambda_k\right)}\cdot \frac{2\delta + \delta^2 + \delta_4 + (1-\delta^2)\lambda_k}{(1+\delta)(1+\delta + (1-\delta)\lambda_k)} \vspace{1ex}\\
& + \frac{\alpha_k(2+\delta_3)\delta_3\lambda_k}{(1-\delta_4)(1-\delta_3)\left( 1 - \delta_3 - \alpha_k\lambda_k\right)}  
+ \frac{\alpha_k^2\lambda_k^2}{(1-\delta_4)(1-\delta_3)\left( 1 - \delta_3 - \alpha_k\lambda_k\right)^2}.
\end{array}
\end{equation}
The following lemma is used to prove Theorems~\ref{th:local_covergence1} and \ref{th:local_covergence2} in the main text.

\begin{lemma}\label{le:apdx-monotone}
The function $H_1(\cdot)$ defined by \eqref{eq:H_func1} is monotonically increasing w.r.t. each variable $\lambda_k \geq 0$, $\delta_2 \geq 0$, $\delta_3 \in [0, 1)$, and $\delta_4\in [0, 1)$ such that $\lambda_k + \delta_3 < 1$.

Similarly, for given $\lambda_k > 0$ and $\delta \in [0, 1)$, the function $H_2(\cdot)$ defined by \eqref{eq:H_func2} is monotonically increasing w.r.t. each variable $\alpha_k \in [0, 1]$, $\delta_2 \geq 0$, $\delta_3 \in [0, 1)$, and $\delta_4\in [0, 1)$ such that $\alpha_k\lambda_k + \delta_3 < 1$.
Moreover, if $0 \leq \delta_3, \delta_4 \leq \delta$, then we can upper bound $H_2$ as $H_2(\alpha_k,\lambda_k, \delta, \theta) \leq \widehat{H}_2(\lambda_k, \delta, \delta_2)$, where
\begin{equation*}
\begin{array}{ll}
\widehat{H}_2(\lambda_k, \delta, \delta_2) &:= \frac{\delta_2}{1-\delta}  
+ \frac{(1+\delta)\left(\delta_2 + \delta\lambda_k\right)}{(1-\delta)\left( 1 - \delta - \lambda_k\right)}  
+ \frac{3\lambda_k\left[3\delta + \delta^2 + (1-\delta^2)\lambda_k\right]}{\left( 1 - \delta - \lambda_k\right)\left[1+\delta + (1-\delta)\lambda_k\right]}  + \frac{(2+\delta)\delta\lambda_k}{(1-\delta)^2\left( 1 - \delta - \lambda_k\right)}  
+ \frac{\lambda_k^2}{(1-\delta)^2\left( 1 - \delta - \lambda_k\right)^2}.
\end{array}
\end{equation*}
The function $\widehat{H}_2(\cdot)$ is also monotonically increasing w.r.t. each variable $\delta_2$ and $\lambda_k$.
\end{lemma}

\begin{proof}
We first consider $H_1$ defined by \eqref{eq:H_func1}.
For $\lambda_k \geq 0$, $\delta_2 \geq 0$, $\delta_3 \in [0, 1)$, and $\delta_4\in [0, 1)$ such that $\lambda_k + \delta_3 < 1$, the term $1$ is $\frac{\delta_2}{1-\delta_4}$, which is monotonically increasing w.r.t. $\delta_2$ and $\delta_4$.
The term   $2$  is monotonically increasing w.r.t. $\lambda_k$, $\delta_2$, $\delta_3$, and $\delta_4$.
The terms  $3$ and $4$ are monotonically increasing w.r.t. $\lambda_k$, $\delta_3$, and $\delta_4$.
Consequently, $H_1(\cdot)$ is monotonically increasing w.r.t. $\lambda_k$, $\delta_2$, $\delta_3$, and $\delta_4$.

For fixed $\lambda_k > 0$ and $\delta \in [0, 1)$, we consider $H_2(\cdot)$ defined by \eqref{eq:H_func2}.
Clearly, for $\alpha_k \in [0, 1]$, $\delta_2 \geq 0$, $\delta_3 \in [0, 1)$, and $\delta_4\in [0, 1)$ such that $\alpha_k\lambda_k + \delta_3 < 1$, the term 1 is monotonically increasing w.r.t. $\delta_2$ and $\delta_4$.
The term   $2$  is monotonically increasing w.r.t. $\lambda_k$, $\delta_2$, $\delta_3$, and $\delta_4$.
The terms $3$ and $4$ are monotonically increasing w.r.t. $\delta_3$, and $\delta_4$.
Consequently, $H_2(\cdot)$ is monotonically increasing w.r.t.  $\delta_2$, $\delta_3$, and $\delta_4$.
Using the upper bound $\delta$ of $\delta_3$ and $\delta_4$ into $H_2$, we can easily get $H_2(\alpha_k,\lambda_k, \delta, \theta) \leq \widehat{H}_2(\lambda_k, \delta, \delta_2)$. 
The monotonic increase of $\widehat{H}_2$ w.r.t. $\delta_2$ and $\lambda_k$ can be easily checked directly by verifying each term separately.
\Eproof
\end{proof}

\noindent\textbf{(b)~The detailed proof for Example 1(c) in Subsection~\ref{ex:global}:}
We provide here the detailed proof of the estimate \eqref{eq:exam1c_bound} in Example 1(c) of Subsection~\ref{ex:global}.

Since $f_1(x) = -\ln(x)$, $f_2(x) = \max\{\delta_1x, \delta_1\}$, and $f(x) = f_1(x) + f_2(x)$, we have $\dom{f} = \{x\in\R \mid x > 0\}$.
Moreover, since $\nabla^2{f_1}(x) = H(x) = \frac{1}{x^2}$, the condition $\vert\!\Vert y - x\vert\!\Vert_{x} < 1$ (here we use $\delta_0 = 0$) leads to $\frac{(y-x)^2}{x^2} < 1$, which is equivalent to $-x < y-x < x$, or equivalently, $0 < y < 2x$. Since $y > 0$, the condition $\Vert y - x\Vert_{x} < 1$ is equivalent to $0 < y < 2x$.
In this case, we have 
\begin{equation*}
g_2(x) = \begin{cases}
0 &\text{if}~x < 1\\
\delta_1&\text{if}~x > 1\\
[0, \delta_1] &\text{if}~ x = 1.
\end{cases}
\end{equation*}
Using this expression, one can show that
\begin{equation*}
\begin{array}{ll}
f_2(y) - f_2(x) -  \langle g_2(x), y-x\rangle &= \max\{\delta_1y, \delta_1\} - \max\{\delta_1x, \delta_1\} - \delta_1(y-x) \vspace{1ex}\\
& \leq \begin{cases}
\delta_1 - \delta_1 &\text{if}~0 < x < 1~\text{and}~~0 < y < 1\\
2\delta_1 - \delta_1 &\text{if}~0 < x < 1~\text{and}~~1 < y \leq 2x < 2\\
\delta_1 - \delta_1x - \delta_1(y-x) &\text{if}~x > 1~\text{and}~~0 < y \leq 1\\
\delta_1y - \delta_1x - \delta_1(y-x) &\text{if}~x > 1~\text{and}~~y > 1\\
\delta_1 - \delta_1 - \xi(y-1) &\text{if}~x = 1~\text{and}~~0 < y \leq 1,~~\xi\in [0, \delta_1]\\
2\delta_1 - \delta_1 - \xi(y-1) &\text{if}~x = 1~\text{and}~~1 < y \leq 2x = 2.
\end{cases}\\
&\leq \begin{cases}
0 &\text{if}~0 < x < 1~\text{and}~~0 < y < 1\\
\delta_1 &\text{if}~0 < x < 1~\text{and}~~1 < y \leq 2x < 2\\
\delta_1 &\text{if}~x > 1~\text{and}~~0 < y \leq 1\\
0 &\text{if}~x > 1~\text{and}~~y > 1\\
\delta_1 &\text{if}~x = 1~\text{and}~~0 < y \leq 1,~~\xi\in [0, \delta_1]\\
\delta_1 &\text{if}~x = 1~\text{and}~~1 < y \leq 2x = 2.
\end{cases}
\end{array}
\end{equation*}
In summary, we get $f_2(y) - f_2(x) -  \langle g_2(x), y-x\rangle \leq \delta_1$, which is exactly \eqref{eq:exam1c_bound}.
\Eproof

\beforesubsec
\subsection{\bf Implementation details: Approximate proximal Newton directions}\label{app:sub}
\aftersubsec
When solving for $z^k$ in \eqref{eq:prox_nt_scheme}, we use FISTA \cite{Beck2009}. 
At the $j^{\mathrm{th}}$ iteration of the inner loop, $d^j$ is computed as
\begin{equation*}
d^j := \prox_{\alpha R}\left(x^k+w-\alpha(g(x^k)+H(x^k)w)\right)-x^k,
\end{equation*}
where $w^j :=d^{j-1}+\frac{t_{j-1}-1}{t_{j}}(d^{j-1}-d^{j-2})$. 
By the definition of $\prox_{\alpha R}$, the following relation holds:
\begin{equation*}
\tfrac{1}{\alpha}(w^j - d^j)\in g(x^k)+H(x^k)w^j +\partial R(x^k+d^j),
\end{equation*}
which guarantees that the vector $\nu^k :=\frac{w^j - d^j}{\alpha}+H(x^k)(d^j - w^j)=\left(\frac{\Id_p}{\alpha}-H(x^k)\right)(w^j - d^j)$ satisfies the condition $\nu^k \in g(x^k)+H(x^k)(d^j)+\partial R(x^k+d^j)$. 
In our implementation, this $\nu^k$ was used in \eqref{eq:inexact_subp} to determine whether to accept this $d^k := d^j$ as an inexact proximal Newton direction at the iteration $k$ in \eqref{eq:prox_nt_scheme}.
\Eproof
%
\beforesec
\bibliographystyle{plain}

\end{document}